\documentclass[english]{smfart}
\usepackage{baskervald}
\usepackage[utf8x]{inputenc} %
\usepackage[english]{babel}
\usepackage{amsthm} 
\usepackage{amsmath} 
\usepackage{amsfonts}
\usepackage{amssymb}
\usepackage{mathdots}
\usepackage[]{todonotes}
\usepackage{mathtools}
\usepackage{enumerate}
\usepackage{graphicx}

\newcommand\ZZ{\mathbb{Z}} 
\newcommand\QQ{\mathbb{Q}} 
\newcommand\RR{\mathbb{R}} 
\newcommand\CC{\mathbb{C}} 
 
\newcommand\FF{\mathbb{F}} 
\newcommand\FP{\mathbb{F}_p} 
 
\newcommand{\Qp}{\QQ_p}
\newcommand\eps{\varepsilon} 
\newcommand\fleche{\longrightarrow}
\newcommand{\cris}{\mathrm{cris}}
\newcommand{\dR}{\mathrm{dR}}
\newcommand{\modag}{\mathrm{mod},\mathrm{ag}}

\DeclareMathOperator{\Ker}{Ker}

\DeclareMathOperator{\car}{char}

\DeclareMathOperator{\tr}{tr}
\DeclareMathOperator{\Tr}{Tr}

\DeclareMathOperator{\pr}{pr}
\DeclareMathOperator{\Hom}{Hom}
\DeclareMathOperator{\Id}{Id}
\DeclareMathOperator{\Spec}{Spec}
\DeclareMathOperator{\Spf}{Spf}

\DeclareMathOperator{\id}{id}
\DeclareMathOperator{\Gal}{Gal}
\DeclareMathOperator{\Fil}{Fil}
\DeclareMathOperator{\im}{Im}

\DeclareMathOperator{\Frob}{Frob}
\DeclareMathOperator{\Lie}{Lie}

\DeclareMathOperator{\Def}{Def}

\DeclareMathOperator{\Diag}{Diag}
\DeclareMathOperator{\End}{End}
\DeclareMathOperator{\Res}{Res}

\DeclareMathOperator{\gr}{gr}

\DeclareMathOperator{\GL}{GL}

\DeclareMathOperator{\wt}{wt}
\DeclareMathOperator{\Ad}{Ad}

\DeclareMathOperator{\Char}{Char}

\DeclareMathOperator{\ad}{ad}

\DeclareMathOperator{\SL}{SL}

\DeclareMathOperator{\Tri}{\mathrm{Tri}}

\usepackage{tikz-cd}

\DeclarePairedDelimiter{\vabs}{\lvert}{\rvert}

\DeclarePairedDelimiter{\scalar}{\langle}{\rangle}
\DeclarePairedDelimiter{\set}{\{}{\}}

\newcommand{\Art}{\mathrm{Art}}
\newcommand{\rig}{\mathrm{rig}}
\newcommand{\tildeg}{\tilde{\mathfrak{g}}}
\newcommand{\ab}{\mathrm{ab}}

\newcommand{\pol}{\mathrm{pol}}
\newcommand{\pdR}{\mathrm{pdR}}
\newcommand{\rhobar}{\overline{\rho}}
\newcommand{\tri}{\mathrm{tri}}
\newcommand{\qtri}{\mathrm{qtri}}

\newcommand{\nc}{\mathrm{nc}}

\usepackage{tikz}
\usetikzlibrary{matrix,arrows}
\definecolor{cqcqcq}{rgb}{0.752941176471,0.752941176471,0.752941176471}
\definecolor{ffqqqq}{rgb}{0.333333333333,0.333333333333,0.333333333333}
\definecolor{cqcqcq}{rgb}{0.752941176471,0.752941176471,0.752941176471}
\definecolor{qqqqff}{rgb}{0.,0.,1.}
\definecolor{cqcqcq}{rgb}{0.752941176471,0.752941176471,0.752941176471}
\definecolor{ffqqqq}{rgb}{1.,0.,0.}
\usepackage{mathabx,epsfig}

\usepackage[all]{xy}
 \def\dar[#1]{\ar@<2pt>[#1]\ar@<-2pt>[#1]}
 \def\tar[#1]{\ar@<4pt>[#1]\ar@<0pt>[#1]\ar@<-4pt>[#1]}

\theoremstyle{definition} 
\newtheorem{definen}{Definition}[section]
\newtheorem{defin}[definen]{Definition}
\newtheorem{hypothese}[definen]{Hypothesis}

\theoremstyle{plain} 

\newtheorem{theor}[definen]{Theorem}
\newtheorem{lemma}[definen]{Lemma}  
\newtheorem{prop}[definen]{Proposition}

\newtheorem{cor}[definen]{Corollary}

\theoremstyle{remark} 
\newtheorem{rema}[definen]{Remark}
\newtheorem{exemple}[definen]{Example}

\definecolor{vert}{rgb}{0.09,0.62,0.40}

\begin{document}

\title{The infinite fern in higher dimensions}

\author{Valentin Hernandez}
\address{Université Paris-Saclay, CNRS, Laboratoire de mathématiques
  d’Orsay, 91405, Orsay, France}
\email{valentin.hernandez@math.cnrs.fr}
\author{Benjamin Schraen}
\address{Université Paris-Saclay, CNRS, Laboratoire de mathématiques
  d’Orsay, 91405, Orsay, France}
\address{Institut Universitaire de France}
\email{benjamin.schraen@universite-paris-saclay.fr}
\date{}

\maketitle

\section{Introduction}

Let $p$ be a prime number, $E$ a number field and
\[ \overline{\rho} : \Gal(\overline E/E) \fleche \GL_n(\overline\FF_p),\]
a continuous (i.e.~which factors through a finite extension)
representation. It was shown by Mazur that the set of deformations of
$\overline\rho$, with minor technical conditions, with values in
finite extensions of $\QQ_p$, can be arranged in a natural rigid space
$\mathcal X_{\overline\rho}$. Conjecturally, and now in many known
cases (\cite{GRFA,CH,Shin,HLTT,SchTor}), we can use automorphic
representations to  construct Galois representations, and in
particular points in $\mathcal X_{\overline\rho}$ that we call
\emph{of automorphic nature}, or just \emph{automorphic}. It is then
natural to wonder which structure has this set of automorphic points in $\mathcal X_{\overline\rho}$: is it an algebraic subspace ? a closed one ? is it Zariski dense ?

The first example beyond the case of characters was studied by
Gouvêa and Mazur. In this situation $E = \QQ$ and $n=2$, and
$\overline\rho$ is irreducible, modular and unobstructed, so that
$\mathcal X_{\overline\rho}$ is a 3-dimensional open ball. In that
case, automorphic points are related to modular forms. In
\cite{GouveaMazur}, Gouvêa and Mazur show that  automorphic points are
Zariski dense in $\mathcal{X}_{\overline\rho}$ using the so called
\emph{infinite fern}. Let us explain this name: up to twisting by
powers of the cyclotomic character, we can replace $\mathcal
X_{\overline\rho}$ by a two dimensional open ball. Modular forms (of
finite slope) can be interpolated by a geometric object, the
Coleman--Mazur \emph{Eigencurve} $\mathcal E$ (\cite{CM}), which is a
rigid-analytic curve, whose points are refined $p$-adic modular forms of
finite slope. Generically, a classical modular forms has two
refinements, thus gives rise to two distinct points in the
Eigencurve. Moreover the points corresponding to refined classical
modular form are Zariski-dense in the Eigencurve. By $p$-adic
interpolation, it is possible to associate a $2$-dimensional $p$-adic
representation of $\Gal(\overline{\QQ}/\QQ)$ to a point of
$\mathcal{E}$. The points giving rise to deformations of
$\overline\rho$ form a union $\mathcal{E}(\overline{\rho})$ of
connected components of $\mathcal{E}$ and the universal property of
$\mathcal{X}_{\overline{\rho}}$ implies the existence of a map
\[ \mathcal E(\overline{\rho}) \fleche \mathcal X_{\overline\rho}.\]
Generically, each modular point $f$ in $\mathcal X_{\overline\rho}$ has two preimage in $\mathcal E$, giving rise to two distincts small curves around those preimages, whose image in $\mathcal X_{\overline\rho}$ meet only at $f$. By density, each of these two small curves has a Zariski-dense set of modular points, and for each of these points there is another small curve passing through, and so on, giving a fractal-like object which we picture as follows :
\begin{center}\includegraphics[scale=0.4]{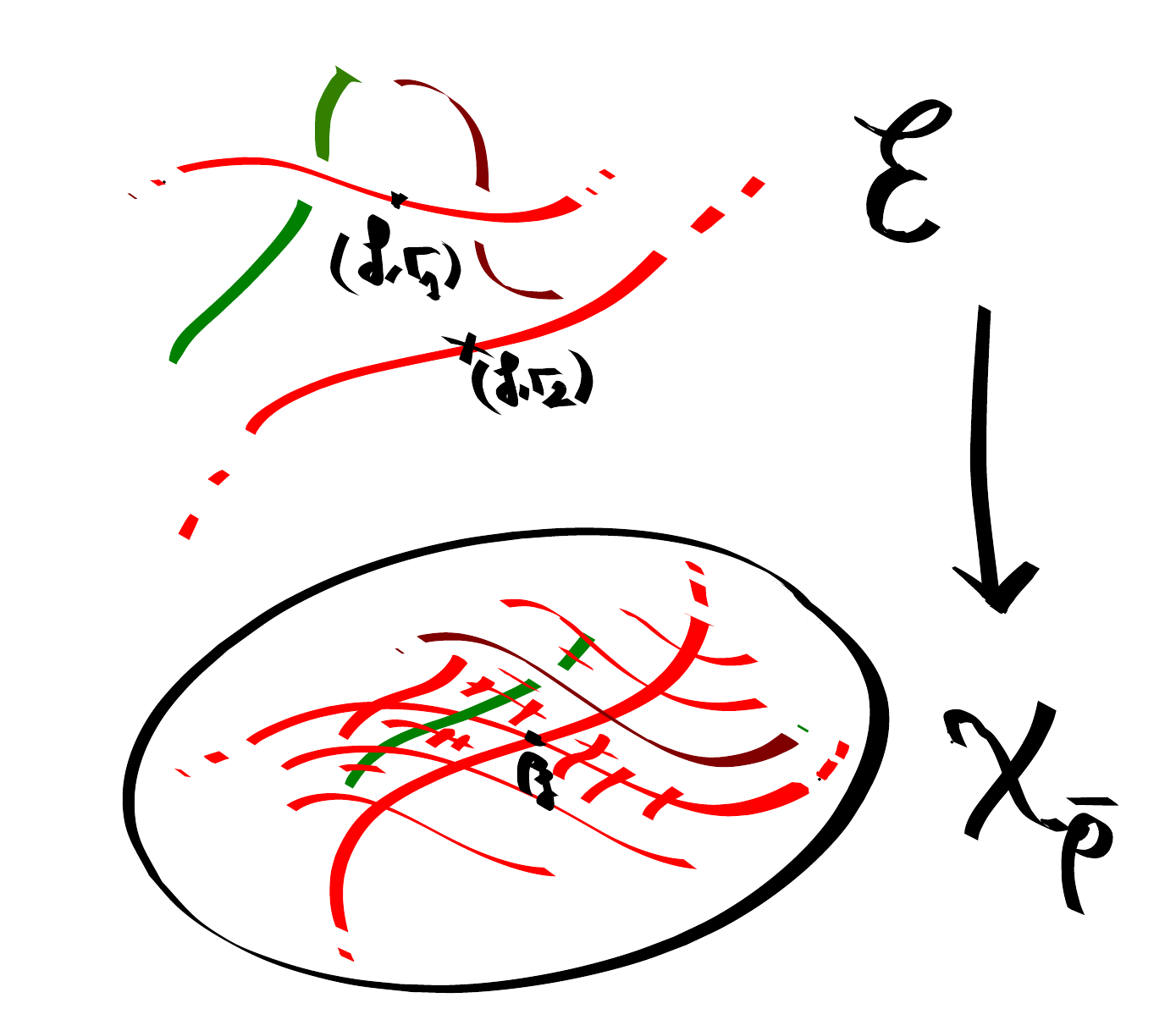}\end{center}
giving a justification for the name of the \emph{infinite fern}. 

This article deals with a generalization of this result to more
general number fields and greater values of $n$. First, we need to
assume that the number field $E$ is a CM field, with totally real
field $F$, in order to be able to associate Galois representations to
automorphic representations. Second, {it is expected} that for general
$n$ the automorphic points are not Zariski dense in
$\mathcal{X}_{\overline{\rho}}$, thus we reduce to the case of
\emph{$\chi$-polarized} Galois representations, for a character $\chi : G_E \fleche \overline{\QQ_p}^\times$,
i.e.~continuous group homomorphisms $\rho : \Gal(\overline E/E) \fleche \GL_n(\overline \QQ_p)$ such that 
\[ \rho^\vee \simeq \rho^c \otimes \chi\eps^{n-1}, \quad \text{where } \rho^c := \rho(c \cdot c^{-1}),\]
where
$\eps$ is the cyclotomic character, and $c\in\Gal(\overline E/F)$ is
a lift of the unique non trivial element of $\Gal(E/F)$. Fix $S$ a
finite set of primes of $E$ containing all primes above $p$. In this
situation, assume that $\overline\rho$ is $\chi$-polarized, absolutely
irreducible (for simplicity) and unramified away from $S$. Let
$R_{\overline\rho}^{\chi-\pol}$ be the
complete noetherian local algebra parametrizing deformations of $\overline\rho$ which are
$\chi$-polarized and unramified away from $S$. Its rigid fiber
$\mathcal X_{\overline\rho}^{\chi-\pol}$ is a rigid space of dimension
at least $[F:\QQ]\frac{n(n+1)}{2}$. A natural source of automorphic
points in $\mathcal X_{\overline\rho}^{\chi-\pol}$ is given by the
regular, algebraic, essentially polarized, cuspidal automorphic
representations of $\GL_n(\mathbb A_E)$, by work of many authors
(\cite{HT,CHT,GRFA,CH,Shin} for example). In this paper, we make the following hypothesis,
\begin{hypothese}
\begin{enumerate}
\item $\overline\rho$ is conveniently modular (see Definition \ref{defin:modular}),
\item All primes above $p$ in $F$ are unramified, and split in $E$,
\item the character $\chi$ satisfies $\chi = \chi^c$ and satisfies a sign condition (see Hypothesis \ref{hypchi} and section \ref{sect:signs})
\end{enumerate}
\end{hypothese}

Under the previous hypothesis, we prove the following result:

\begin{theor}
The Zariski-closure of automorphic points contains a (non empty) union of irreducible components of $\mathcal X_{\overline\rho}^{\chi-\pol}$, each of which are of dimension $[F:\QQ]\frac{n(n+1)}{2}$.
\end{theor}

We say that our deformation problem is unobstructed if $H^2(G_{F,S},\ad(\overline r)) = \{0\}$ where $\overline r$ is some extension of $\overline\rho$ to $G_{F,S}$ the Galois group of the maximal unramified extension of $F$ (see section \ref{sect:defspaces}). In this situation, we know that $\mathcal X_{\overline\rho}^{\chi-\pol}$ is a rigid open unit ball in $[F:\QQ]\frac{n(n+1)}{2}$-variables.

\begin{cor}
Under the previous hypothesis, if moreover $\overline\rho$ is unobstructed, then automorphic points are Zariski dense in $\mathcal X_{\overline\rho}^{\chi-\pol}$.
\end{cor} 

\begin{rema}
  In \cite{Gui} Giraud proved that if $\pi$ is an extremely regular
  automorphic representation of $\GL_n(\mathbb{A}_E)$ (see
  \cite{BLGGTChange} section 2.1), then there exists a density 1 set
  of primes $\lambda$ of $E$ such that $\rho_{\pi,\lambda}$ is
  unobstructed. As we have assumed $\overline\rho$ to be conveniently
  modular, we can actually find some extremely regular $\pi$, so that
  $\overline\rho = \overline\rho_{\pi,\lambda}$ for $\lambda |p$, and
  thus using \cite{Gui}, up to change $\lambda$ in a density 1 set, we
  can assume unobstructedness. In particular, under our assumption,
  $\overline\rho$ is part of a compatible system for which in a
  density 1 set of primes, we have Zariski density of automorphic
  points in the associated deformation spaces.
\end{rema}

Before explaining the strategy of proof, let us say what was
known. The first case was the non-polarized case, $n= 2$ and
$E=F=\QQ$, when unobstructed, which was proven by Gouvêa-Mazur
\cite{GouveaMazur}, and generalised by Böckle \cite{Bockledense}. The non-polarised case for $n=2$ and totally real fields $E=F$, and the polarised case of $n=3$ (and general CM fields $E/F$) was proved by Chenevier (\cite{CheFern}). A generalisation for greater $n$; but under restrictive hypothesis  (of Taylor--Wiles type) on $E$ and $\overline\rho$, was proven recently by Hellmann--Margerin--Schraen (\cite{HMS}). All of these proofs use the analogue (in higher dimensions) of the infinite fern. Now we explain our strategy together with the relation to the previous works.

The Galois representations that we study can be viewed as $p$-adic
($L-$ or $C-$) parameters of a reductive group: $\GL_2/\QQ$ in the
situation of Gouvêa-Mazur and $U_{E/F}(n)$, or $GU_{E/F}(n)$, or one
of its inner forms, a (similitude) unitary group in $n$-variables, for
polarized deformation problems. A natural source of automorphic points
is given by automorphic representations of these groups. It turns out
that these groups give rise to Shimura data, that we can use to
construct $p$-adic refined families of automorphic forms, that is
$p$-adic automophic eigenforms together with the extra data of a
refinement. These families generalize the Eigencurve of Coleman--Mazur and are called Eigenvarieties (see \cite{CM,Che1,Urb,Eme,AIP,Her4} for example). Let us assume for simplicity $\chi = 1$ for the rest of the introduction.

For general $n$, a given automorphic form $f$ has at most $n![F:\QQ]$ refinements $f_i$, and generically exactly $n![F:\QQ]$ refinements. Moreover the Eigenvariety $\mathcal E$ has equidimension $n[F:\QQ]$, $
\mathcal X_{\overline\rho}^{\chi-\pol}$ has dimension at least (but conjecturally exactly) $\frac{n(n+1)}{2}[F:\QQ]$ and there is a map $\mathcal E(\overline\rho)\footnote{an open-closed subvariety of $\mathcal E$} \fleche 
\mathcal X_{\overline\rho}^{\chi-\pol}$ which forgets the refinement. 
\begin{defin}
The image of the map $\mathcal E(\overline\rho) \fleche \mathcal
X_{\overline\rho}^{\chi-\pol}$ is called the \emph{infinite fern} and is denoted by $\mathcal F(\overline\rho)$.
\end{defin}

Actually we can make our main theorem more precise:
\begin{theor}
\label{thr:main}
Under the previous hypothesis, the Zariski closure of the infinite fern $\mathcal F(\overline\rho)$ in $\mathcal X_{\overline\rho}^{\chi-\pol}$ is a non-empty union of irreducible components, each of which are of dimension $[F:\QQ]\frac{n(n+1)}{2}$.
\end{theor} 

Let us comment the various hypothesis we made. The main advantage compared to \cite{HMS} is that we don't have any hypothesis on $p$, $E/F$ or $n$, and moreover
we don't need to assume $\overline\rho$ absolutely irreducible (or has big image) if we
use Chenevier's determinants, which we do (see section
\ref{sect:defspaces}). The hypothesis of being conveniently modular is
necessary to expect the infinite fern to be non-empty, and is in practice very close to the usual modularity hypothesis which is anyway necessary. The hypothesis on the splitting of primes above $p$ is technical, and we hope to come back on this question soon. 

Following the strategy of Chenevier, the 
main goal is to prove that for a Zariski dense set of points $\rho$ in the infinite fern, the part of the tangent space at $\rho$ in $\mathcal X_{\overline\rho}^{\pol}$ coming from $\mathcal 
E$ has dimension at least $\frac{n(n+1)}{2}[F:\QQ]$.
This will imply that the closure of the infinite fern $\mathcal F(\overline\rho)$, has dimension at least $\frac{n(n+1)}{2}[F:\QQ]$. As, by construction, automorphic points
are Zariski-dense in $\mathcal E$, thus in $\mathcal F(\overline\rho)$, this will prove that the Zariski closure of automorphic points has dimension at least $\frac{n(n+1)}{2}[F:\QQ]$.
Thus to prove the assertion on the tangent space, we need to show that the tangent spaces $T_{f_i}\mathcal E$ are ``transverse'', more precisely that the sum of the images of the tangent spaces $T_{f_i}\mathcal E$ in
$T_{\rho_f}\mathcal X_{\overline\rho}^{\chi-\pol}$, for well chosen
automorphic forms $f$, has large enough dimension. But, clearly, as
soon as $n \geq 3$ it is not sufficient that these tangent
spaces are pairwise distinct, and this is the main difficulty to
extend the proof of Gouv\^ea and Mazur.
To overcome this problem, Chenevier suggested a strategy which he
applied successfully when $n =3$ and which can be sketched as follows:
\begin{enumerate}
\item find a good Zariski-dense subset $D$ of the infinite fern
  $\mathcal F(\overline\rho)$ ;
\item show that the analogue of the transversality on the tangent spaces of
  points in $D$ but for local deformation rings is valid ;
\item show that the Global situation ``embeds well'' in the local situation, and thus gives the result.
\end{enumerate}

For the first part, Chenevier suggests to look at automorphic points
$\rho$ which he calls \textit{generic}: they are crystalline at $p$
and all their refinements are \textit{non-critical}. More precisely,
if $\rho$ is crystalline at a place $v\mid p$, its restriction
$\rho_v$ to a decomposition group at $v$ is characterized by a
$n$-dimensional vector space $V = D_{\cris}(\rho|_{\Gal(\overline{E}_v/E_v)})$ with
its Hodge-Tate filtration $\mathcal F_{HT}$ (a complete flag) and
Frobenius operator $\varphi$. The refinements of $\rho$ correspond to
the complete flags of $V$ stable by $\varphi$. We say that a
refinement of $\rho$ is
non critical it is opposite to $\mathcal{F}_{HT}$. Actually,
Chenevier proved that the second step works for crystalline points which have $n$ \emph{well-positioned}
non-critical flags and call those points \emph{weakly-generic}. He moreover proved that weakly generic points (with some extra but harmless conditions) are Zariski dense in the infinite fern when $n = 3$ and uses those as the subset $D$.

Concerning the second point, the weakly generic condition is used to
carry an induction in the local situation and prove that tangent spaces of local, refined, deformations problem spawn the tangent space of the full local deformation ring. This is where the definition of \emph{well positioned} refinements comes from.

For the last point, actually it is enough to embed the situation at
the level of tangent spaces. Chenevier proves that for all preimages
of points in $D$, the map $\mathcal E \fleche \mathcal
W$ to the weight space is \'etale, and deduces that some Selmer group vanishes at those points, allowing to \emph{embed} infinitesimally the global situation into the local one, transversally to the cristalline locus, and thus deduce the result. This last argument is classical in the Taylor-Wiles method.
 
The main issue to generalise Chenevier's strategy in higher dimensions is that it is completely unclear that weakly generic points satisfies some density assumptions when $n\geq 4$ (see remark \ref{rema:WG}). 
 
The strategy of \cite{HMS} is different but shares some similarities:
for the first point they choose points which are crystalline with some
genericity assumption\footnote{precisely on the Frobenius eigenvalues
  and Hodge-Tate weights} which are less restrictive than being
\emph{generic} or \emph{weakly-generic}. Their set $D$ is then
automatically Zariski dense. For the second point, they use a local
model for the local deformations spaces, which is of purely geometric
nature, and a rather evolved but completely elementary argument allows
to conclude in the second point, using not only all refinements but
also \emph{companion points}, which are extra-points appearing when
the refinement is critical (whose existence is proved in \cite{BHS}) and relies on the Taylor-Wiles hypothesis.
 
Then the third point is the most delicate one and is proved by
Taylor--Wiles--Kisin method via ``patched eigenvarieties'' (see
\cite{BHS}). The second and third point relies deeply on the Taylor-Wiles hypothesis, in particular it crucially needs $\overline\rho$ to have adequate image.

In this article we use a strategy closer to Chenevier's, but using the
local model of \cite{BHS} as in \cite{HMS}, but we never need any Taylor-Wiles assumption. Namely, using the local
model and a careful study of its geometry, we first prove the second
point without using companion points but rather generalizing
Chenevier's transversality result at critical refinements (see section
\ref{sect:localdefrings}). For the first point, we show that setting
for $D$ the set of crystalline points satisfying genericity conditions as
in \cite{HMS} and which have moreover \emph{enormous} image are
actually Zariski dense in the infinite fern ; we call those points
\emph{almost-generic} (see Definition \ref{defin:ag}) because they
will replace Chenevier's generic points in our argument. The density
of these points is far from being automatic and the argument is originally due to Bellaïche--Chenevier and Taïbi (see section \ref{sect:aggen}). Then, for the third point, we show that using the enormous image, and a result of Newton--Thorne, we have the vanishing of the expected Selmer group at points of $D$. We then show that this can be used to relate the global situation to the local one. As a byproduct, we obtain that our Eigenvariety is smooth at those points, as it was the case in other situations (see \cite{BHS,Bergdall}) (see section \ref{sect:localglobal}).
Then, a local calculation which was previously carried out in \cite{All1}, we show that our almost generic points are smooth points of $\mathcal X_{\overline\rho}^{\chi-\pol}$ of the expected dimension.

The results of Chenevier were combined with those by Allen (\cite{All2}) (who proves that under some hypothesis every component contains an automorphic point) to prove full density of the infinite fern when $n \leq 3$, without assuming unobstructedness. We can adapt this generalisation also here,

\begin{cor}[Allen]
Assume hypothesis \ref{HypAllen+}, then the infinite fern is Zariski dense in $\mathcal X^{\chi-\pol}_{\overline\rho}$.
\end{cor}

The only thing we need to check for this corollary is that we use classical points which are automorphic representations for a similitude unitary group, which moreover contributes to the coherent $H^0$, whereas Allen's proof a priori only provides an (essentially) polarised autormorphic representation of $\GL_n$.

 \textbf{Acknowledgements}: We would like to warmly thank Ga\"etan
 Chenevier for very helpful suggestions concerning this work. We would also like to thank George Boxer, Laurent Clozel and Olivier Schiffmann for many interresting discussions. Finally, we would like to thank Anne vaugon and Marc Mezzarobba for their help for using SAGE to compute the local tangent spaces.

\section{A lemma on Borel enveloppes}
\label{sec:Borel_enveloppes}

In this section we prove a technical, but essential, generalization of the linear algebra result \cite[Thm.~2.3]{HMS}. The main result of this section is Lemma \ref{lemmaMargerin}.

Fix $n$ an integer, $k$ a field, $G = \GL_{n/k}$, $B$ the upper triangular Borel, $T$ its diagonal torus, $\mathfrak{g},\mathfrak{b}, \mathfrak{t}$ their respective Lie algebras. We also set $\mathfrak{u}$ the nilpotent radical of $\mathfrak{b}$. Let $W$ be the Weyl group of the pair $(G,T)$ that we identify to $\mathfrak{S}_n$. Let $w_0 \in W \simeq \mathfrak{S}_n
$ be the longuest element for the order given by $\mathfrak{b}$. For $g,h \in G(k)$, we set $\mathfrak{b}_{g} = g^{-1}\mathfrak{b} g$ and $(\mathfrak{b}_g \cap \mathfrak{b}_h)^{gr = w_0}$ the $k$-linear subspace of $M \in \mathfrak{b}_g \cap \mathfrak{b}_h$ such that, in $\mathfrak{t}$,
\[ hMh^{-1} \pmod{\mathfrak u} = \Ad(w_0)(gMg^{-1} \pmod{\mathfrak u}) \in \mathfrak{t}.  \]
Let $\mathfrak{C}_n$ be the set of ``full'' cycles:
\[ \mathfrak C_n \coloneqq \{ c_{i,j} := (i,i-1,\dots,j+1,j) \in \mathfrak{S}_n | i \geq j\}.\]

\begin{lemma}
\label{lemmaMargerin}
For every Borel subalgebra $\mathfrak{b}'\subset\mathfrak{g}$, we have
\[ \mathfrak{b}' = \sum_{w \in \mathfrak{C}_nw'} (\mathfrak{b}' \cap \mathfrak{b}_{w})^{gr = w_0},
\]
for some element $w'\in\mathfrak{S}_n$ depending on $\mathfrak{b}'$. Moreover if $\mathfrak{b'}=\mathfrak{b}_{w_0b}$ for some $b\in B(k)$, we can choose $w'=1$.
\end{lemma}

\begin{proof}
We can write $\mathfrak{b}' = \mathfrak{b}_g$ for $g \in G(k)$. Let $g = u l s$, $u \in \mathfrak{b}, l \in \mathfrak{b}_{w_0}$ lower triangular i.e.~$l = w_0 b w_0$ with $b \in B(k)$, and $s \in W$. Thus $\mathfrak{b}_g = (\mathfrak{b}_{w_0})_{bq}$, for $q = w_0s \in W$.
Up to conjugate by $bq$, we check at once that it is enough to show
\[ \mathfrak{b}_{w_0} = \sum_{w \in \mathfrak C_n} (\mathfrak{b}_{w_0}\cap\mathfrak{b}_{wb^{-1}})^{gr=w_0}.\]
Thus we reduce to show the following lemma :

\begin{lemma}
For all $i \geq j$, there exists $x_\ell \in k$ such that
\[ a^{i,j} := \delta_{i,j} + \sum_{\ell = j+1}^{i} x_\ell \delta_{\ell,j} \in (\mathfrak{b}_{w_0}\cap\mathfrak{b}_{c_{i,j}b^{-1}})^{gr=w_0}.\]
\end{lemma}

\begin{proof} We follow the proof of \cite{HMS}. For $i\geq j$, let $a^{i,j}$ be the element constructed at the beginning of the proof of Lemma 2.1 in \emph{loc.~cit.} For the convenience of the reader we recall its construction. Let $e_1,\dots,e_n$ be the standard basis of $k^n$ and let $V_\bullet$ be the standard flag of $k^n$. Let $\mathcal{B}$ be the basis
\[ b(e_1),b(e_2),\dots,b(e_{j-1}),e_j,b(e_{j+1}),\dots,b(e_i),e_{i+1},\dots,e_n\]
of $k^n$. Then $a^{i,j}$ is the matrix, in the standard basis, of the endomorphism $\pi$ of $k^n$ defined by $\pi(x)=0$ if $x\in\mathcal{B}\setminus\set{e_j}$ and $\pi(e_j)=e_i$. As in \emph{loc.~cit.} we check that
\begin{enumerate}[(i)]
\item $e_\ell\in \ker(\pi)$ if $\ell<j$ or $\ell> i$;
\item $\mathrm{Im}(\pi)\subset ke_i$ and $\pi(e_j)=e_i$;
\item the endomorphism $b^{-1}\pi b$ stabilizes the flag $c_{i,j}^{-1} V_\bullet$.
\end{enumerate}
The first two points are checked in \emph{loc.~cit.} The third point follows from the fact that
\begin{equation}\label{eq:stabilite_drapeau} b^{-1}\pi b(c_{i,j}^{-1}V_\ell)=\begin{cases} 0 & \text{if }\ell=1,\dots,i-1\\
kb^{-1}(e_i) & \text{if }\ell \geq i\end{cases}\end{equation}
and $kb^{-1}(e_i)\in V_i=c_{i,j}^{-1}V_i$. This shows that the matrix $a^{i,j}$ is an element of $\mathfrak{b}_{w_0}\cap\mathfrak{b}_{c_{i,j}b^{-1}}$ and has the form
\begin{equation}\label{eq:matrixaij} \delta_{i,j}+\sum_{\ell=j+1}^i x_{\ell}\delta_{i,k}\end{equation}
for some $x_\ell\in k$.

It remains to check that $a^{i,j}\in (\mathfrak{b}_{w_0}\cap\mathfrak{b}_{c_{j,i}b^{-1}})^{gr=w_0}$ which is equivalent to check that the diagonal elements of $a^{i,j}$ and $\Ad(c_{i,j}b^{-1})a^{i,j}$ are the same. It follows from \eqref{eq:stabilite_drapeau} that
\[ (\Ad(c_{i,j}b^{-1})\pi)(e_\ell)=\begin{cases} 0 & \text{if } k< i \\
c_{i,j}b^{-1}(e_i) & \text{if } k=i \end{cases}\]
and $(\Ad(c_{i,j}b^{-1}\pi))(e_\ell)\in V_i$ if $\ell >i$. Therefore the diagonal of $\Ad(c_{i,j}b^{-1})a^{i,j}$ is zero except for the coefficient $(i,i)$. On the other hand, we see by \eqref{eq:matrixaij} that the diagonal entries of $a^{i,j}$ are all zero except $(i,i)$. As the matrices $a^{i,j}$ and $\Ad(c_{i,j}b^{-1})a^{i,j}$ are conjugated they have the same trace and thus the same diagonal.
\end{proof}
\end{proof}

\section{Local deformation rings}
\label{sect:localdefrings}

The aim of this section is to prove Proposition \ref{propraf} which says that the tangent space of the deformation ring of a crystalline $\varphi$-regular and Hodge-Tate regular $(\varphi,\Gamma)$-module is generated by the tangent space of some quasi-trianguline deformation subspaces.

Let $k$ be a field of characteristic $0$. Let $G$ be a split reductive
group over $k$, $B\subset G$ a Borel subgroup of $G$, $T\subset B$ a
maximal split torus of $G$, $U$ the unipotent radical of $B$ and $U^-$
the unipotent radical of the opposite Borel subgroup to $B$ with
respect to $T$ (in particular $U^-\cap B=\set{1}$). Let
$W\coloneqq N_G(T)/T$ be the Weyl group of $(G,T)$. We denote $w_0$ the longest element of $W$ with respect to the Bruhat order induced by the choice of $B$. Let
$\mathfrak{g}$, $\mathfrak{b}$, $\mathfrak{t}$, $\mathfrak{u}$,
$\mathfrak{u}^-$ be the respective Lie algebras of $G$, $B$, $T$, $U$,
$U^-$. Let $\tildeg\subset\mathfrak{g}\times G/B$ be the Grothendieck
simultaneous resolution of $\mathfrak{g}$ and
$X\coloneqq\tildeg\times_{\mathfrak{g}}\tildeg$. We recall that $X$
has irreducible components $X_w$ which are indexed by the elements of
the Weyl group $W$ (see \cite[Def.~2.2.3]{BHS}). The map
$\tildeg\rightarrow\mathfrak{t}$ sending $(\psi,gB)$ to the projection
of $\Ad(g)^{-1}\psi$ on $\mathfrak{t}$ via
$\mathfrak{b}/\mathfrak{u}\simeq\mathfrak{t}$ gives rise to two
different maps $\kappa_1,\kappa_2 : X\rightarrow\mathfrak{t}$
corresponding to the two projections $X\rightarrow\tildeg$ and to a
map
$\kappa\coloneqq(\kappa_1,\kappa_2) :
X\rightarrow\mathfrak{t}\times_{\mathfrak{t}/W}\mathfrak{t}$. If
$w\in W$, we let $\mathfrak{t}^w\subset\mathfrak{t}$ be the subspace
of elements fixed by $w$ and
$T_w\subset\mathfrak{t}\times_{\mathfrak{t}/W}\mathfrak{t}$ be the
irreducible component
\[ T_w\coloneqq\set{(x_1,x_2)\in\mathfrak{t}\times\mathfrak{t}\mid x_1=\Ad(w)x_2}.\]

The space $X$ has a partition by locally closed subschemes $V_w$
defined as inverse images of the Bruhat strata $U_w\subset G/B\times
G/B$ by the map $\pi : X\rightarrow G/B\times G/B$ and $X_w = \overline{V_w}$. We have an
inclusion $\kappa(X_w)\subset T_w$ (\cite[Lem.~2.5.1]{BHS}).

\begin{prop}\label{prop:tangentcompw0ptcrist}
Let $x=(g_1B,0,g_2B)\in X_{w_0}(k)\subset G(k)/B(k)\times\mathfrak{g}\times G(k)/B(k)$ be a $k$-point. Let $w\in W$ be such that $x\in V_w$ and assume that $w_0w^{-1}$ is a product of distinct simple reflexions. Then we have an equality of $k$-vector spaces
\[ T_x X_{w_0}=T_x \kappa^{-1}(T_{w_0}).\]
\end{prop}

\begin{proof}
The inclusion $X_{w_0}\subset\kappa^{-1}(T_{w_0})$ induces an inclusion $T_x X_{w_0}\subset T_x\kappa^{-1}(T_{w_0})$. We will prove that these two $k$-vector spaces have the same dimension.

Let $k[\eps]\coloneqq k[X]/(X^2)$. The tangent space $T_x(\kappa^{-1}(t_{w_0}))$ is the set of $k[\eps]$-points $(\tilde{g}_1B,\eps A,\tilde{g}_2B)$ of $X$ specialising to $x$ such that moreover
\begin{equation}\label{conditionw0} \Ad(\tilde{g}_1)^{-1}(\eps A)=\Ad(w_0)\Ad(\tilde{g}_2)^{-1}(\eps A)\end{equation}
in $\mathfrak{t}\otimes_k k[\eps]$. Let $\tilde{x}=(\tilde{g}_1B,\eps A,\tilde{g}_2B)$ be such a point. We can write $\tilde{g}_i=g_i(1+\eps h_i)$, where $h_i\in\mathfrak{u}^-$. Using $\eps^2=0$, the condition $\tilde{x}\in X(k[\eps])$ is equivalent to $\Ad(g_i^{-1})A\in\mathfrak{b}$ for $i\in\set{1,2}$. The condition \eqref{conditionw0} is then equivalent to $\Ad(g_1^{-1})A=\Ad(w_0)\Ad(g_2^{-1})A$ in $\mathfrak{t}$. Note that, up to changing $x$ by a point of its $G(k)$-orbit, we can assume, without changing the dimensions of the tangent spaces, that $g_1=1$ and $g_2=w$. The conditions above are then equivalent to
\[ A\in \mathfrak{t}^{w_0w^{-1}}+(\mathfrak{u}\cap\Ad(w)\mathfrak{u})\]
which is a $k$-vector space of dimension $\dim_k \mathfrak{t}^{w_0w^{-1}}+\lg(w_0w^{-1})$. As $w_0w^{-1}$ is a product of distinct simple reflexions, we have (\cite[Lem.~2 \& 3]{Carter})
\[ \dim_k\mathfrak{t}^{w_0w^{-1}}=\dim_k\mathfrak{t}-\lg(w_0w^{-1}).\]
Namely, we have a $W$-equivariant isomorphism $\mathfrak{t}\simeq\Hom(X^*(T),k)$, so that it is sufficient to prove that
\[ \dim_\RR(X^*(T)\otimes\RR)^{w'}=\dim_\RR(X^*(T)\otimes\RR)-\lg(w')\]
when $w'$ is a product of simple reflexions. Let $V$ be the subspace of $X^*(T)\otimes\RR$ generated by the roots of $(G,T)$. It is stable under $W$ and has a direct summand on which $W$ is acting trivially. It is therefore sufficient to prove that
\[ \dim_\RR V^{w'}=\dim_\RR V-\lg(w').\]
As $W$ is a finite group, $\dim_\RR V-\dim_\RR V^{w'}$ is equal to the number of eigenvalues of $w'$ acting on $V$ which are different from $1$. By \cite[Lem.~2]{Carter}, this number is equal to the number $l(w')$ of \emph{loc.~cit.} (which is a priori \emph{not} $\lg(w')$). By \cite[Lem.~3]{Carter}, we have $l(w')=\lg(w')$ since the set of simple roots is a set of linearly independent vectors of $V$.

Finally we deduce that
\[ \dim_k T_x(\kappa^{-1}(T_{w_0}))\leq\dim_k(G/B\times G/B)+\dim_k\mathfrak{t}=\dim G.\]
On the other hand, we know that $X_{w_0}$ is irreducible of dimension $G$. Consequently we have
\[ \dim G\leq\dim_k T_x X_{w_0}\leq\dim_k T_x\kappa^{-1}(T_{w_0})\leq\dim G\]
so that $T_x X_{w_0}=\dim_k T_x\kappa^{-1}(T_{w_0})$.
\end{proof}

\begin{lemma}\label{lem:densityTorbit}
Let $w\in W$ and $b\in B$. Then the point $wB$ is in the closure of the $T$-orbit of $bwB$ in $G/B$.
\end{lemma}

\begin{proof}
Let $\nu$ be a cocharacter of $T$ such that $\scalar{\nu,\alpha}>0$ for all positive root $\alpha$ of $(G,B,T)$. Then the map $\mathbb{G}_m\rightarrow G$ defined by $\nu u\nu^{-1}$ for $u$ in the unipotent radical of $B$ extends to a map $\mathbb{A}^1\rightarrow G$ sending $0$ to $1$, thus as does the map $\nu b \nu^{-1}$. Consequently, as $w$ normalises $T$, $wB$ is in the closure of the image of the map
\[t\mapsto \nu(t)bwB=\nu(t)b\nu(t)^{-1}wB.\qedhere\]
\end{proof}

\begin{lemma}\label{lem:relationBruhat}
Let $(w_1,w_2)\in W^2$ and let $b\in B$. If $(w_1B,bw_2B)\in U_w\subset G/B\times G/B$, we have $w_1^{-1}w_2\leq w$ in the Bruhat order.
\end{lemma}

\begin{proof}
If $t\in T$, we have $tw_1B=w_1B$ in $G/B$ so that $(w_1B,tbw_2B)\in U_w$. It follows from Lemma \ref{lem:densityTorbit} that $w_2B$ is in the closure of the set $\set{tbw_2B \mid t\in T}$ so that $(w_1B,w_2B)$ is in the closure if $U_w$. As the closure of $U_w$ is the union of the $U_w'$ with $w'\leq w$ and $(w_1B,w_2B)\in U_{w_1^{-1}w_2}$ we obtain the result.
\end{proof}

From now on we consider $K/\QQ_p$ a finite extension, and denote
$\Gamma = \mathrm{Gal}(K(\zeta_{p^\infty})/K)$. We fix
$L$ a finite extension of $\QQ_p$ that splits $K$, i.e.
\[ L \otimes_{\QQ_p} K \simeq L^{[K:\QQ_p]}\]
and we denote by $k_L$ its residue field. We follow to the
notations of \cite{KPX} concerning $(\varphi,\Gamma)$-modules over
Robba rings. Let $\mathcal{R}(\pi_K)$ be the Robba ring for $K$ (see
\cite{KPX} definition 2.2.2). We define $t \in \mathcal{R}(\pi_K)$ by
$t = \log(1 + \pi_K)$. Let $\mathcal{C}_L$ be the category local
artinian $\mathcal{O}_L$-algebra $A$ with maximal ideal
$\mathfrak{m}_A$ such that the natural map
$k_L\rightarrow A/\mathfrak{m}_A$ is an isomorphism. If $A$ is an
object of $\mathcal{C}_L$, we denote
$\mathcal{R}_A(\pi_K)\coloneqq A\otimes_{\Qp}\mathcal{R}(\pi_K)$. We
refer to \cite[Def.~2.2.12]{KPX} for the notion of
$(\varphi,\Gamma)$-module over $\mathcal{R}_A(\pi_K)$. Let $D$ be a
$(\varphi,\Gamma)$-module over $\mathcal{R}_L(\pi_K)$. We denote
\[ \mathfrak{X}_D : \mathcal C_L \fleche \mathrm{Sets}\] the
deformation functor of $D$, i.e.~for an object $A$ of $\mathcal{C}_L$,
$\mathfrak{X}_D(A)$ is the set of isomorphism classes of pairs
$(D_A,i_A)$ where $D_A$ is a $(\varphi,\Gamma)$-modules over
$\mathcal{R}_A(\pi_K)$ and $i_A : L \otimes_A D_A \simeq D$ is an
isomorphism of $(\varphi,\Gamma)$-modules. If $(\rho,V)$ is a
continuous representation of $G_K$ on a finite dimensional $L$-vector
space, the functor $\mathrm{D}_{\rig}$ of \cite{BergerEDP} induces an
isomorphism of deformation functors (see \cite[\S3.6]{HMS} for
details)
\[ \mathrm{D}_{\rig} : \mathfrak{X}_V \overset{\sim}{\fleche}
  \mathfrak{X}_{\mathrm{D}_{\rig}(V)}.\]

Let
$\mathcal{F} = (0=\mathcal{F}_0\subsetneq\cdots\subsetneq\mathcal{F}_n=
D[t^{-1}])$ be a complete filtration of $D[t^{-1}]$ by
sub-$(\varphi,\Gamma)$-modules over $\mathcal{R}_L(\pi_K)[t^{-1}]$
which are direct factors as $\mathcal{R}(\pi_K)[t^{-1}]$-modules (we
will call such a filtration a\emph{quasi-triangulation}). We define
similarly
\[ \mathfrak{X}_{D,\mathcal F} : \mathcal C_L \fleche \text{Sets}\] the
deformation functor of the pair $(D,\mathcal F)$, i.e.~for $A$ in
$\mathcal{C}_L$, the set $\mathfrak{X}_{D,\mathcal F}(A)$ is the set of
isomorphism classes of triples $(D_A,\mathcal F_{A},i_A)$ where
$(D_A,i_A)\in\mathfrak{X}_D(A)$ and $\mathcal F_{A}$ is a filtration of
$D_A[t^{-1}]$ by $(\varphi,\Gamma)$-stable
$\mathcal{R}_A(\pi_K)[t^{-1}]$-submodules which are direct factors of
$D_A[t^{-1}]$ in the category of
$\mathcal{R}_A(\pi_K)[t^{-1}]$-modules and such that
$i_A(L \otimes_A \mathcal F_{A,i})=\mathcal F_i$ for all $i\in\ZZ$.

We recall some notations of \cite{BHS} section 3 and we refer the
reader to \emph{loc.~cit.} for more precisions. Let $W$ be an
$L\otimes_{\QQ_p} B_{\dR}$-representation of $G_K$ which is almost de
Rham. Let $W^+$ be a $G_K$-stable $L\otimes_{\QQ_p} B_{\dR}^+$-lattice
of $W$. Let $\mathfrak{X}_W : \mathcal C_L \fleche \text{Sets}$ be the
deformation functor of $W$, which means that $\mathfrak{X}_W(A)$ is
the set of isomorphism classes of pairs $(W_A^+,i_A)$ such that
$W_A^+$ is a finite free $A\otimes_{\QQ_p}B_{\dR}^+$-module endowed
with a continuous semilinear action of $G_K$ and $i_A$ is a
$G_K$-equivariant isomorphism $L\otimes_AW_A^+\simeq W^+$ of
$L\otimes_{\QQ_p}B_{\dR}^+$-modules. If we fix an
$L\otimes_{\Qp}K$-linear isomorphism
$\alpha : (L\otimes_{\QQ_p} K)^n \overset{\sim}{\fleche} D_{\pdR}(W)$
we can define
$\mathfrak{X}_{W^+}^\Box : \mathcal C_L \fleche \text{Sets}$ the
deformation functor of the pair $(W^+,\alpha)$.
Let $F_\bullet$ be a $G_K$-stable flag of
$L\otimes_{\QQ_p}B_{\dR}$-submodules of $W$, we define
$\mathfrak{X}_{W^+,{F}_\bullet}$ the deformation functor of
the pair $(W^+,{F}_\bullet)$ and
$\mathfrak{X}_{W^+,{F}_\bullet}^\Box$ the deformation functor
of the triple $(W^+,{F}_\bullet,\alpha)$.

\begin{rema} 
  \label{rema:basealpha} Assume that $D$ is a crystalline,
  $\varphi$-generic, $(\varphi,\Gamma)$-module (see
  \cite[\S3.3]{HMS}). In this case, it can be convenient to choose
  $\alpha$ compatible with the Frobenius structure. The isocrystal
  $(D_{\cris}(D),\varphi)$ has exactly $n!$-refinements, i.e.~complete
  flags stable by $\varphi$, which are in natural bijection with the
  orderings of the eigenvalues of the linearized Frobenius
  $\varphi^f$. By $\varphi$-genericity, there exists an isomorphism
  $\alpha_0 : (L\otimes_{\QQ_p}K_0)^n\simeq
  L\otimes_{\QQ_p}D_{\cris}(D)$ sending the canonical basis on an
  eigenbasis of $\varphi$. Then $\alpha_0$ induces a bijection between
  the complete flags of $(L\otimes_{\QQ_p}K_0)^n$ stable under the
  group of diagonal matrices and the set of refinements of
  $D_{\cris}(D)$. Denote
  $\alpha=\alpha_0\otimes_{K_0}\Id_K : (L\otimes_{\QQ_p}K)^n\simeq
  D_{\dR}(D) = D_{\pdR}(D)\simeq D_{\cris}(D)\otimes_{K_0}K$. By
  \cite{HMS} Lemma 3.7, the map
  $\mathcal{F}\mapsto D_{\cris}(\mathcal{F}[1/t])$ induces a bijection
  between the set of triangulations of $D$ and the set of $n!$
  refinements of $D_{\cris}(D)$. If $\mathcal{F}$ is a triangulation
  of $D$ and $w\in\mathfrak{S}_n\subset\GL_n(L\otimes_{\QQ_p}K)$ we
  denote $w\cdot\mathcal{F}$ the unique triangulation of $D$ such that
  $w\alpha^{-1}(D_{\dR}(\mathcal{F}[1/t]))=\alpha^{-1}(D_{\dR}(w\cdot\mathcal{F}[1/t]))$. This
  defines a simply transitive action of the group $\mathfrak{S}_n$ on
  the set of triangulations of $D$ (which depends on the choice of
  $\alpha$).
\end{rema}

Now we fix $G=\GL_{n,K}$, $B\subset G$ the Borel subgroup of upper
triangular matrices and $T\subset B$ the maximal torus of diagonal
matrices. We recall that $\mathfrak{g}$ is the $K$-Lie algebra of $G$
and
$X=\widetilde{\mathfrak{g}}\times_{\mathfrak{g}}\widetilde{\mathfrak{g}}$. We
also note $X_{K/\Qp}$ and $\widetilde{\mathfrak{g}}_{K/\Qp}$ their Weil
restrictions from $K$ to $\Qp$. If $A$ is an object of $\mathcal{C}_L$
and $(W_A^+,\mathcal{F}_{\bullet,A},\alpha_A)$ is an element of
$\mathfrak{X}_{W^+,\mathcal{F}_\bullet}^\Box(A)$, we can produce an
element of $X_{K/\Qp}(A)$ by sending
$(W_A^+,\mathcal{F}_{\bullet,A},\alpha_A)$ to
$x_A\coloneqq(\alpha^{-1}(\mathrm{D}_{\pdR}(\mathcal{F}_{\bullet})),N_{W_A},\alpha^{-1}(\Fil_{W_A^+}^\bullet))$. By
\cite{BHS} Corollary 3.5.8, this map is a bijection. This implies that
the functor $\mathfrak{X}_{W^+,\mathcal{F}_\bullet}^\Box$ is
pro-represented by the complete local ring of $X$ at $x_L$.

Let $w \in \mathfrak{S}_n^{[K : \Qp]}$. Recall that
$X_{K/\Qp,w}$ is the irreducible component of
$(X_{K/\Qp})_L\simeq (X\times_KL)^{[L:K]}$ associated to $w$. Let $D$
be a crystalline $(\varphi,\Gamma)$-module over
$\mathcal{R}_L(\pi_K)$, together with a filtration $F_\bullet$ of
$D[1/t]$. Let $W^+=W_{\dR}^+(D)$, $W=W^+[t^{-1}]$ and
$\mathcal{F}_\bullet=W_{\dR}(F_\bullet)$. As $D$ is crystalline, the
$B_{\dR}$-representation $W$ is de Rham and thus almost de Rham. The
functors $W_{\dR}$ and $W_{\dR}^+$ induce a morphism of functors
$\mathfrak{X}_{D,F_\bullet}\rightarrow
\mathfrak{X}_{W^+,\mathcal{F}_\bullet}$. If $D$ is moreover assumed to
be $\varphi$-regular, this morphism if formally smooth by
\cite[Cor.~3.5.4]{BHS}.

We define $\mathfrak{X}_{W^+,\mathcal{F}_\bullet}^{w,\Box}$ as the
subfunctor of $\mathfrak{X}_{W^+,\mathcal{F}_\bullet}^\Box$
pro-represented by the quotient of $\widehat{\mathcal{O}}_{X_L,x_L}$
corresponding to the complete local ring of $X_{K/\Qp,w}$ at $x_L$ (with the
convention that it is empty if $x_L\notin X_w$). We also define
$\mathfrak{X}_{W^+,\mathcal{F}_\bullet}^{w}\subset
\mathfrak{X}_{W^+,\mathcal{F}_\bullet}$ as the image of
$\mathfrak{X}_{W^+,\mathcal{F}_\bullet}^{\Box,w}$ via
$\mathfrak{X}_{W^+,\mathcal{F}_\bullet}^\Box \fleche
\mathfrak{X}_{W^+,\mathcal{F}_\bullet}$ and we define
$\mathfrak{X}_{D,F_\bullet}^w\subset \mathfrak{X}_{D,F_\bullet}$ as
the inverse image of $\mathfrak{X}_{W^+,\mathcal{F}_\bullet}^w$ by
$\mathfrak{X}_{D,F_\bullet}\rightarrow
\mathfrak{X}_{W^+,\mathcal{F}_{\bullet}}$.

We assume from now on that $D$ is crystalline and $\varphi$-generic
(see \cite[\S3.3]{HMS}). Let
$\mathfrak{X}_D^{\cris}\subset\mathfrak{X}_D$ be the subfunctor of
crystalline deformations of $D$. Let $\mathcal F$ be a triangulation
of $D$, we use the same symbol for the filtration induced on
$D[1/t]$.

\begin{lemma}\label{lem:comp_cris}
  We have a inclusion
  $\mathfrak{X}_D^{\cris}\subset\mathfrak{X}_{D,\mathcal F}^{w_0}$.
\end{lemma}

\begin{proof}
  It follows from \cite[\S3.3]{HMS} that
  $\mathfrak{X}_D^{\cris}\subset\mathfrak{X}_{D,\mathcal F}$. We fix an
  isomorphism of $L\otimes_{\Qp}K_0$-modules
  $\beta : (L\otimes_{\Qp}K_0)^n\simeq D_{\cris}(D)$ such that
  $\beta\otimes\Id_K=\alpha$. Let $\mathfrak{X}_{D}^{\cris,\Box}$ be
  the functor of crystalline deformations of the pair
  $(D,\alpha)$. Let's consider the composite
  \[
    \mathfrak{X}_{D}^{\cris,\Box}\longrightarrow\mathfrak{X}_{D,F_\bullet}^\Box\longrightarrow
    \mathfrak{X}_{W^+_\dR(D),W_\dR(F_\bullet)}^\Box.\] Let
  $A\in\mathcal{C}_L$ and let
  $(D_A,\alpha_A)\in\mathfrak{X}_D^{\cris,\Box}(A)$ and let
  $(D_A,\alpha_A,F_{\bullet,A})$ be its image in
  $\mathfrak{X}_{D,F_\bullet}^\Box(A)$. As $D_A$ is crystalline, the
  operator $\nu_A$ on $W_{\dR}(D_A)$ is zero.

  Now we remark that the schematic inverse image of $\set{0}$ by the
  natural map $(X_{K/\Qp})_L\rightarrow(\mathfrak{g}_{K/\Qp})_L$ of
  $L$-schemes is contained in the irreducible component
  $X_{K/\Qp,w_0}$. Namely it is sufficient to prove the inverse image
  $Z$ of $\set{0}$ by the natural map of $K$-schemes
  $X\rightarrow\mathfrak{g}$ is contained in $X_{w_0}$. But $Z$ is
  $G/B\times\set{0}\times G/B$ which is the Zariski closure of
  $V_{w_0}\cap(G/B\times\set{0}\times G/B)$, so that
  $Z\subset X_{w_0}$.

  This implies that the image of $(D_A,\alpha_A,F_{\bullet,A})$ in
  $\mathfrak{X}_{W^+_\dR(D),W_\dR(F_\bullet)}(A)$ is contained in
  $\mathfrak{X}_{W^+_\dR(D),W_\dR(F_\bullet)}^{w_0}(A)$ and finally
  that
  $\mathfrak{X}_D^{\cris,\Box}\subset
  \mathfrak{X}_{D,F_\bullet}^{w_0,\Box}$ and
  $\mathfrak{X}_D^\cris\subset\mathfrak{X}_{D,F_\bullet}^{w_0}$.
\end{proof}

\begin{defin}\label{def:transpo_simples}
Let $D$ be a crystalline $\varphi$-generic and HT regular $(\varphi,\Gamma)$-module over $\mathcal{R}_L(\pi_K)$. Let $F_\bullet$ be a triangulation of $D$. Let $w_{(D,F_\bullet}) \in \mathfrak S_n^{[K:\QQ_p]}$ be the element measuring the relative position of $F_\bullet$ and the Hodge filtration of $D$ (see \cite{BHS} before Proposition 3.6.4). We say that $(D,F_\bullet)$ is \emph{associated to a product of distinct transpositions} if $w_0w_{D,F_\bullet}$ is a product of distinct simple transpositions. Moreover we say that the triangulation $F_\bullet$ is \emph{non-critical} if $w_{D,F_\bullet}=w_0$.
\end{defin}

We can now prove the main result of this section.

\begin{prop}
  \label{propraf}
  Let $D$ be a $\varphi$-generic, regular, crystalline,
  $(\varphi,\Gamma)$-module over $\mathcal{R}_L(\pi_K)$. Denote $\Tri(D)$ the set of triangulations of $D$.
  \begin{enumerate}[(i)]
  \item The following $L$-linear map is surjective
   \[ \bigoplus_{\mathcal F \in \Tri(D)}
    T\mathfrak{X}^{w_0}_{D,\mathcal{F}[1/t]} \fleche T
    \mathfrak{X}_D.\]
  \end{enumerate}
   Assume moreover that there exists $\mathcal{F}^{\nc}$ a non-critical triangulation of $D$.
  \begin{enumerate}[(ii)]
  \item\label{propraf1} For any $c\in\mathfrak{C}_n$, the pair $(D,c\cdot\mathcal{F}^{\nc})$ is associated to a product of simple transpositions.
  \item\label{propraf2} The following $L$-linear map
  is surjective:
  \[ \bigoplus_{c\in\mathfrak{C}_n}
    T\mathfrak{X}^{w_0}_{D,(c\cdot\mathcal{F}^{\nc})[1/t]} \fleche T
    \mathfrak{X}_D.\]
  \end{enumerate}
\end{prop}

\begin{proof}
  Let $\Sigma$ be the set of embeddings $\tau : K\rightarrow L$.  Let
  $U$ be the kernel of the map
  $T\mathfrak{X}_D\rightarrow T\mathfrak{X}_{W^+_{\dR}(D)}$. It
  follows from Lemma \ref{lem:comp_cris} as in \cite[Cor.~3.13]{HMS}
  that the following sequence is exact for any triangulation
  $\mathcal{F}$ of $D$:
  \[ 0\longrightarrow U\longrightarrow
    T\mathfrak{X}_{D,\mathcal{F}[\frac{1}{t}]}^{w_0}\longrightarrow
    T\mathfrak{X}_{W_{\dR}^+(D),W_{\dR}(\mathcal{F}[\frac{1}{t}])}^{w_0}\longrightarrow0.\]
  Therefore, if $\mathcal{S}\subset\Tri(D)$ is any subset, we have the
  following commutative diagram
  \[
    \begin{tikzcd}
      0\ar[d]&0\ar[d]&\\
      \bigoplus_{\mathcal{F}\in\mathcal{S}} U \ar[r, "\sum"] \ar[d] &  U \ar[d,] \ar[r] & 0 \\
    \bigoplus_{\mathcal{F}\in\mathcal{S}} T\mathfrak{X}^{w_0}_{D,\mathcal{F}[\frac{1}{t}]} \ar[r] \ar[d]& T\mathfrak{X}_{D} \ar[d,"W^+_{\dR}"] &  \\
  \bigoplus_{\mathcal{F}\in\mathcal{S}} T\mathfrak{X}_{W^{+}_{\dR}(D),W_{\dR}(\mathcal{F}[\frac{1}{t}])}^{w_0} \ar[d] \ar[r] & T\mathfrak{X}_{W^{+}_{\dR}(D)} &&\\
  0&&
    \end{tikzcd}
  \]
  Thus to prove that the middle horizontal arrow is surjective, it is
  sufficient to prove that the bottom horizontal arrow is
  surjective. As $\mathfrak{X}_{D}^\Box\rightarrow\mathfrak{X}_{D}$ is
  formally smooth, it is sufficient to prove that the map
  \[ \bigoplus_{\mathcal{F}\in\mathcal{S}}
    T\mathfrak{X}_{W^{+}_{\dR}(D),W_{\dR}(\mathcal{F}[\frac{1}{t}])}^{w_0,\Box}
    \longrightarrow T\mathfrak{X}_{W^{+}_{\dR}(D)}^{\Box}\] is
  surjective. Fix $\alpha : (L\otimes_{\QQ_p}K)^n\simeq D_{\dR}(D)$ an isomorphism compatible with Frobenius as in Remark \ref{rema:basealpha}.
Let $\mathcal{F}$ the triangulation of $D$ such that $\alpha^{-1}(D_{\dR}(\mathcal{F}[1/t]))$ is the standard flag of $(L\otimes_{\QQ_p}K)^n$ and let
  $F_\bullet\coloneqq D_{\dR}(\mathcal{F}[1/t])$. We denote
  $x_{\mathcal{F}}\coloneqq(\alpha^{-1}(F_\bullet),0,\alpha^{-1}(\Fil_{\dR}^\bullet))\in
  X_{K/\Qp}(L)$. It follows from \cite[Thm.~3.2.5 \& Cor.~3.5.9]{BHS}
  that the vertical arrows in the following commutative diagram are
  isomorphisms

  \[ \begin{tikzcd}
      \mathfrak{X}_{W^+,F}^{\Box} \ar[r] \ar[d,"\simeq"] & \mathfrak{X}_{W^+}^\Box \ar[d,"\simeq"]\\
      \widehat{(X_{K/\Qp,L})}_{x_{\mathcal{F}}} \ar[r,"\pi_2"]&
      (\widehat{\widetilde{\mathfrak{g}}}_{K/\Qp,L})_{\pi_2(x_{\mathcal{F}})}
    \end{tikzcd}\] ($\pi_2 : \widetilde{\mathfrak{g}}\times_{\mathfrak{g}}\widetilde{\mathfrak{g}}\rightarrow\widetilde{\mathfrak{g}}$ is the second projection
  ).

  Recall that we have a decomposition
  $X_{K/\Qp,L}\simeq \prod_{\tau\in\Sigma}X_\tau$ where
  $X_\tau\simeq L\times_{K,\tau}X$ and
  $\widetilde{\mathfrak{g}}_{K/\Qp,L}\simeq\prod_{\tau\in\Sigma}\widetilde{\mathfrak{g}}_\tau$
  and the map $\pi_2$ is of the form $(\pi_{2,\tau})_{\tau\in\Sigma}$
  with $\pi_{2,\tau}$ the base change of the second projection
  $X\rightarrow\widetilde{\mathfrak{g}}$.  Moreover the irreducible
  component of $X_{K/\Qp,L}$ corresponding to the longest element is
  isomorphic to $\prod_{\tau\in\Sigma}X_{w_0,\tau}$ with $w_0$ the
  longest element of $\mathfrak{S}_n$.

  Therefore we have to prove that the map
  \begin{equation}
    \label{eq:desired_surj_alltau}
    \bigoplus_{\tau\in\Sigma}\bigoplus_{\mathcal{F}\in\mathcal{S}}\widehat{X}_{w_0,\tau,x_{\mathcal{F},\tau}}\longrightarrow\bigoplus_{\tau\in\Sigma}\widehat{\widetilde{\mathfrak{g}}}_{\tau,\pi_{2,L}(x_{\mathcal{F},\tau})}
  \end{equation}
  is surjective at the level of tangent spaces. As the formation of
  tangent spaces commutes with finite products, it is sufficient to
  prove that for a fixed embedding $\tau\in\Sigma$, the
  following map is surjective
  \begin{equation}
    \label{eq:desired_surj}
    \bigoplus_{\mathcal{F}\in\mathcal{S}}\widehat{X}_{w_0,\tau,x_{\mathcal{F},\tau}}\longrightarrow\widehat{\widetilde{\mathfrak{g}}}_{\tau,x_{\mathcal{F},\tau}}.
  \end{equation}
  Let $g_\tau,h_\tau\in \GL_n(L)$ such that $x_{\mathcal{F},\tau}=(g_\tau B(L),0,h_\tau B(L))$. By the choice of $\alpha$, we have $g_\tau = 1$, for all $\tau$. 

  The desired surjectivity
  is equivalent to the following equality
  \[ \im(\sum_{w \in S} T_{(wB(L),0,h_\tau B(L))}X_{w_0,\tau}
    \fleche T_{(0,h_\tau B(L))}\widetilde{\mathfrak{g}}_\tau) =
    T_{(0,h_\tau B(L))}\widetilde{\mathfrak{g}}_\tau\] for a well
  chosen $S=\set{w\in\mathfrak{S}_n \mid
    w\mathcal{F}\in\mathcal{S}}$. By Proposition
  \ref{prop:tangentcompw0ptcrist}, we deduce that, for any $w\in\mathfrak{S}_n$,
  \begin{multline*} T_{(wB(L),0,h_\tau B(L))}X_{w_0,\tau} =
    T_{(wB(L),0,h_\tau B(L))}\kappa^{-1}(T_{w_0}) \\ = T_{wB(L)}G/B
    \oplus (w\mathfrak{b}w^{-1} \cap h_\tau \mathfrak{b}
    h_\tau^{-1})^{gr=w_0}\oplus T_{h_\tau B(L)}G/B.\end{multline*} By
  Lemma \ref{lemmaMargerin} there exists an element
  $w_\tau\in\mathfrak{S}_n$ such that
  \begin{multline*} 
    \im(\sum_{c\in\mathfrak{C}_n} T_{(cw_\tau  B(L),0,h_\tau B(L))}X_{\tau,w_0}
    \fleche T_{(0,h_\tau B(L))}\widetilde{\mathfrak{g}}_\tau) \\ =
    \left(\sum_{c\in\mathfrak{C}_n} (cw_\tau\mathfrak{b}w_\tau^{-1} c^{-1} \cap h_\tau
      \mathfrak{b} h_\tau^{-1})^{gr=w_0}\right)\oplus T_{h_\tau B}G/B = h_\tau
    \mathfrak{b} h_\tau^{-1} \oplus T_{h_\tau B}G/B =
    T_{(0,h_\tau B)}\widetilde{\mathfrak{g}}_\tau.\end{multline*}
  This implies the surjectivity of the map \eqref{eq:desired_surj} when $\mathcal{S}=\set{cw_\tau\mathcal{F} \mid c\in\mathfrak{C}_n}$, and in particular when $\mathcal{S}=\Tri(D)$ (which is independant of $\tau$). This implies the surjectivity of the map \eqref{eq:desired_surj_alltau} and our first statement.
  
  Now we assume that $\mathcal{F}$ is non critical. By non
  criticality of $\mathcal{F}$, $h_\tau\in B(L)w_0$. Let $c\in\mathfrak{C}_n$ and
  let $\sigma\in \mathfrak{S}_n$ be such that
  $(cB(L),0,h_\tau B(L)) \in U_{\sigma}$. We claim that
  $w_0\sigma^{-1}$ is a product of distinct simple reflections. Namely
  it follows from Lemma \ref{lem:relationBruhat} that
  $c^{-1}w_0 \leq \sigma$ so that $w_0 \sigma^{-1}\leq c$ and, as $c$
  is a product of simple reflections, so is $w_0\sigma^{-1}$. This
  proves part \eqref{propraf1}.

  Moreover as $h_\tau\in B(L)w_0$, we can choose $w_\tau=1$ by Lemma
  \ref{lemmaMargerin}. This implies that the surjectivity of
  \eqref{eq:desired_surj} is true with
  $\mathcal{S}=\set{c\mathcal{F} \mid c\in\mathfrak{C}_n}$ for all
  $\tau\in\Sigma$. This proves the surjectivity of
  \eqref{eq:desired_surj_alltau} for this $\mathcal{S}$ and part
  (\ref{propraf2}).
\end{proof}
 
\section{A remark on signs}
\label{sect:signs}

By a theorem of Artin, all elements of order $2$ are conjugate in
$G_{\QQ}$. Let $C_\infty\subset G_{\QQ}$ be their conjugacy class and
let $H\subset G_{\QQ}$ be the closed subgroup generated by
$C_\infty$. There is a unique continuous morphism $\eps :
H\rightarrow\set{\pm1}$ such that $\eps(c)=-1$ for all $c\in
C_\infty$. Let $K$ be a number field. Then $K$ is totally real if and
only if $H\subset G_K$. Let $E$ be a CM field with totally real subfield $F$,
we have $H_1 := \ker \eps\subset G_E$ and $G_F=G_EH$. Let $c\in C_\infty$. We can consider the action of $c$ on $G_E$
by conjugacy, and we have
$G_F=G_E\rtimes\set{1,c}$.  If $\rho$ is a morphism of $G_E$ in some group and $c\in
C_\infty$, we set $\rho^c=\rho(c^{-1}\cdot c)=\rho(c\cdot c)$. The
$c$-conjugacy induces an automorphism $c$ of $G_E^{\ab}$. As
$H_1\subset G_E$, this automorphism does not depend on the choice of $c\in C_\infty$.

\begin{lemma}\label{lem:extension} Let $\chi : G_F\rightarrow \overline{\QQ_p}^\times$ be a continuous character. Then $\chi|_{G_E}=(\chi|_{G_E})^c$. If there exists a continuous character $\psi : G_E\rightarrow \overline{\QQ_p}^\times$ such that $\psi\psi^c=\chi|_{G_E}$, then the element $\chi(c)$ for $c\in C_\infty$ does not depend on the choice of $c\in C_\infty$. Conversely if we assume that $\chi$ is algebraic and that $\chi(c)$ does not depend on the choice of $c\in C_\infty$, then there exists an algebraic continuous character $\psi : G_E\rightarrow \overline{\QQ_p}^\times$ such that $\psi\psi^c=\chi|_{G_E}$. If moreover $\chi$ is unramified outside a finite set $S$ of places of $F$, we can assume that $\psi$ is unramified outside $S$.
\end{lemma}

\begin{proof}
  The first statement is clear as $\chi(G_F)$ is abelian. Assume that $\chi|_{G_E}=\psi\psi^c$ for some $\psi$ and let $c_1$ and $c_2$ be two elements of $C_\infty$. Then
\begin{align*} \chi(c_1c_2^{-1})&=\chi(c_1c_2)=\psi(c_1c^2c_2)\psi(cc_1c_2c)\\
&=\psi(c_1c)\psi(cc_2)\psi(cc_1)\psi(c_2c)=\psi^{-1}(cc_1)\psi(cc_1)\psi(cc_2)\psi^{-1}(cc_2)=1.\end{align*}
  The last statements are then direct consequences of \cite[Lem.~4.1.1 \&
  4.1.4]{CHT}.
\end{proof}

Note that a continuous morphism $\chi : G_E\rightarrow G$ extends to a morphism $G_F\rightarrow G$ if and only if $\chi=\chi^c$. The fact that the extension of $\chi$ to $G_F$ satisfies the assumptions of Lemma \ref{lem:extension} depends only on the restriction of $\chi$ to $G_E$ and is equivalent to the fact that $\chi|_{H_1}$ is trivial.



Let $\chi: G_F\rightarrow \overline{\Qp}^\times$ be a continuous
character. We say that a continuous representation $\rho : G_E
\fleche \GL_n(\overline{\QQ_p})$ is polarized by $\chi$ if
\begin{equation}\label{eq:rhopol} \rho^\vee \simeq \rho^c \otimes
  (\chi|_{G_E}\eps^{n-1}).\end{equation}

If $\rho$ is irreducible
then we can define its sign, with respect to $\chi|_{G_E}\eps^{n-1}$,
$\lambda \in \{\pm 1\}$ as in \cite[1.1]{BCsign}. This is the sign of
the pairing appearing in (\ref{eq:rhopol}).

Fix an isomorphism $\iota : \CC\simeq\overline{\QQ_p}$. Let $\rho : G_E \fleche \GL_n(\overline{\QQ_p})$ be a semi-simple continuous representation. Assume that there exists a cuspidal regular algebraic automorphic representation $\Pi$ of $\GL_n(\mathbb{A}_E)$ such that $\rho$ is strongly associated with $\Pi$ (see Definition \ref{SatakeGLn}) with respect to $\iota$.
We will say that $\Pi$ is \emph{polarized by $\chi$} if the pair $(\Pi,\chi\circ\Art_F)$ is polarized in the sense of \cite[2.1]{pot_aut}, which means $\Pi^\vee\simeq\Pi^c\otimes(\iota^{-1}\circ\chi|_{G_E}\circ\Art_E)$. Then $\rho$ is polarized by $\chi$ and $\rho$ satisfies the properties of \cite[Thm.~2.1]{BLGGTII}.

The following result is essentially \cite[Thm.~1.2]{BCsign}.

\begin{theor}[Bellaïche-Chenevier]\label{thr:BCsign}
Assume that $\Pi$ is conjugate self-dual and regular algebraic and that $\chi|_{H_1}$ is trivial. Let $\psi : G_E \fleche \overline{\QQ_p}$ be a continuous character such that $\psi\psi^c=\chi|_{G_E}$. Then every irreducible constituent $r$ of $\rho\otimes\psi^{-1}$ satisfying (\ref{eq:rhopol}) has sign $\lambda = +1$ with respect to $\chi|_{G_E}\eps^{n-1}$.
\end{theor}

\begin{proof}
If $\psi = 1$ this is \cite{BCsign} Theorem 1.2. We deduce the general case from \cite{BCsign} Lemma 2.1.
\end{proof}

\section{Deformation spaces}
\label{sect:defspaces}

Denote by $k$ a topological field and $\mathcal{O}$ a complete noetherian local $\ZZ_p$-algebra with residue field $k$.

Fix $E$ a totally imaginary CM number field with
maximal totally real subfield $F$ and
fix $S$ a finite set of finite places of $E$ containing the places
above $p$, and the ramified places of $E$. Denote
\[ G_{E,S} = \Gal(E_S/E),\]
the Galois group of the maximal unramified outside $S$ extension $E_S/E$.

Suppose given \[\overline{\rho} : G_{E,S} \fleche \GL_n(k),\]
a continuous \textit{semi-simple} Galois representation. From now on we choose $c \in G_F \backslash G_E$  such that 
$c^2 = 1$ a complex conjugation, denoting similarly its image in $\Gal(E/F)$. As $\mathcal{O}$ is a $\ZZ_p$-algebra, we have a continuous ring homomorphism $\ZZ_p \fleche k$ and we denote $\overline\eps$ its composition with the cyclotomic character. We also assume that $\overline{\rho}$ is polarized by $\overline\chi$, i.e.
\[ \overline{\rho}^\vee \simeq \overline{\rho}^c \otimes ({\overline\eps}^{n-1}\overline\chi),\]
for some continuous character $\overline\chi : G_{E,S} \fleche k^\times$ satisfying $\overline\chi^c = \overline\chi$. We denote $\mathcal C_{\mathcal O}$, or $\mathcal C$ if the context is clear, the category of artinian local $\mathcal{O}$-algebras with residue field $k$.

\begin{hypothese}\label{hyp:twocases}
From now on in this section, we suppose that we are in either one of the two situations:
\begin{itemize}
\item $k \subset \overline\FP$ with the discrete topology, $\mathcal{O}$ a finite totally ramified extension of $W(k)$
\item $k \subset \overline{\QQ_p}$ a finite extension of $\QQ_p$ with
  its $p$-adic topology and in this case we set $\mathcal{O} = k$.
\end{itemize}
In the second case $\overline\eps = \eps$ is just the $\ZZ_p^\times$-valued cyclotomic character.
\end{hypothese}

Denote by $\tr \overline{\rho}$ the Determinant (in the sense of
Chenevier \cite[Définition 1.5]{Chedet}) of $\overline{\rho}$. As
$\overline{\rho}$ is semi-simple it is completely determined by $\tr
\overline \rho$ (by \cite[Cor.~2.13]{Chedet}). We fix once and for all a
continuous character $\chi : G_{E,S} \fleche \mathcal{O}^\times$
lifting $\overline\chi$ and such that $\chi^c = \chi$. We will need sometimes to assume the following hypothesis on $\chi$.
\begin{hypothese}
\label{hypchi} The character $\chi$ is algebraic at $p$ and $\chi|_{H_1}$ is trivial.
\end{hypothese}

\begin{defin}
\label{defin:psdef}
We denote by $\mathcal{F}^{\chi-\pol}_{\overline{\rho}}$ the functor that associates to any object $A$ of $\mathcal C$ the set of continuous determinants $D$ lifting $\tr \overline{\rho}$
such that $D^\vee = D^c \otimes \chi \eps^{n-1}$. It is pro-representable by a ring $R^{\chi-\pol}_{\overline{\rho}}$ (\cite[Prop.~3.3]{Chedet}\footnote{for $R^{}_{\overline{\rho}}$, and then $R^{\chi-\pol}_{\overline{\rho}} = R^{}_{\overline{\rho}}/I$ with $I = (D^{univ,\vee}(g)-D^{univ,c}(g)\chi\eps^{n-1}(g), g \in G)$}). We denote the associated formal scheme $\mathfrak X_{\overline{\rho}}^{\chi-\pol} = \Spf(R_{\overline{\rho}}^{\chi-\pol})$. When $k$ is a finite field of characteristic $p$, we denote the generic fibre of $R^{\chi-\pol}_{\overline{\rho}}$ by 
\[ \mathcal{X}_{\overline{\rho}}^{\chi-\pol} := \Spf(R^{\chi-\pol}_{\overline{\rho}})^{rig}.\]
If $\overline{\rho}$ is absolutely irreducible, this coincides with the rigid fiber of the polarized-by-$\chi$ deformation space of $\overline{\rho}$.
\end{defin}

Recall that we fixed an isomorphism $\iota : \CC\xrightarrow{\sim}\overline{\QQ_p}$.

\begin{defin}
A point of $x\in\mathcal X_{\overline\rho}^{\chi-\pol}(\overline{\QQ_p})$ is $GL_n$-\textit{modular}\footnote{Compare with Definition \ref{defin:modular}.} if there exists a polarized by $\chi$ automorphic representation $\Pi$ of $\GL_n(\mathbb A_E)$ such that $\rho_{\Pi,\iota}\simeq\rho_x$.
\end{defin}

It follows from Lemma \ref{lem:extension} that there exists a finite extension $\mathcal{O}'/\mathcal{O}$ of discrete valuation rings and a continuous algebraic character $\psi_0 : G_{E,S} \fleche (\mathcal{O}')^\times$.
We fix such a character $\psi_0$.

We start with the following reduction to lighten slightly the notations in the rest of the text.

\begin{lemma}\label{lemma:chi=1}
Assume $\chi$ is as before, and $\overline \rho$ satisfies $\overline\rho^\vee \simeq \overline \rho^c \otimes \eps^{n-1}$. Then there is an isomorphism, which identifies modular points,
\[ \mathcal X_{\overline\rho}^{1-\pol}\otimes_{\mathcal{O}[1/p]}\mathcal{O}'[1/p] \overset{\sim}{\fleche} \mathcal X_{\overline\rho\psi_0^{-1}}^{\chi-\pol}\otimes_{\mathcal{O}[1/p]}\mathcal{O}'[1/p].\]
In particular it is enough to prove theorem \ref{thr:main} for $\chi=1$.
\end{lemma}

\begin{proof}
The character $\psi_0\circ\Art_E$ is automorphic as $\psi_0 : G_{E,S} \fleche (\mathcal{O}')^\times$ is algebraic. Moreover, the isomorphism is given by
\[ \rho \longmapsto \rho\psi_0^{-1}.\]
This is obviously an isomorphism, and because $\psi_0\circ\Art_E$ is automorphic it identifies ($\GL_n$-)modular points on both sides.
\end{proof}

Our goal is to understand the geometry of the set of modular points in $\mathcal{X}_{\overline{\rho}}^{\chi-\pol}$ when $k \subset \overline\FP$, $\mathcal{O} = \mathcal{O}_K$, $K/\QQ_p$ finite. However we will need some result concerning the situation where $k=\mathcal{O}$ is a finite extension of $\QQ_p$ but for completeness we prove them in the two situations of Hypothesis \ref{hyp:twocases}.
Following \cite{CHT} we introduce
\[ \mathcal G_n = (\GL_n \times \GL_1)\rtimes \Gal(E/F),\]
where $c \in \Gal(E/F)$ acts on $(g,x) \in \GL_n \times \GL_1$ via $(x {^t}g^{-1},x)$. We denote $\nu$ the homomorphism $\mathcal G_n \fleche \GL_1$ sending $(g,x)$ to $x$ and $c$ to $-1$. 

Finally we assume until the end of this section that $\rhobar$ is absolutely irreducible and denote $\lambda$ its sign (as defined in section \ref{sect:signs}). Then we can extend $\chi$, which satisfies $\chi^c = \chi$, to $G_F \simeq G_E \rtimes \Gal(E/F)$ by setting $\chi(c) := (-1)^n\lambda$ and we extend compatibly $\overline{\chi}$ to $G_F$, so that $\mu := \chi \eps^{n-1}$ satisfies $\mu(c) = -\lambda$. By \cite{CHT} Lemma 2.1.1 we can thus extend $\overline{\rho}$ to a continuous homomorphism
\[ \overline{r} : G_{F,S} \fleche \mathcal G_n(k),\]
such that $c \in G_F$ is sent to $c \in \Gal(E/F)$ via $\overline{r}$ and projection and $\nu \circ \overline{r} = \overline{\chi}^{-1}\overline{\eps}^{1-n}$ (as extended before to $G_F$).

\begin{defin}
Let $\Def_{\overline{r}}^{\chi}$ be the functor that associates to any object $R$ of $\mathcal{C}$  the set $\Def_{\overline{r}}^{\chi}(R)$ of lifts $r : G_{F,S}\rightarrow\mathcal{G}_n(R)$ of $\overline{r}$ such that $\nu\circ r=\overline{\chi}^{-1}\overline{\eps}^{1-n}$ considered up to $1+\mathfrak{m}_RM_n(R)$-conjugation. As in \cite[Prop.~2.2.9]{CHT}\footnote{There the field $k$ is finite, but we can check that everything carries over in our setting, as already remarked in \cite{KisinFM}}, this functor is pro-represented by a local complete noetherian $\mathcal{O}$-algebra $R_{\overline{r}}^\chi$. When $k$ is a finite field of characteristic $p$, we denote by $\mathcal{X}^{\chi}_{\overline{r}}$ the generic fiber of the formal scheme $\mathfrak{X}^{\chi}_{\overline{r}} = \Spf(R_{\overline{r}}^\chi)$.
\end{defin}

In the following, all cohomology groups are continuous cohomology groups.

\begin{prop}
\label{prop:dimdefr}
Assume that $\overline{\rho}$ is absolutely irreducible, $\car k \neq 2$, $\chi(c') = (-1)^n$ for all complex conjugacy $c'$ and, if $k$ is of characteristic $p$, $R^{\chi}_{\overline{r}}[1/p]\neq0$. Then
\[\dim(R^{\chi}_{\overline{r}}[1/p]) \geq \dim_k H^1(G_{F,S},\ad(\overline{r})) -  \dim_k H^2(G_{F,S},\ad(\overline{r})) = \frac{n(n+1)}{2}[F:\QQ].\]
Moreover the topological $\mathcal{O}[1/p]$-algebra $R^{\chi}_{\overline{r}}[1/p]$ is formally smooth of relative dimension $\frac{n(n+1)}{2}[F:\QQ]$ if $H^2(G_{F,S},\ad(\overline{r})) = 0$.
\end{prop}

\begin{proof} This appeared already in \cite{CHT,All1}, let us give the argument.
As $\overline{\rho}$ is absolutely irreducible, $\ad(\overline{r})^{G_F} = H^0(G_F,\ad(\overline{r}))= 0$ by \cite[Lem.~2.1.7(3)]{CHT}. 
For each place $v | \infty$, we have (\cite[Lem.~2.1.3]{CHT})
\[ \dim_k H^0(G_{F_v},\ad(\overline{r})) =  \frac{n(n+\overline{\chi}\overline{\eps}^{n-1}(c_v))}{2}= \frac{n(n-1)}{2}\]
(where $c_v$ is the complex conjugation in $F_v$).

Now the equality 
\[\dim_k H^1(G_{F,S},\ad(\overline{r})) -  \dim_k H^2(G_{F,S},\ad(\overline{r})) = \frac{n(n+1)}{2}[F:\QQ]\]
follows from \cite[Lem.~2.3.3]{CHT}\footnote{Note that in \cite[2.3]{CHT}, it is supposed that the places of $S$ are split in $E$ but this is not used in their Lemma 2.3.3.} when $k$ is a finite field and from  \cite[Lem.~1.3.4]{All1} when $k$ is a finite extension of $\QQ_p$.

When $k$ is a
finite extension of $\QQ_p$, the result follows from the
analogue of \cite[Cor.~2.2.12]{CHT} (but without the $+1$ since here
$\mathcal{O}=k$).
When $k$ is a finite field, it follows from \cite[Cor.~2.2.12]{CHT} that
\[ \dim(R_{\overline{r}}^{\chi})\geq 1+\frac{n(n+1)}{2}[F:\QQ].\]
Let $x\in\Spec(R_{\overline{r}}^{\chi}[1/p])$ be a closed point, $\mathfrak{p}_x$ the corresponding prime ideal and $r_x : G_F\rightarrow\mathcal{G}_n(k(x))$ the corresponding representation. It follows from \cite[Prop.~1.3.11(1)]{All1} that the localization-completion of $R_{\overline{r}}^{\chi}$ at $\mathfrak{p}_x$ is isomorphic to $R_{r_x}^\chi$. It follows that
\[ \dim(R_{\overline{r}}^{\chi}[1/p]) \geq \dim(R_{r_x}^\chi)\geq \frac{n(n+1)}{2}[F:\QQ]\]
using the case where $k$ has characteristic $0$.

The assertion concerning the formal smoothness follows from \cite[Cor.~2.2.12]{CHT} and \cite[Prop.~1.3.11]{All1}
\end{proof}

\begin{rema}
  The hypothesis on the sign of $\chi$ in Proposition \ref{prop:dimdefr} will be satisfied in the rest of the text where we choose $\chi = \psi_0\psi_0^c$ for some $\psi_0 : G_E\rightarrow\overline{\QQ_p}^\times$, actually because of the sign theorem \ref{thr:BCsign}. Indeed, by \cite{CHT} Lemma 2.1.1 we need to extend $\mu = \chi \eps^{n-1}$ to $G_{F,S}$, for an absolutely irreducible $\rho$, by sending $c$ to $-\lambda$. Thus the previous hypothesis is equivalent to $\lambda = 1$.
\end{rema}

\begin{prop}
\label{prop:defrrho}
Denote $\overline{\rho}$ as before. Suppose it is absolutely irreducible, and denote $\overline{r}$ the chosen $\mathcal G_n$-extension as before. Suppose $\Char(k) \neq 2$. Then the natural map
\[ R_{\overline{r}}^{\chi} \fleche R_{\overline{\rho}}^{\chi-\pol} \]
is an isomorphism.
\end{prop}

\begin{proof}
This is also \cite[Prop.~2.2.3]{All2}. Denote $\rho$, $\rho'$ resp. $R,R'$ valued points of $\mathcal{F}^{\chi-\pol}_{\overline{\rho}}$, with $R' \twoheadrightarrow R$, and $\rho' \otimes R = \rho$.
Suppose we have a fixed pairing $< , > : \rho \otimes \rho^c \fleche \chi^{-1} \eps^{1-n}$ inducing $r : G_{E,S} \fleche \mathcal G_n(R)$ by \cite{CHT} Lemma 2.1.1. Choose any pairing fixing $<,>'_0$ for $\rho'$. Then reducing to $R$ this gives a pairing for $\rho$, but as $\overline{\rho}$ is absolutely irreducible, $\rho$ is also and thus
there is only one pairing up to scalar for $\rho$, i.e.~ $<,>'_0 \otimes R = \alpha <,>$ for some $\alpha$ in $R^\times$. Choose a lift $\beta$ of $\alpha^{-1}$, then set $<,>' := \beta <,>'_0$, then $<,>'$ reduces to $<,>$ and to $(\rho', \chi, <,>')$ is associated by \cite{CHT} Lemma 2.1.1 an $r' : G_{E,S} \fleche \mathcal G_n(R')$, reducing to $r$. Let $r''$ another point over $R'$, inducing $\rho'$ and reducing to $r$ , then it corresponds to $\gamma <.,.>'$ with $\gamma \equiv 1 \pmod {m_{R'}}$, thus writing $\gamma = 1 + m$ with $m \in m_{R'}$ we have $\gamma = (1 + \frac{1}{2}m)^2 \pmod {m_{R'}^2}$ and as $R'$ is artinian, a direct induction shows that $\gamma$ is a square in $R'$, thus $r' = r''$.\footnote{\cite{CHT} is written over a field, but their proof applies here}

The same argument with $\overline{r}$ for $r$ and $r$ for $r'$ shows that we can actually choose $r$ inside $\mathfrak{X}_{\overline{r}}^{\chi}$ (and thus automatically for any $r'$ above) and thus proves etaleness, and surjectivity. As the map is an isomorphism in special fiber, this is an isomorphism.
\end{proof}

By \cite[2.1]{CHT}, the adjoint action $ad(r)$ for some $r : G_F \fleche \mathcal G_n(k)$ coincides with the extension of $\ad(\rho)$ to $G_F$ as in \cite{NT} (when $\chi = 1$), when $\rho$ corresponds to $r$ by \cite[Lemma 2.1.1]{CHT}. Thus we will often denote abusively $H^1(G_F,\ad(\rho))$ instead of $H^1(G_F,\ad(r))$.

\section{Eigenvarieties and the infinite fern}
\label{se:eigenvarieties}

There are at least two ways to define Eigenvarieties as explained in \cite{BC2} which, at least, in our case of interest end up to be the same. 
By the Lemma \ref{lemma:chi=1}, we can assume $\chi = 1$.

In this section $(\mathcal{O},k)$ are as in the first case of Hypothesis \ref{hyp:twocases}. We fix $\overline \rho : G_E\rightarrow\GL_n(k)$ a
semi-simple polarised-by-$1$ continuous representation, i.e.~$\rhobar^\vee\simeq\rhobar^c\otimes\overline{\eps}^{n-1}$. Let
$\mathfrak{X}_{\overline\rho}^{\pol} := \mathfrak X_{\overline\rho}^{1-\pol}$ be its polarised
pseudodeformation space. Let $G$ be the quasi-split similitude unitary
group of dimension $n$ over $\QQ$ whose $R$-points, for $R$ a
$\QQ$-algebra, are:
\[ G(R)=\set{(g,\nu)\in \GL_n(R\otimes_{\QQ}E)\times R^\times \mid
    {}^tc(g)Jg=\nu J}\]
where $J$ is the $n\times n$
matrix $\left(
  \begin{smallmatrix}
    0 & \cdots & 1 \\ \vdots & \iddots& \vdots \\ 1 &\cdots & 0
  \end{smallmatrix}\right)$. Moreover let $G_1$ be the kernel of the
morphism $\nu : G\rightarrow \mathbb{G}_m$. As $p$ is unramified in $E$, we also fix a reductive model $G_{\ZZ_p}$ over $\ZZ_p$ of $G$ defined by the similar formula (replacing $R \otimes_\QQ E$ by $R \otimes_{\ZZ} \mathcal O_E$).

We fix embeddings $\overline{\QQ}\hookrightarrow\CC$ and $\overline{\QQ} \hookrightarrow \overline{\QQ_p}$ that we use to
identify the embeddings of $E$ (resp.~$F$) in $\overline{\QQ}_p$ with
the set $\Sigma_E$ (resp.~$\Sigma_F$) of embeddings of $E$ (resp.~$F$)
in $\CC$. We fix a CM type $\Phi$ for $E$. For $\sigma\in\Sigma_E$, we use the notation
$\overline{\sigma}=\sigma\circ c$. If $\tau\in\Sigma_F$,
let $\sigma_\tau\in\Sigma_E$ be the unique element such that
$\tau=\sigma_\tau|_{F}$ and $\sigma_\tau\in\Phi$. 

We fix a PEL datum {$(E,c,V,\scalar{\cdot,\cdot},h)$} for the
previous group $G$ and denote its signature
$(p_{\sigma_\tau},q_{\sigma_\tau})_{\tau \in \Sigma_F}$ at
infinity\footnote{Because $G$ is quasi-split these integers are
  explicit, but we keep the slightly general notation as we think it
  is a bit clearer.}. In particular we have that
$p_{\sigma_\tau}+q_{\sigma_\tau} =n$ doesn't depend on $\tau$. We
define more generally $(p_\sigma)_{\sigma \in \Sigma_E}$ by
$p_{\sigma} = p_{\sigma_\tau}$ if $\sigma = \sigma_\tau$ and
$p_{\sigma} = q_{\sigma_\tau} = n - p_{\sigma_\tau}$ if
$\sigma = \overline{\sigma_\tau}$. We also sometimes abuse notation
and write $p_\tau,q_\tau$ for $p_{\sigma_\tau},q_{\sigma_\tau}$. Let
$(G,h)$ be a Shimura datum associated to $G$. We let
$\mathcal{S} = (S_K)_K$ be the tower of Shimura varieties for $(G,h)$
(\cite{Lan} or \cite{Her4} which we will use later). Let
$\mu : \mathbb{G}_m\rightarrow G_\CC$ be the cocharacter associated to
$h$ and let $P$ be the parabolic subgroup fixing the Hodge filtration
associated to $\mu$. Let $M$ be the Levi subgroup of $P$ fixing the
Hodge decomposition of $V_\CC$ (defined over some extension $L$ of the
reflex field). Let $\mathfrak{p}$ be the Lie algebra of $P$ and let
$K_\infty$ be the centralizer of $h(i)$ in $G_1(\RR)$.

\begin{defin}
\label{defin:modular}
We say that a polarised-by-$1$ representation
\[ \rho : G_E \fleche \GL_n(\overline{\QQ_p}),\]
is \emph{modular} if $\rho$ is (strongly essentially) associated to a cuspidal algebraic automorphic
representation $\pi$ for $G$ as in Definition \ref{Satake}.
We say that $\rho$ is \emph{holomorphically modular} if its Hecke
eigensystem appears in the space of cuspidal sections of some coherent
automorphic sheaf on some Shimura variety of $\mathcal{S}$. This is
equivalent to the fact that $\pi$ is cuspidal and holomorphic at infinity; i.e.~$H^0(\mathfrak{p},K_\infty,\pi_\infty \otimes \sigma)\neq0$ for some algebraic representation 
$\sigma$ of $K_\infty$ \footnote{These hypothesis are here to assure a concrete (= computable) way to verify if our $\pi$ "appears" in an Eigenvariety. We could introduce the notion of \textit{$p$-adically modular} for which we ask for a Hecke eigensystem appearing in the considered Eigenvariety $\mathcal{E}$ whose associated trace is $\tr \rho$. It is enough to assume $p$-adic modularity for $\overline\rho$ to get Theorem \ref{thr:main}} by \cite{Har} Proposition 5.4.2.

We say that $\overline{\rho}$ is modular if it admits a lift $\rho$
which is modular. We say that $\overline\rho$ is \emph{conveniently
  modular} if it has a lift $\rho$ associated to a cuspidal
automorphic representation $\pi$ which can be chosen unramified at $p$
and outside $S$ and its Hecke eigensystem appears in $i$-th interior
coherent cohomology group on some Shimura variety of $\mathcal S$, with values in some coherent automorphic sheaf, for some $i \geq 0$.

If $K^p$ is a compact open open subgroup of $G(\mathbb{A}^{p,\infty})$, we say that $\overline{\rho}$ is \emph{conveniently modular of tame level $K^p$} if $\pi$ can be moreover chosen such that $\pi^{K^p}\neq0$.
\end{defin}

\begin{hypothese}
\label{hypmodular}
For the rest of the article, we assume that every $v | p$ in $F$ is unramified, and splits in $E$. Moreover we assume that $\overline{\rho}$ is conveniently modular.
\end{hypothese}

 In particular, if $v$ is a place of $F$
dividing $p$, among the two places $w,\overline w$ of $E$ above $v$ only one, say $w$, corresponds to an element
of $\Phi$. We fix this choice, which allows us to identify $E_w$ with $F_v$. Choose a sufficiently large $p$-adic field $L$ such that $M$ and $P$ are defined over $L$ and $L$ splits $E$, i.e.~ $E \otimes_{\QQ_p} L = \prod_{w | p \in E} L$.
Let $\mathcal T$ be the rigid space over $L$ given by $\prod_{v | p} \Hom((F_v^\times)^n,\mathbb G_m)$, and $\mathcal W = \prod_{v | p} \Hom((\mathcal{O}_{F_v}^\times)^n,\mathbb G_m)$ the weight space. There is thus a restriction map
\[ \mathcal T \fleche \mathcal W.\]

From now on we fix a finite field $k$ of characteristic $p$ and a continuous semisimple representation $\rhobar : G_E\rightarrow\GL_n(k)$ which polarized by $\overline{\eps}^{n-1}$ and is conveniently modular. We fix a tame level outside of $p$, $K^p$, which is hyperspecial
 outside $S$ and deep enough so that $\overline{\rho}$ is conveniently modular of tame level $K^p$. 

Let $\mathcal Z'_{K^p} \subset \mathfrak{X}_{\overline{\rho}}^{\chi-\pol}(\overline{\QQ_p}) \times \mathcal T(\overline{\QQ_p})$ the set of pairs $(D,\delta)$ where $D$ is the determinant associated to 
(the Galois representation of) a cuspidal, regular, algebraic, {unramified at $p$} automorphic form $\pi$ for $G$ appearing in degree 0 coherent cohomology \footnote{this means that $H^0(\mathfrak{p},K_\infty,\pi_\infty \otimes V) \neq 0$ for a finite dimensionnal representation of $K_\infty$. In particular $D$ is holomorphically modular.}  by Corollary \ref{cor:A5}, of level $K^p$ outside $p$, of Hodge-Tate weights $k_{v,\tau,1} > k_{v,\tau,2} > \dots > k_{v,\tau,n}$\footnote{We choose the convention for which the cyclotomic character has Hodge-Tate weight +1}  for each $v|p$ in $F$,$\tau$, and $\delta$ such that for all $v,i$, $\delta_{v,\tau,i}$ coincides on $\mathcal{O}_{F_v}^\times$ with $\prod_\tau \tau^{k_{v,\tau,i}}$ and
sends $p$ to $\phi_{v,i}$, where $\phi_{v,1},\dots,\phi_{v,n}$ is an admissible refinement for $\pi_v$ (and obviously such that $D$ lifts $\overline\rho$).\footnote{for these local data at $p$ we have used the implicit choice of $w | v$}

\begin{rema}
By \cite{Box} Theorem D, \cite{PS4} or \cite{GK} Theorem I.3.1, we have that, under the hypothesis \ref{hypmodular}, $\mathcal Z_{K^p}'$ is non empty, i.e.~ we can choose a lift of $\overline{\rho}$ that is holomorphically modular.
\end{rema}

\begin{defin}
The Eigenvariety for $G$, $\overline{\rho}$, $\chi = 1$ and $K^p$  is the Zariski closure 
\[ \mathcal{E}_{K^p}(\overline{\rho}) \subset \mathcal{X}_{\overline{\rho}}^{\pol} \times \mathcal T,\]
of $\mathcal Z'_{K^p}$. The infinite fern $\mathcal{F}_{K^p}(\overline{\rho})$ is the image of $\mathcal{E}_{K^p}(\overline{\rho})$ by the first projection.
\end{defin}
As $G$ is a unitary similitude group with similitude in $\QQ$, thus giving rise to a PEL Shimura datum, and $p$ is unramified in $E$, we also constructed in \cite{Her4} an Eigenvariety for $G$, for any type $K^p$ outside $p$. Actually these two constructions compare, and allow us to deduce the following proposition.

\begin{rema}
Actually we could take $G$ any similitude unitary group with similitude factor in $\QQ$ instead of the quasi-split one. Indeed, as long as $p$ is unramified for $G$ the construction of \cite{Her4} applies and we get the following proposition. In particular, if we have a result analogous to \cite{Her4} for ramified primes (i.e.~ for primes  $v|p$ in $F$ which are ramified, but still assuming $v = w\overline w$ in $E$) then all the methods of this article applies (see \cite{BP2}). For the moment, we still need our $p$-adic group to be (a product of) $GL_n$ to use results on the trianguline variety, but we hope to come back on this question in the future.
\end{rema}

\begin{prop}
\label{prop:geoHecke}
The rigid space $\mathcal E_{K^p}(\overline\rho)$ is equidimensional of dimension $n[F:\QQ]$. The map
\[ h : \mathcal E_{K^p}(\overline\rho) \fleche \mathcal W,\]
is locally, on the goal and the source, finite. In particular the image of any irreducible component of $\mathcal E_{K^p}(\overline\rho)$ is open in $\mathcal W$. Moreover, for all $C > 0$, if $\mathcal Z_C \subset \mathcal Z_{K^p}'$ consists of classical points, crystalline at $p$, which are moreover $C$-very regular (i.e.~ its Hodge-Tate weights satisfies $k_{v,\tau,i} > k_{v,\tau,i+1} + C$ for all $v,\tau,i$), then $\mathcal Z_C$ is Zariski dense in $\mathcal E_{K^p}(\overline\rho)$ and accumulation at every point of $\mathcal{Z}_{K^p}'$.

\end{prop}

We need to introduce a few notations. Let $T$ be the diagonal torus of $G_{\ZZ_p}$. Its group of $\Qp$-points has the following description,
\[T(\QQ_p) = \left\{ 
\left(
\begin{array}{ccc}
 a_{1,w} &   &   \\
  & \ddots  &   \\
  &   &   a_{n,w}
\end{array}
\right)
\in
\prod_{w | p \text{ in} E} \GL_n(E_w), \exists r \in \QQ_p^\times, a_{i,w}a_{n-i+1,\overline w} = r, \forall i,w\right\}.\]
Denote by $T^1$ the subtorus with trivial similitude character
(i.e.~ $r = 1$). We identify $\mathcal T$ with the space
of characters of $T^1(\Qp)$ using the isomorphism
$(F_v^\times)^n\simeq T^1(\QQ_p)$ sending $(a_{1,v},\dots,a_{n,v})$ to
the diagonal matrix of $\GL_n(E_w)$ with diagonal
$(a_{1,v},\dots,a_{n,v})$, via the identification $E_w \simeq F_v$
where $w\mid v$ and $w \in \Phi$. $T(\QQ_p)$ (resp. $T^1(\QQ_p)$) can be identified also with a subgroup of the $L$-points of the torus (resp. subtorus of $r=1$ elements) of $M \simeq \mathbb G_m \times \prod_{v | p \text{ in } F} \prod_{\sigma_\tau \in \Hom(F_v,\overline{\QQ_p})} \GL_{p_{\sigma_\tau}}\times \GL_{q_{\sigma_\tau}}$, using,
\[ \left(
\begin{array}{ccc}
 a_{w,1} &   &   \\
  & \ddots  &   \\
  &   &   a_{w,n}
\end{array}
\right)_w \longmapsto \left(r, \left(\left(
\begin{array}{ccc}
\tau(a_{w,p_{\sigma_\tau}})^{-1} &   &   \\
  & \ddots  &   \\
  &   & \tau(a_{w,1})^{-1}  
\end{array}
\right), \left(
\begin{array}{ccc}
\overline\tau(a_{\overline w,q_{\sigma_\tau}})^{-1}&   &   \\
  & \ddots  &   \\
  &   & \overline\tau(a_{\overline w,1})^{-1}
\end{array}
\right)\right)_\tau\right)
\]

\begin{defin}
Let
$\kappa=(k_{\sigma,i})_{\substack{\sigma\in\Sigma_E \\ 1\leq i\leq
    p_\sigma}}\in\ZZ^{n[F:\QQ]}$. We say that a character
$\chi\in\mathcal{W}(\CC_p)$ is \emph{algebraic of coherent weight
  $\kappa$} if
  \begin{equation}
  \label{eq:weightkappa}\forall z=(z_{v,i})\in\prod_{v\mid p}(\mathcal{O}_{F_v}^\times)^n, \quad
    \chi(z)=\prod_{v\mid p}\prod_{i=1}^{p_\tau=p_{\sigma_\tau}}\sigma_\tau(z_{v,i})^{k_{\sigma_\tau,i}}\prod_{i=1}^{q_\tau=p_{\overline{\sigma_\tau}}}\sigma_\tau(z_{n+1-i})^{-k_{\overline{\sigma_\tau},i}}.\end{equation} We say that an algebraic character of coherent weight
  $\kappa$ is \emph{$M$-dominant} if $k_{\sigma,i}\geq k_{\sigma,i+1}$ for
 $\sigma\in\Sigma_E$ and $1\leq i\leq p_\sigma-1$. 
 \end{defin} This corresponds to the choice of the upper triangular Borel for $M$, in the sense that if we have character $\kappa$ of $M$ which is dominant for the upper Borel of $M$, then its restriction to $T^1(\QQ_p)$ via the previous embedding gives a \textit{M-dominant} $\chi$ in the previous sense. Suppose $\chi$ is algebraic for some coherent weight $\kappa$. For $h = (h_{\tau,i})_{\tau \in \Sigma_F, 1 \leq i \leq n} \in \ZZ^{n[F:\QQ]}$, if\footnote{This just means that $k_{\sigma_\tau,i} = h_{\tau,i}, i \leq p_\tau$ and $-k_{\overline{\sigma_\tau},i} = h_{\tau,n+1-i}$ for $i > p_\tau$.}
  \[ \chi(z) = \prod_{\tau \in \Sigma_F} \sigma_\tau(z_i)^{h_{\tau,i}},\]
then we say that $\chi$ is of \emph{infinitesimal weight} $h$. We say that such a $\chi$ is \emph{dominant} (or $G$-\emph{dominant}) if $h_{\tau,1} \geq h_{\tau,2} \geq \dots\geq h_{\tau,n}$ for all $\tau$.

\begin{proof}
Let $\mathcal E'$ together with a locally finite map $w : \mathcal E' \fleche \mathcal W$ the Eigenvariety for $G$ of tame level $K^p$ constructed in \cite{Her4}.
It is an equidimensionnal rigid space of dimension $n[F:\QQ]$. 

Let $S$ be a finite set of places of $F$ containing the places
dividing $p$ and the places where $K^p$ is not hyperspecial. For
$v\notin S$, let $\mathcal{H}_v$ be the spherical Hecke algebra
$\ZZ[G(F_v)/\!/K^p_v]$ of $G$ and let
$\mathcal{H}^S=\bigotimes_{v\notin S}\mathcal{H}^v$. For $v|p$, let
$\mathcal A_v$ be the (commutative) $\ZZ$-algebra generated by $T_n^-(F_v)$ and their inverses
with $T_n$ the diagonal torus of $\GL_n$, and $T_n^{-}(F_v)$ the subgroup of matrices $\Diag(p^{a_1},\dots,p^{a_n})$ with $a_i \geq a_{i+1}$. Let $\mathcal A(p) = \bigotimes_{v | p} \mathcal A(v)$ be the Atkin-Lehner algebra. It follows from \cite[\S7]{Her4}\footnote{There $\mathcal{H}^S$ is denoted $\mathcal H_{K^p}$. See also remark 7.12 of \cite{Her4}} that
there exists homomorphism
$\lambda : \mathcal{H}^S\rightarrow\Gamma(\mathcal{E}',\mathcal{O}_{\mathcal{E}'}^+)$ and
$\mathcal{A}(p)\rightarrow\Gamma(\mathcal{E}',\mathcal{O}_{\mathcal{E}'})$, sending
$\Diag(\underbrace{1,\dots,1}_{i \text{ times}},p^{-1},\dots,p^{-1})_v$ to a Hecke operator $U_{v,i}$ at $v$, such that,
if $z\in\mathcal{E}'(\CC_p)$, the evaluation of these morphisms at
$z$ induces a non-zero eigenspace in $H^0(S_{K}^{tor}(v),\omega^{w(z),\dagger}(-D))$, with $K = K^pI$, and $I$ a Iwahori subgroup at $p$,
(which is a space of overconvergent cuspidal forms, defined in \cite{Her4} Definition 6.12). Moreover if
$\kappa\in\mathcal{W}(\CC_p)$ is an algebraic character of
$M$-dominant coherent weight, then the action of $\mathcal{H}^S$
preserves the subspace of classical forms $H^0(S_K^{tor},\omega^{\kappa}(-D))$ and coincides with
the ``usual'' action of $\mathcal{H}^S$ on $H^0(S_K^{tor},\omega^\kappa(-D))$.

We remind now that $\mathcal{E}'$ contains an accumulation and Zariski
dense subspace of automorphic points that we will call \textit{very regular small slope classical points}.

Let $\mathcal{Z}\subset\mathcal{E}'(\CC_p)$ be the set of points $z$ satisfying \cite{Her4} Proposition 8.2, a slightly stronger form of Theorem 8.3, namely
\begin{equation}
\label{eq:bij} \max(n_\tau + v_p(\alpha_\tau),0) < \inf_\tau
(k_{\sigma_\tau,p_\tau} + k_{\overline{\sigma_\tau},q_\tau}), \quad \forall \tau\in\Sigma_F\end{equation}
where $w(z)$ is (thus) a $G$-dominant algebraic character of coherent weight
$(k_{\sigma,i})_{\substack{\sigma\in\Sigma_E \\ 1\leq i\leq p_\sigma}}$ and
$\alpha_\tau$ is the eigenvalue for the operator $U_{v,\min(p_\tau,q_\tau)}$ and $n_\tau$ is a normalisation constant depending only on $(p_\tau,q_\tau)_\tau$, asking that $w(z)$ is moreover \textit{far from
  the walls} as in \cite{HarCor} Lemma 3.6.1. For $C>0$, we define
$\mathcal{Z}_C\subset\mathcal{Z}$ adding the condition
$k_{\sigma,i}-k_{\sigma,i+1}>C$ for all $i$. For $C>\!>0$, these points give rise to crystalline
representations at $p$ as we will see. Each of the sets $\mathcal{Z}_C$ is
accumulation in $\mathcal Z$, we can thus prove the claim with
$\mathcal Z_C$ replaced by $\mathcal Z$. By \cite{Bijmu} (see
\cite[Thm~9.4]{Her4}), if $z\in\mathcal{Z}$, the system of eigenvalues
corresponding to $z$ has an eigenvector in $H^0(S_K^{tor},\omega^{\kappa}(-D))$. This
implies that actually $z\in\mathcal{E}'(\overline{\QQ}_p)$. It follows from \cite{Su}, or \cite{HarCor}, that there
exists a cuspidal automorphic form $\pi$ of $G(\mathbb{A}_F)$ such
that $\pi_f^{K^p}\neq0$, the Satake parameter of $\pi_v$, $v \nmid p$, corresponds
to $\lambda|_{\mathcal{H}_v}\otimes k(z)$ and $\pi_\infty$ is tempered of weight
$((k_{\sigma_\tau,p_\tau},\cdots,k_{\sigma_\tau,1},-k_{\overline{\sigma_\tau},1},\cdots,-k_{\overline{\sigma_\tau},q_{\tau}}))_{\tau\in\Sigma_F}-\rho_G - w_{0,M}\rho_G$, with $\rho_G$ the half-sum of positive roots, as we will explain.

To be able to clearly label the weights, let $P_h$\footnote{also called $P_\mu^{std}$ in other references} be the parabolic corresponding to $\mathfrak p$, and choose a Borel, equivalently a set $\Phi^+$ of positive roots such that if $\Phi^+_c$ is the set of positive roots contained in the Levi of $\mathfrak p$, then $\Phi^+_{nc} := \Phi^+ \backslash \Phi^+_c$ is chosen to be included in $\mathfrak g_\CC/\mathfrak p$. Equivalently $B \subset P_h^{opp} = P_\mu$, the parabolic opposite to $P_h$. This allows us to label similarly (classical, dominant) weights of representations of $K_\infty$ (with respect to $\Phi^+_c$) and of $G$ (with respect to $\Phi^+$). Let us be more precise for these choices. Let $G(\QQ_p)$ our unitary group, thus given by the hermitian form $J$, and let $T$ be its diagonal torus, and $T^1$ the subtorus of elements of similitude 1. We have an embedding
\[ 
\begin{array}{ccc}
 T^1(\ZZ_p) &\fleche &\prod_{\tau \in \Phi} \GL_{p_\tau,K} \times \GL_{q_\tau,K} 
 \\
\left(
\begin{array}{ccc}
a_1  &   &   \\
  &  \ddots &   \\
  &   &   a_n
\end{array}
\right) 
&\longmapsto &
\left(\left(
\begin{array}{ccc}
\sigma_\tau(a_{p_\tau})^{-1} &   &   \\
  & \ddots  &   \\
  &   & \sigma_\tau(a_1)^{-1}  
\end{array}
\right), \left(
\begin{array}{ccc}
\overline{\sigma_\tau}(a_{q_{\tau}})^{-1}&   &   \\
  & \ddots  &   \\
  &   & \overline{\sigma_\tau}(a_1)^{-1}
\end{array}
\right)\right)_\tau
\end{array}
\]
where $a_i \in \mathcal O_E \otimes \ZZ_p$ and $a_i \overline{a_{n+1-i}} = 1$. Writing $a_i = (z_i,t_i)$ (using the choice of $w|v$ for all $v$ and $z_i$ corresponding to $w$ in $\Phi$), we can rewrite the previous embedding 
 \[ \iota : \begin{array}{ccc}
(( \mathcal O_F \otimes \ZZ_p)^\times)^n &\fleche &\prod_{\tau \in \Phi} \GL_{p_\tau,K} \times \GL_{q_\tau,K} 
 \\
\left(
\begin{array}{ccc}
z_1  &   &   \\
  &  \ddots &   \\
  &   &   z_n
\end{array}
\right) 
&\longmapsto &
\left(\left(
\begin{array}{ccc}
\tau(z_{p_\tau})^{-1} &   &   \\
  & \ddots  &   \\
  &   & \tau(z_1)^{-1}  
\end{array}
\right), \left(
\begin{array}{ccc}
\tau(z_{p_\tau+1})&   &   \\
  & \ddots  &   \\
  &   & \tau(z_n)
\end{array}
\right)\right)_\tau
\end{array}
\]
We choose the diagonal torus and upper Borel for $M \simeq \prod_\tau \GL_{p_\tau} \times \GL_{q_\tau}$\footnote{This is actually the subgroup of $M$ of element with similitude factor 1, but in all this discussion we ignore the similitude factor to simplify the notations.}, which we see as the Levi for $P_\mu \subset \prod_{\tau \in \Phi} \GL_{p_\tau+q_\tau}$, the standard lower parabolic with blocs $p_\tau,q_\tau$. Thus the choice of $\Phi^+$ for $G$ corresponds to the standard upper Borel. Denote $\rho_G$ the half-sum of positive roots in $\Phi^+$.

Thus if we choose $\kappa := (k_{\sigma,i})_{\sigma \in \Sigma_E, 1 \leq i \leq p_{\sigma}}$ a dominant (integral) weight for $M$ for the previous upper triangular Borel, then the algebraic representation of highest weight $\kappa$ is constructed using $-w_{0,M}\kappa$, which induces the weight $\chi$ of $((\mathcal O_F\otimes \ZZ_p)^\times)^n$ which is algebraic of coherent weight $\kappa$ in the sense of equation (\ref{eq:weightkappa}). This weight $\chi$ is $\Phi^+$-dominant if and only if $k_{\sigma_\tau,p_\tau} \geq -k_{\overline{\sigma_\tau},q_\tau}$.

Let $z \in \mathcal Z$, which corresponds to a classical automorphic form (itself giving an automorphic representation $\pi$) appearing in $H^0(S_K^{tor},\omega^\kappa)$ for $\kappa = (k_{\sigma,i})_{\sigma \in \Sigma_E,1\leq i \leq p_\sigma}$ as before, which is a classical (and $M$-dominant) weight. Then $\chi = w(z) \in \mathcal W$ is the algebraic character of weight $\kappa$, i.e.~ is $-w_{0,M}\kappa \circ \iota$. This sheaf $\omega^\kappa$ coincides with the coherent sheaf $V_{s}$ over $\CC$ defined by Harris (\cite{Har}), associated to the highest weight $s$ representation of $M = \prod_{\tau \in \Phi} \GL_{p_\tau}\times\GL_{q_\tau}$, with $s = (-k_{\tau,p_\tau},\dots,-k_{\tau,1},k_{\overline\tau,1},\dots,k_{\overline\tau,q_{\tau}})$ with the previous identifications. This calculation is the one done in \cite{FP} section 7.4, based on \cite{Gold}. Remark that if $\chi$ is algebraic of weight $\kappa$ and dominant, then the dominant representative of $-s$ is given by $\wt(\chi) = (k_{\sigma_\tau,1},\dots,k_{\sigma_\tau,p_\tau},-k_{\overline{\sigma_\tau},q_{\tau}},\dots,-k_{\overline{\sigma_\tau},1})$. In particular as the Hecke eigensystem corresponding to $z$ appears in $H^0(S^{tor}_K,\omega^\kappa)$ thus in $H^0(S^{tor}_K(\CC),V_{s})$ this means that \[ H^0(\mathfrak p,K_\infty, \pi_\infty \otimes V_{s}) \neq \{0\},\]
i.e.~ that the infinitesimal character of $\pi_\infty$ is $-s-\rho_G$ (up to reordering) by e.g. \cite{Har} Proposition 4.3.2 (see also \cite{BP2} Proposition 5.37). But if $z \in \mathcal Z$ and $\pi$ an automorphic representation
corresponding to its system of eigenvalues $\lambda(z)$ of $\mathcal{H}^S$, then using that $w(z)$ is far from the walls, by 
Corollary \ref{cor:A5} there is a semisimple representation $\rho^{u} = \rho^u_z : G_E\rightarrow\GL_n(\overline{\QQ}_p)$ such that
$\rho^{u}(\Frob_v)$ is associated to the semi-simple conjugacy class at $v$ determined by $\lambda$, for all $v\notin S$ and satisfying moreover
\[ (\rho^{u})^\vee \simeq (\rho^{u})^c \otimes \eps^{n-1}.\] 
By for example \cite{BLGGTII}, the previous calculation of the infinitesimal weight means that $\rho^u$, associated to $\pi$, has Hodge-Tate weights given by $-s - \rho_G - \underline{\frac{n-1}{2}}$ i.e.~ the $v,\tau$ Hodge-Tate weights of $\rho^u_z$ are (up to order)
\begin{eqnarray*}
 \left(k_{v,\tau,p_\tau}+1-n,k_{v,\tau,p_\tau-1} + 2- n,\dots, k_{v,\tau,1} + p_\tau -n , -k_{v,\overline\tau,1} +p_\tau-n+1,\dots,-k_{v,\overline\tau,q_\tau}\right)
\end{eqnarray*}
which we can reorder to be dominant for very regular $z \in \mathcal Z$, 
\begin{equation}
\label{HTwtsz} (k_{v,\tau,1} +(p_\tau-n) ,k_{v,\tau,2} + (p_\tau-n -1),\dots, k_{v,\tau,p_\tau} + 1 - n,-k_{v,\overline\tau,q_\tau},\dots,-k_{v,\overline\tau,1} + (p_\tau - n +1)).\end{equation}

We first construct a union of connected components of $\mathcal{E}'$
and a map from this subspace to $\mathcal{X}_{\rhobar}^{\pol}$. As in \cite{Che1}, we construct a determinant
\[ D : G_E \fleche \mathcal{O}_{\mathcal{E'}}.\] 
Let $z \in \mathcal Z$ and $\pi$ an automorphic representation
corresponding to its system of eigenvalues $\lambda(z)$ of $\mathcal{H}^S$, as we have seen, by
Corollary \ref{cor:A5} there is a semisimple representation $\rho^{u} : G_E\rightarrow\GL_n(\overline{\QQ}_p)$ associated to $\lambda(z)$ such that
\[ (\rho^{u})^\vee \simeq (\rho^{u})^c \otimes \eps^{n-1}.\] Let
$D_{z}$ the pseudo-representation of $\rho^{u}_z$. The continuous map
$(D_{z})_{z\in\mathcal{Z}}$ from $G_E$ to $\prod_{x\in\mathcal{Z}}k(x)$
factors actually through $\Gamma(\mathcal{E}',\mathcal{O_{\mathcal{E'}}^+})$ and
gives rise to a pseudo deformation $D$ on
$\Gamma(\mathcal{E}',\mathcal{O}_{\mathcal{E'}}^+)$. By continuity, we have $D^\vee
\simeq D^c \eps^{n-1}$. 

As there is only a finite number of possible reductions modulo $p$ of
$D$, there is $\mathcal{E'}(\overline{\rho})$ an open and closed
subset of $\mathcal{E'}$ of points whose reduction of $D$ is
$(\tr)\overline{\rho}$. This is non empty by Hypothesis \ref{hypmodular}. In particular the restriction of the previous
$D$ to $\mathcal{E'}(\overline{\rho})$ induces a morphism of rigid
analytic spaces
\[ \mathcal{E'}(\overline{\rho}) \fleche \mathcal{X}_{\overline{\rho}}^{\pol}.\] 

Now we construct a rigid analytic map $\mathcal{E}'\rightarrow\mathcal{T}$.

Denote $w = (w_{v,1},\dots,w_{v,n})_v$ the universal character of $((\mathcal O_F\otimes\ZZ_p)^\times)^n$. In \cite[Section 7.2.2]{Her4} we constructed Hecke operators which are in $\mathcal O(\mathcal E')^\times$, denoted by $U_{v,i}$ for $v|p$ in $F$ and $i = 0,\dots,n$. The operator $U_{v,i}$ coincides up to normalisation (this normalisation is made in order to vary in family) with the double class,
\[ \left(
\left(
\begin{array}{cc}
 I_{i} &   \\
 &p^{-\delta_{v=v'}}I_{n-i}
\end{array}
\right)_w,
\left(
\begin{array}{cc}
 I_{n-i} &   \\
 &p^{-\delta_{v=v'}}I_{i}
\end{array}
\right)_{\overline w},
\right)_{w\overline w = v' | p} \in G(\QQ_p) \subset \prod_{w\overline w=v' |p \text{ in F}} \GL_n(E_w)\times \GL_n(E_{\overline w}).
\]
If $U_{v,i}^{class}$ denotes the action of the (classical, i.e.~ non normalised) Hecke operator corresponding to the previous Iwahori double class acting on global sections of the classical automorphic sheaf, for (fixed) algebraic weight $w \in \mathcal W$,  then the normalisation is
\[ U_{v,i} = 
\widetilde{w}_{v}
\left(
\begin{array}{cccccc}
 1 &   &   \\
  & \ddots  &   \\
  &   & 1 \\
  &&&p^{-1}\\
  &&&&\ddots\\
  &&&&&p^{-1}
\end{array}
\right)
U_{v,i}^{class},\]
where $\widetilde{w}_v$ is the (unique) algebraic extension of $w_v$ as a character of $(F_v^\times)^n$\footnote{This normalisation factor, which is a power of $p$, comes when trying to express a character of $T$, the maximal torus of $G$, as one of $F\otimes\QQ_p^\times$. The key point is that the double class corresponding to $U_{v,i}$ has similitude factor $p^{-1}$, thus is not in $T^1$. But remark that then the Hecke operator $F_{v,i}$ has similitude factor 1.}, and there is $i$ times 1 (and $n-i$ times $p^{-1}$) appearing in the matrix (see before and remark 7.5, together with remark 8.3 of \cite{Her4}).
For all $i \in \{1,\dots,n\}$, we set
\[ F_{v,i} := U_{v,i}U_{v,i-1}^{-1}.\]
It corresponds, up to normalisation, to the Hecke operator in $\mathcal A(p)$\footnote{But not the the double class associated to the following matrix, as the Hecke algebra at Iwahori level is not commutative! See e.g. \cite{BC2} Proposition 6.4.1 and the remark that follows.} \[ F_{v,i}^{cl} := \left(
\left(
\begin{array}{ccc}
 I_{i-1} &  & \\
 &p^{\delta_{v=v'}} & \\
 && I_{n-i}
\end{array}
\right)_w,
\left(
\begin{array}{ccc}
 I_{n-i} &   \\
 &p^{-\delta_{v=v'}} & \\
 && I_{i-1}
\end{array}
\right)_{\overline w},
\right)_{w\overline w = v' | p} \in G(\QQ_p),
\]
and the normalisation is the following, for $w$ algebraic of infinitesimal weight $h = (h_{\tau,i})_{\tau,i}$
\[ F_{v,i} = \widetilde{w}_{v}
\left(
\begin{array}{ccccc}
 1 &   &   \\
  & \ddots  &   \\
  &&p\\
  &&&\ddots\\
  &&&&1
\end{array}
\right)
F_{v,i}^{cl} = p^{\sum_\tau h_{i,\tau}}F_{v,i}^{cl},\] where $p$ is in position $i$. The all point is that $F_{v,i}^{cl} \notin \mathcal O(\mathcal E')$, i.e.~ they don't interpolate, whereas $F_{v,i} \in \mathcal O(\mathcal E')^\times$.
We construct characters $\delta^0_{v,i} : F_v^\times \fleche \mathcal O(\mathcal E')^\times$ by setting  
$\delta^0_{v,i}(p) := F_{v,i}$ and 
\[ (\delta^0_{v,i})_{|\mathcal O_{F_v}^\times} = w_{v,i},\] and we finally set
\begin{equation}
\label{defdelta} \delta_{v,i} := \delta_{v,i}^0 \times \prod_{\tau} x_\tau^{s_\tau(i)} \times |.|^{\frac{1-n}{2}},
\end{equation}
where for $\tau : F_v \hookrightarrow \CC_p$, ${x_\tau}_{|\mathcal{O}_{F_v}} = \tau_{\mathcal{O}_{F_v}}$ and $x_\tau(p) = 1$, and $s_\tau(i) = \frac{1-n}{2} + p_\tau - i$ if $1 \leq i \leq p_\tau$, and $s_\tau(i) = \frac{n-1}{2} - (i - p_\tau - 1)$ if $i > p_\tau$\footnote{Remark that actually in our quasi split situation we have $p_\tau, q_\tau$ which doesn't depend on $\tau \in \Phi$. In any case $(s_\tau(1),\dots,s_\tau(n))_\tau = w_{0,M}(0,\dots,n-1)_\tau+ \underline{\frac{1-n}{2}} = -w_{0,M}\rho_G$, with $\rho_G$ defined in the next paragraph.}. Thus the characters $(\delta_{v,i})_{v,i}$ gives a map
\[ \mathcal E' \fleche \mathcal T.\]

Still denote $\mathcal Z$ for $\mathcal Z \cap \mathcal{E'}(\overline\rho)$, which is Zariski dense and accumulation. Now we are reduced to prove that the two constructed maps $\mathcal{E'}(\overline{\rho}) \fleche \mathcal{X}_{\overline{\rho}}^{\pol}$ and $ \mathcal{E'}(\overline{\rho}) \fleche \mathcal T$ are compatible, in the sense that for $z \in \mathcal Z$ the second map is the parameter of a triangulation for the image of $z$ via the first map. 
 By local global compatibility at $v$ for $\pi$ and $\rho^{u}$, we have that, using $\pi_v^I \neq \{0\}$ by 
construction of $\mathcal E'$, that $\pi_v$ is a subquotient of the Borel induction of an unramified character $\chi$ of $(F_v^\times)^n$ (e.g. \cite{BC2} Proposition 6.4.3 and 6.4.4) with 
$\chi$ related to the eigenvalues of $F_{v,i}^{class}$ at $z$ ($\chi = (\varphi_1,\dots,\varphi_n) = (F_{v,1}^{cl},\dots,F_{v,i}^{cl})\delta_B^{1/2}$). But $F_{v,i}$ has a locally 
constant valuation (thus not $F_{v,i}^{cl}$ !), so up to choose another point of $\mathcal Z$ close to $z$, we can assume that this induced representation is irreducible, and thus unramified. By local global compatibility this proves that there is an accumulation subset of $\mathcal Z$, which accumulates at any point of $\mathcal E$ with algebraic weight, consisting of points $z$ with  representation $\rho_z$ semi simple corresponding to $D_z$, crystalline at every $v | p$ and such that $D_{\cris}(\rho_v)$ has all its refinement, one of which is given by $(D_{\cris}(\delta_{v,i}))_{1\leq i \leq n}$, i.e.~ $(F_{v,i}^{cl})_i$. Moreover, the calculation for $z \in \mathcal Z$ we did in equation \ref{HTwtsz}, together with the definition of the weight of $\delta$ in equation \ref{defdelta} implies that the Hodge-Tate weights of $\rho_z$ are given by $\delta_{|(\mathcal O_F \otimes \ZZ_p)^\times}$, in the right order ! This means that the map
\[ D \times \delta : \mathcal E(\overline\rho) \fleche \mathcal{X}_{\overline{\rho}}^{\pol} \times \mathcal T,\]
sends a dense subset of $\mathcal Z$ (namely the previous one where points are crystalline) into $\mathcal Z'_{K^p}$, but conversely by construction of $\mathcal E'$, all points of $\mathcal Z'_{K^p}$ are in the image of the previous map.

Moreover $D \times \delta$ is a closed immersion. Indeed, by construction $\mathcal{E'}$ is (locally) constructed as the image of $\mathcal{H} \otimes \mathcal{O}_\mathcal W = \mathcal{H}^S\otimes \mathcal{O}_\mathcal T$ where $\mathcal{H} = \mathcal{H}^S \otimes \mathcal A(p)$ acting on some space of overconvergent, locally analytic modular forms (of finite slope). Let $U \subset \mathcal{X}_{\overline{\rho}}^{\pol} \times \mathcal T$ be an affinoid, in particular it is quasi-compact thus the slopes of $\mathcal A(p)$ on $U$ are bounded (say by $\alpha$). Thus $(D\times\delta^{-1})(U)$ is included in $\mathcal{E}_{v,w}^{\leq \alpha}$ for some $v,w$ (see \cite{Her4}) and then, by local-global compatibility, it is clear that $(D\times\delta)^{-1}(U) \fleche U$ is a closed immersion. As explained, $\mathcal Z$ accumulates to any point with classical weight of $\mathcal{E'}$, thus to $\mathcal Z'$. If we denote by $h$ the composite of the map
\[ \mathcal E \fleche \mathcal T \fleche \mathcal W,\]
it coincides with a map
\[ \mathcal E \overset{w}{\fleche} \mathcal W \overset{\phi}{\fleche} \mathcal W,\]
where $\phi$ is the isomorphism of $\mathcal W$ given by the definition (\ref{defdelta}). The properties of the map $h$ thus comes from the analogous one for $w$, proven in \cite{Her4} Theorem 9.5 (see also \cite{Che1} which was the first to prove those properties).
\end{proof}

\begin{cor}\label{cor:finite_slope}
If $\pi$ is a cuspidal, algebraic, regular automorphic form (of tame level $K^p$), which appears in degree 0 coherent cohomology and which is finite slope at $p$, then $\rho_\pi$ appear in $\mathcal F_{K^p}(\overline\rho_\pi)$. If moreover $\delta$ is an accessible refinement for $\rho_\pi$, then $(\rho_\pi,\delta) \in \mathcal E_{K^p}(\overline\rho_\pi)$. In particular, $\mathcal E_{K^p}(\overline\rho)$ is the closure in $\mathcal X_{\overline\rho}^{\pol}\times \mathcal T$ of homomorphically modular (for a cuspidal, algebraic, regular automorphic representation of tame level $K^p$), finite slope at $p$, representations $\rho$, together with an accessible refinement. 
\end{cor}

\begin{proof}
As any finite slope automorphic form has an accessible refinement, it is enough to prove the second part of the statement. This is a direct consequence of the fact that $\mathcal E(\overline\rho)$ coincides with the Eigenvariety $\mathcal E'(\overline\rho)$ constructed in \cite{Her4}. We take the notation of \cite{Her4} section 6. Let $\pi$ be an automorphic representation as in the statement and assume given $n$ and $f \in H^0(\mathcal X_1(p^n)^{tor},\omega^\kappa)$ be a cuspidal holomorphic modular form corresponding to $\pi$ which is a finite slope eigenvector at $p$, where $\mathcal X_1(p^n)^{tor}$ is a toroidal compactification of the Shimura variety of level at $p$,
\[ \Gamma_1(p^n) = \{ M \in G( \ZZ_p) | M \in U(\ZZ_p) \pmod{p^n}\},\]
with $U \subset B \subset G_{\ZZ_p}$ are unipotent and Borel subgroup.
Then there is an open $\mathcal X_1(p^n)^{tor}(v) \subset \mathcal X_1(p^n)^{tor}$, for $v$ small enough (a strict neighborhood of the ordinary locus), and for the Shimura variety $\mathcal X_0(p^n)^{tor}$ of level 
\[ \Gamma_0(p^n) = \{ M \in G(\mathcal \ZZ_p) | M \in B(\ZZ_p) \pmod{p^n}\}.\]
we have a corresponding open $\mathcal X_0(p^n)^{tor}(v)$, such that over $\mathcal X_0(p^n)^{tor}(v)$ there exists the sheaf 
$\omega^{\chi\dag}_w(-D)$ of $w$-analytic cuspidal overconvergent modular forms of (any $n$-analytic) $p$-adic weight $\chi$. By definition, finite slope (at $p$) sections of $H^0(\mathcal X_0(p^n)^{tor}(v),\omega_w^{\chi\dag}(-D))$ which are eigenvectors for the Hecke operators do appear in the Eigenvariety $\mathcal E_{K^p}$. By construction, there exists a rigid open $\mathcal{IW}_w^+ \subset \mathcal T^\times/U$ in the torsor of trivialisations of $\omega$ above $\mathcal X_1(p^n)^{tor}(v)$, so that $H^0(\mathcal X_0(p^n)^{tor}(v),\omega_w^{\chi\dag}(-D))$ is the set of sections $H^0(\mathcal{IW}_w^+,\mathcal O_{\mathcal{IW}_w^+}[\chi](-D))$ which are $\chi$-equivariant for the action of $T(\ZZ_p)$\footnote{A warning : this action is twisted compared to the one on $\mathcal T^\times/U$, see \cite{Her4} see Proposition 6.16 and its proof.} on $\mathcal{IW}_w^+$ above $\mathcal X_0(p^n)^{tor}(v)$ (see \cite{Her4} Definition 6.10).
Let $\eps$ be a finite order character of $T(\ZZ_p)$ trivial on $T(1+p^n\ZZ)$, i.e. a character of $\Gamma_0(p^n)/\Gamma_1(p^n)$. We claim that there is an injection, for any algebraic $\kappa$
\[ H^0(\mathcal X_1(p^n)^{tor},\omega^{\kappa}(-D))(\eps) \hookrightarrow H^0(\mathcal{IW}_w^+,\mathcal O_{\mathcal{IW}_w^+}[\eps\kappa](-D)).\]
But this comes from the identification (see \cite[Definition 3.1]{Her4})
\[ H^0(\mathcal X_1(p^n)^{tor},\omega^\kappa) = H^0(\mathcal T^\times/U,\mathcal O[\kappa^\vee]),\]
the fact that the restriction to the open $\mathcal X_1(p^n)^{tor}(v)$ is injective by analytic continuation, and that the restriction from $\mathcal T^\times/U \times_{\mathcal X_1(p^n)^{tor}}\mathcal X_1(p^n)^{tor}(v)$ to $\mathcal{IW}_w^+$ is also injective again by analytic continuation. Moreover all those maps are equivariant for the action of the Atkin-Lehner algebra $\mathcal A(p)$. More precisely, the action of $\mathcal A(p)$, i.e. of matrices of the form $\Diag(p^{a_1},\dots,p^{a_n})$, on $\mathcal E$, together with the map $T(\ZZ_p) \fleche \Gamma(\mathcal W,\mathcal O_{\mathcal W})$, thus giving a map $T(\ZZ_p) \fleche \Gamma(\mathcal E,\mathcal O_{\mathcal E})$, can be packaged together as an map $T(\QQ_p) \fleche \Gamma(\mathcal E,\mathcal O_{\mathcal E})$. Then the previous injection is equivariant for this action of $T(\QQ_p)$.
But we can assume that $f$ is an eigenvector for $\Gamma_0(p^n)/\Gamma_1(p^n)$, of some character $\eps$ as before. Thus the associated refinement at $p$ of $f$ gives a character $\delta$ of $T(\QQ_p)$ depending on the refinement, such that $\delta_{T(\ZZ_p)} = \kappa\eps$. Then the previous injection assure that $f$ with this $\delta$ appears in $\mathcal E'_{K^p}$.

For the last part of the statement, since any crystalline point at $p$ is indeed finite slope, and finite slope holomorphic points are already in $\mathcal E_{K^p}(\overline\rho)$ as we just proved, we get the result.
\end{proof}

From now on, to lighten notations denote $\mathcal{E}(\overline{\rho}) := \mathcal{E}_{K^p}(\overline{\rho})$, $\mathcal{F}(\overline{\rho}) := \mathcal{F}_{K^p}(\overline{\rho})$ $\mathcal Z' := \mathcal Z'_{K^p}$ accordingly\footnote{Everything we will say is still dependant of this level $K^p$. At the end of the article, in corollary \ref{cor:Koutsidep}, we show that there exists an optimal level $K^p$, but we can't choose it right away. Compare \cite{CheFern} Lemma 2.4}.

\section{Automorphic forms, infinite fern and big image}
\label{sect:aggen}

Let $K$ be finite extension of $\Qp$ and $\overline{K}$ an algebraic closure of $K$. {Denote $v_p$ the $p$-adic absolute value of $\overline K$ such that $v_p(p)=1$.} Let $K_0\subset K$ be the maximal unramified extension of $\Qp$ with Frobenius operator $\sigma$ and set $f\coloneqq [K_0:\Qp]$. Let $L$ be a finite extension of $\Qp$ such that $K\otimes_{\Qp}L\simeq L^{[K:\Qp]}$. If $(\rho,V)$ is a crystalline representation of $\Gal(\overline{K}/K)$, we denote $(D_{\cris}(V),\varphi)$ its associated $\varphi$-module. It is a finite dimensional free  $K_0\otimes_{\Qp}L$-module of rank $n=\dim_LV$ with a $\sigma\otimes\Id_L$-linear automorphism $\varphi$. Its de Rham module $D_{\dR}(V)$ is a filtered finite free $K\otimes_{\Qp}L$-module. The \emph{Hodge-Tate type} of $V$ is the $[K:\Qp]$-uple $(k_{1,\tau}\geq \cdots\geq k_{n,\tau})_{\tau : K\hookrightarrow L}$ where the $k_{i,\tau}$ are the integers $m$ such that $\gr^{-m}(D_{\dR}(V)\otimes_{K,\tau}L)\neq0$ counted with multiplicity.

\begin{defin}
We say that a crystalline Galois representation $(\rho,V)$ over $L$ is \emph{Hodge-Tate regular} (or simply \emph{HT-regular}) if, for all $\tau : K\hookrightarrow L$, the integers $k_{i,\tau}$ are pairwise distinct. It is said to be \emph{$\varphi$-generic} if the \emph{linear} endomorphism $\varphi^f$ of the finite free $K_0\otimes_{\Qp}L$-module $D_{\cris}(V)$ is split semisimple regular ie has $\dim_LV$ pairwise distinct eigenvalues $\varphi_i$ in $L$ such that $\varphi_i/\varphi_j \notin \{ 1,p^f\}$ (note that these eigenvalues are in $L$ since $\varphi^f$ commutes to the semilinear action of $\Gal(K_0/\Qp)$ on $D_{\cris}(V)$ given by $\varphi$).
\end{defin}

\begin{rema}
\label{rema:WG}
In \cite[\S3]{CheFern} Chenevier introduced the notion of \textit{weakly generic}
crystalline $(\varphi,\Gamma)$-module, i.e.~ crystalline $(\varphi,\Gamma)$-modules for which
all refinements are non-critical, and \textit{weakly-generic}
crystalline $(\varphi,\Gamma)$-module for which $n$ (the rank of the
$(\varphi,\Gamma)$-module) well-positioned refinements are non
critical. It is possible to deduce from the results of \cite{CheFern}
that if the classical points of $\mathcal{E}$ or
$\mathcal{X}_{\overline{\rho}}^{\pol}$ which are weakly-generic at $p$
are Zariski-dense, then modular points are dense in
$\mathcal{X}_{\overline{\rho}}^{\pol}$ (or the closure has at least the expected dimension). He moreover proves that if $n = 3$, every $\varphi$-generic, HT regular absolutely irreducible $(\varphi,\Gamma)$ is automatically weakly generic. Unfortunately it does not seem to be true anymore even for $n=4$ that absolutely irreducible points are weakly generic, as shown in the following example, and thus it does not seem any easier to prove that weakly generic points are dense in $\mathcal{X}_{\overline{\rho}}^{\pol}$ when $n \geq 4$ than proving the analogous result for generic points.
\end{rema}

\begin{exemple}
Let $(V,\varphi,\Fil^\bullet)$ be the filtered $\varphi$-module of an
irreducible, $\varphi$-generic, HT regular, crystalline 4 dimensional
representation of $G_{\QQ_p}$ with $L$-coefficients. Choosing $L$ big
enough, we can assume that there exists a basis $(f_1,f_2,f_3,f_4)$ of $V$ such that $\varphi(f_i) = \varphi_i f_i$. Refinements of $V$ are thus given by a permutation $\sigma$ of this basis. 
By irreducibility and weak-admissibility, we check that it is impossible for $\Fil^kV \neq \{0\},V$ to be $\varphi$-stable. 
 For example, suppose that the HT weights are ${-k_4 < -k_3 < -k_2 < -k_1}$ (i.e.~ the jumps of the filtration on $V$ are $k_1 < k_2 < k_3 < k_4$) with 
 \begin{itemize}
\item $\Fil^{k_1}V = <f_2,f_3,f_1+f_4> = <f_2,f_2+f_3,f_1+f_4>$
\item $\Fil^{k_2}V = <f_1+f_4,f_2+f_3>$
\item $\Fil^{k_3}V = <f_1 + f_4>$
\item $\Fil^{k_4}V = \{0\}.$
\end{itemize}
We can check that the non critical refinements are given by $\sigma = \id,(23),(14),(14)(23)$, and they don't form a weakly-nested sequence. Moreover, for generic choices of $k_i$ and $v(\varphi_i)$, the associated $p$-adic representation $V(D)$ will be irreducible. Indeed, if $v(\varphi_1) = v(\varphi_2) = v(\varphi_3) = 16, v(\varphi_4) = 12, k_1 = 0,k_2=10,k_3 = 20,k_4 = 30$, then we can check that no non-trivial $\varphi$-stable submodule of $D$ is weakly-admissible. In particular if we do not already know that generic points are Zariski dense, it is not likely to prove that 
weakly generic ones are. 

Moreover we can check that locally the image of the tangeant space of those refined points doesn't cover all the tangent space for the corresponding point in the local deformation ring (i.e.~ at this point the analogous of proposition \ref{propraf} only for non-critical refinements isn't true). In the following we will use a replacement of those generic points by the so-called \emph{almost generic} ones, which are Zariski dense in the fern and for which we can apply proposition \ref{propraf} and thus Chenevier's stategy. Remark that some irreducible weakly-admissible filtered $\varphi$-modules of dimension 4 which admits a critical refinement are weakly-generic, but we don't know how to discriminate these from the previous example on the deformation rings (or on the infinite fern).
\end{exemple}

Recall that $k$ is a finite field of characteristic $p$ and that $\rhobar : G_E\rightarrow\GL_n(k)$ is a continuous semisimple representation which is conveniently modular of tame level $K^p$.

Denote by $\overline{\mathcal{F}(\overline{\rho})}$ the Zariski-closure of the image $\mathcal{E}(\overline{\rho})\rightarrow\mathcal{X}_{\overline{\rho}}^{\pol}$, i.e.~ the Zariski closure of the infinite fern $\mathcal{F}(\overline{\rho})$.

\begin{defin}[] 
\label{defin:ag}
We say that a Galois representation $\rho : G_E \fleche \GL_n(\overline{\QQ_p})$ has enormous image, if $\rho(G_{E(\zeta_{p^\infty})})$ is enormous, in the sense of \cite{NT} Definition 2.27.
We say that a point $x$ of $\mathcal{E}(\overline\rho)$ (resp.~of
$\mathcal{X}_{\overline{\rho}}^{\pol}$) is \textit{almost generic}
if it is in $\mathcal{Z}'=\mathcal{Z}'_{K^p}$ (resp.~in the image of $\mathcal{Z}'$ in $\mathcal{X}_{\overline{\rho}}^{\pol}$), the associated Galois representation $\rho$ has 
enormous image, and if $\rho_{|G_{E_v}}$ is crystalline, $\varphi$-generic, HT-regular, posses an irreducible refinement and is absolutely irreducible for all $v | p$.
\end{defin}

Let $v$ be a place of $E$ dividing $p$. Let $(\rho,V)$ be a continuous finite dimensional representation of $G_{E_v}$ over $L$. It follows from the compactness of $G_{E_v}$ that $\rho(G_{E_v})$ is a closed subgroup of $\GL_L(V)$. The $p$-adic analogue of Cartan's Theorem shows that $\rho(G_{E_v})$ is a $p$-adic Lie subgroup of $\GL_L(V)$ so that we can define $\mathfrak{g}_\rho\coloneqq\Lie(\rho(G_{E_v}))$ (see \cite[Rem.~I.1.1]{SerreMcGill}) and $\mathfrak{g}_{\rho,L}$ the $L$-span of $\mathfrak{g}_\rho$ in $\End_L(V)$. Our goal is to prove that almost generic points are Zariski-dense in the infinite fern $\mathcal{F}(\overline{\rho})$. Such a result was proven by Taïbi (\cite{Taibi}) in a slightly different (and more difficult) context (improving results of \cite{BC2}). As the case we consider is easier, and for convenience of the reader, we repeat the argument in our context.

Denote by $K_n/E_v$ the compositum of extensions of degree dividing $n$, this is a finite Galois extension.

\begin{prop}
Let $(\rho,V)$ be some continuous $n$-dimensional representation of $G_{E_v}$ over $L$
and assume that $(\rho|_{G_{K_n}},V)$ is absolutely irreducible. Then $V$ is a simple $\mathfrak{g}_{\rho,L}$-module, $\mathfrak{g}_{\rho,L}$ is a reductive Lie algebra and $\mathfrak h_{\rho,L}$, the semisimple part of $\mathfrak{g}_{\rho,L}$, is isomorphic to a sub Lie algebra of $\mathfrak{sl}(V)$ of semisimple rank at most $\dim_LV-1$.
\end{prop}

\begin{proof}
Let $K$ be a finite extension of $E_v$, we claim that $\rho|_{G_K}$ is
absolutely irreducible. We can assume that $K/E_v$ is Galois. Suppose
that $\rho|_{G_K}$ is not absolutely irreducible, then as 
\[\rho : G_{E_v} \fleche \GL_L(V) \simeq \GL_n(L)\]
is absolutely irreducible and $G_K$ is normal in $G_{E_v}$, we have,
up to enlarge $L$, a decomposition into absolutely irreducible $G_K$-representations
\[ \rho|_{G_K} = \bigoplus_{k=1}^r W_k,\]
and $G_{E_v}$ permutes these representations, in particular they have the same dimension. Let $H \subset G_{E_v}$ be the stabiliser of $W_1$, then $G_{E_v}/H$ acts transitively on the $W_k$ and thus $[G_{E_v}:H]\dim(W_1) = n$. In particular $H = G_{K'}$ with $K'$ a finite extension of $E_v$ of degree dividing $n$, thus $K' \subset K_n$, but $\rho|_{G_{K_n}}$ is irreducible, thus as is $\rho|_{G_{K'}}$ and thus $r=1$ i.e.~ $\rho|_{G_K}$ is irreducible.

For each open subgroup $H\subset G_{E_v}$, the representation $(\rho|_H,V)$ is irreducible, so that $V$ is a simple $\mathfrak{g}_{\rho,L}$-module. Thus $\mathfrak{g}_{\rho,L}$ has a simple faithful module, which implies that $\mathfrak{g}_{\rho,L}$ is a reductive Lie algebra. Let $\mathfrak{g}_{\rho,L}=\mathfrak{a}_{\rho,L}\oplus\mathfrak{h}_{\rho,L}$ be the decomposition of $\mathfrak{g}_{\rho,L}$ as direct sum of an abelian and of a semisimple Lie algebras. As $V$ is an absolutely simple $\mathfrak{g}_{\rho,L}$-module, the Lie algebra $\mathfrak{a}_{\rho,L}$ acts on $V$ by scalars and $V$ is an absolutely simple $\mathfrak{h}_{\rho,L}$-module. Then $\mathfrak{a}_{\rho,L}\subset L\Id_V$ and, as $\mathfrak{h}_{\rho,L}$ is semisimple, $\mathfrak{h}_{\rho,L}\subset\mathfrak{g}_{\rho,L}\cap\mathfrak{sl}(V)$. As a consequence $\mathfrak{h}_{\rho,L}=\mathfrak{g}_{\rho,L}\cap\mathfrak{sl}(V)$ is a semisimple Lie algebra and $V$ is an absolutely simple $\mathfrak{h}_{\rho,L}$-module. As $\mathfrak{sl}(V)$ has rank $\dim_LV-1$, the rank of $\mathfrak{h}_{\rho,L}$ is at most $\dim_LV-1$.
\end{proof}

\begin{prop}\label{prop:crit-image-ouverte}
There exist finitely many nonzero $\QQ$-linear forms $\Lambda_1,\dots,\Lambda_r$ on $\QQ^n$ such that the following is true: let $(\rho,V)$ be a crystalline $n$-dimensional representation of $G_{F_v}$ over $L$, with Hodge-Tate weights $(k_{1,\sigma} \leq \dots \leq k_{n,\sigma})_\sigma$ such that $(\rho|_{G_{K_n}},V)$ is absolutely irreducible and such that there exists at least one $\sigma : F_v \hookrightarrow L$ such that for all $1\leq i\leq r$, $\Lambda_i(k_{1,\sigma},k_{2,\sigma},\dots,k_{n,\sigma})\neq0$, then $\rho(G_{F_v})$ contains an open subgroup of $\SL(V)$.
\end{prop}

\begin{proof}
Let $C$ be some algebraically closed field of characteristic $0$. The
classification of semisimple Lie algebras and their representations
shows that all semisimple Lie algebras and their finite dimensional simple modules are defined over $\QQ$, that there are finitely many isomorphism classes of semisimple Lie algebras of bounded rank and that each of them has finitely many semisimple modules of bounded rank. Consequently, for a fixed $n\geq2$, there exist a finite number of pairs $(\mathfrak{h}_i,\theta_i)$ where $\mathfrak{h}_i$ is a semisimple Lie algebra and $\theta_i$ an embedding of $\mathfrak{h}_i$ in $\mathfrak{sl}_{n,\QQ}$ such that for each semisimple Lie subalgebra $\mathfrak{h}\subset\mathfrak{sl}_{n,C}$, there exists $i$ such that $\mathfrak{h}\simeq\mathfrak{h}_i\otimes_{\QQ}C$ and the inclusion is $\GL_n(C)$-conjugated to $\theta_i\otimes\Id_C$. As a consequence a Cartan subalgebra of $\mathfrak{h}$ is conjugated to one of finitely many $\QQ$-linear subspaces of the space of diagonal matrices in $\mathfrak{sl}_{n,C}$. Moreover it follows from \cite[VIII.\S3~Prop.2.(ii)]{BourbakiLie} and from Borel-de Siebenthal Theorem (\cite[Thm.~12.1]{Kane}) that a semisimple Lie subalgebra of $\mathfrak{sl}_{n,C}$ containing $\mathfrak{h}$ is equal to $\mathfrak{sl}_{n,C}$ or of rank strictly less than $n-1$. Thus there exist finitely many nonzero $\QQ$-linear forms $\Lambda_1',\dots,\Lambda_s'$ on $\QQ^n$ such that if $\mathfrak{h}$ a semisimple subalgebra of $\mathfrak{sl}_{n,C}$ of rank strictly less than $n-1$ and $x\in\mathfrak{h}$ is a semisimple element of eigenvalues $\lambda_1,\dots,\lambda_n$ (counted with multiplicities), then there exists $1\leq i\leq s$ and $w\in\mathfrak{S}_n$ such that $w(\Lambda_i')(\lambda_1,\dots,\lambda_n)\coloneqq\Lambda_i'(\lambda_{w(1)},\dots,\lambda_{w(n)})=0$. We set
\[ \set{\Lambda_1,\dots,\Lambda_r}=\set{w(\Lambda_i') \mid 1\leq i\leq s, \, w\in\mathfrak{S}_n}.\]

Let $\Theta \in \End_{\CC_p\otimes_{\QQ_p}L}(\CC_p\otimes_{\QQ_p}V)$ the Sen operator of $V$. As $(\rho,V)$ is Hodge-Tate, it follows from \cite[Thm.~1]{SenLie} that $\Theta$ belongs to $\CC_p\otimes_{\Qp}\mathfrak{g}_{\rho,L}\subset\CC_p\otimes_{\Qp}\End_{\QQ_p}V$.

Suppose $L$ is big enough so that $F_v \otimes_{\QQ_p} L \simeq L^{[F_v:\QQ_p]}$. Then\[ \CC_p \otimes_{\QQ_p} L = \prod_{\sigma : F_v \hookrightarrow L} \CC_p \otimes_{F_v,\sigma} L,\]
decomposes over all embeddings of $F_v$ and let $\Theta_\sigma$ be the $\CC_p\otimes_{F_v,\sigma}L$-linear endomorphism of $\CC_p\otimes_{F_v,\sigma}V$ induced by $\Theta$. The eigenvalues of $\Theta_\sigma$ are the $\sigma$-Hodge-Tate weights $(k_{1,\sigma}\leq\cdots\leq k_{n,\sigma})$ of $(\rho,V)$ (counted by multiplicities).

Assume that $\Lambda_i(k_{1,\sigma},\dots,k_{n,\sigma})\neq0$ for all $1\leq i\leq r$. Then by what preceeds, the element $\Theta_\sigma$ can't be contained in a strict semisimple Lie subalgebra of $\CC_p\otimes_{F_v,\sigma}\mathfrak{sl}(V)$ so that $\CC_p\otimes_{F_v,\sigma}\mathfrak{h}_{\rho,L}=\CC_p\otimes_{F_v,\sigma}\mathfrak{sl}(V)$. For dimension reasons, we have $\mathfrak{h}_{\rho,L}=\mathfrak{sl}(V)$. We conclude that $\rho(G_{F_v})$ contains an open subgroup of $\SL(V)$.
\end{proof}

\begin{prop}\label{prop:density-absirred}
The set of points $x\in \overline{\mathcal{F}(\overline{\rho})}$ such
that $\Tr\rho_x|_{G_{K_n}}$ is absolutely irreducible is a
Zariski-dense and Zariski-open subset of $\overline{\mathcal{F}(\overline{\rho})}$.
\end{prop}

\begin{proof}
The fact that the absolutely irreducible locus is Zariski-open is a
consequence of \cite[\S4.2]{Chedet}. In order to prove that it is
Zariski-dense, it is then sufficient to prove that each Zariski-open subset $U$ of $\mathcal{F}(\overline{\rho})$ contains a point $x$ such that $\Tr\rho_x|_{G_{K_n}}$ is absolutely irreducible.

Now we follow the strategy of \cite{BCsign} and \cite{Taibi}. Let us
fix some notation. If
$x=(\rho_x,\delta_x)\in\mathcal{Z}'\subset\mathcal{E}(\rhobar)$ and
$\sigma : K_n\hookrightarrow\overline{\Qp}$, we let
$(k_{\sigma,1}(x)\leq\cdots \leq k_{\sigma,n}(x))$ the Hodge-Tate
weights at $\sigma$ of $\rho_x|_{G_{K_n}}$ and
$(\phi_1(x),\cdots,\phi_n(x))\in k(x)^n$ the ordered eigenvalues of
the linearized Frobenius of $D_{\cris}(\rho_x|_{G_{K_n}})$
corresponding to the refinement of $\rho_x|_{G_{F_v}}$ defined by
$\delta_{x,v}$. We also set $k_i(x)=\sum_{\sigma}k_{\sigma,i}(x)$. Let
$e$ be the ramification index of $K_n/\Qp$. The functions
$x\mapsto v_p(\phi_i(x))+e^{-1}k_i(x)$ are therefore locally constant on
$\mathcal{Z}'$. Now fix $U$ a Zariski open non empty subset of
$\overline{\mathcal{F}(\overline{\rho})}$. We have
$U\cap\mathcal{F}(\overline{\rho})\neq\emptyset$ so that the inverse
image $V$ of $U$ in $\mathcal{E}(\overline{\rho})$ is a non empty
Zariski-open subset. Let $x\in V\cap\mathcal{Z}'$. Let
$c\geq\max_i\vabs{v_p(\phi_i(x))+e^{-1}k_i(x)}$ and larger than
$n^2+1$ and let $\mathcal{Z}_c''$ be the subset of point
$z\in\mathcal{Z}'$ such that
$k_{i+1}(z)-k_i(z) > cen(k_i(z)-k_{i-1}(z))$ $2\leq i\leq n-1$, $2\leq i\leq n-1$, and $k_2(z) - k_1(z) > 3cen$. If
$z\in\mathcal{Z}_c''$, then
$\vabs{\sum_{i\in I}k_i(z)-\sum_{i\in J}k_i(z)}>cen$ for all distinct
non empty proper subsets $I$ and $J$ of the same cardinal in
$\set{1,\dots,n}$ (by the same proof than \cite[Lem.~4.5.5]{BC2}). Then $\mathcal{Z}''_c$
is Zariski dense and accumulates at $\mathcal{Z}'$ in
$\mathcal{E}(\rhobar)$. Therefore there exists a point
$y\in\mathcal{Z}''_c\cap V$ such that moreover $\max_i\vabs{v_p(\phi_i(y))+e^{-1}k_i(y)}\leq c$. Then, for $i \neq j$,
\begin{eqnarray*}\vabs{v_p(\phi_i(y))-v_p(\phi_j(y))} \geq \frac{1}{e}\vabs{k_i(y) - k_j(y)} - \vabs{v_p(\phi_i(y) + \frac{k_i(y)}{e}} - \vabs{v_p(\phi_j(y))+\frac{k_j(y)}{e}} \\ > 3cn - c - c > 1.\end{eqnarray*}
In particular, $\phi_i(y)/\phi_j(y) \neq p$, so if $\Pi$ is an
automorphic representation corresponding to $y$, we have that $\Pi_v$
is an irreducible principal series. In particular, all its refinements
are accessible. Now we choose a
transitive permutation $\sigma$ of $\set{1,\dots,n}$ and, since all
the refinements of $\Pi_v$ are accessible, there exists
$z_0\in\mathcal{Z}'$ such that $\rho_y\simeq\rho_{z_0}$ and
$(\phi_1(z_0),\dots,\phi_n(z_0))=(\phi_{\sigma(1)}(y),\dots,\phi_{\sigma(n)}(y))$. As in the proof of \cite[Lem.~3.3]{BCsign}, we
deduce that $\sum_{i\in I}v_p(\phi_i(z_0))+e^{-1}k_i(z_0))\neq0$ for
any non empty proper subset $I$ of $\set{1,\dots,n}$. Let
$C=\max_i\vabs*{v_p(\phi_i(z_0))+e^{-1}k_i(z_0)}$. Choose $V'$ a neighborhood of $z_0$ in $\mathcal E(\overline\rho)$. Let
$\mathcal{Z}'_C$ be the subset of point $z\in\mathcal{Z}'$ such that
$k_{\sigma,j+1}(z)-k_{\sigma,j}(z)>C$ for any
$\sigma\in\Hom(K_n,\overline{\Qp})$ and $1\leq i\leq n-1$. The set
$\mathcal{Z}'_C$ is Zariski-dense and accumulates at $\mathcal{Z}'$ in
$\mathcal{E}(\overline{\rho})$ so that there exists a point
$z\in V'\cap\mathcal{Z}'_C$ such that
$\sum_{i\in I}(v_p(\phi_i(z))+e^{-1}k_i(z))\neq0$ for any non empty
proper subset $I$ of $\set{1,\dots,n}$. For any of those $z$, by \cite[Lem.~3.2.3]{Taibi},
if $D'\subset D_{\cris}(\rho_z|_{G_{K_n}})$ is a nonzero proper weakly
admissible sub-$\varphi$-module, then there exists a non empty proper
subset $I\subset\set{1,\dots,n}$ such that
\[ \sum_{i\in I}v_p(\phi_i(z))+e^{-1}\sum_{i\in I}k_i(z)=0\] which
is not possible. Therefore $\rho_z|_{G_{K_n}}$ is absolutely
irreducible.
\end{proof}

\begin{prop}
\label{prop:noncritref}
Let $x \in \mathcal E$ be a point which is crystalline at $v | p$. Then there exists $C > 0$ such that for any $z \in \mathcal Z'_C$ sufficiently close to $x$ the local representation $\rho_{z,v}$ admits a non-critical refinement (see Definition \ref{def:transpo_simples}).
\end{prop}

\begin{proof}
  Let $x$ be a point as in the statement. As the maps $y\mapsto v_p(\phi_i(y))+e^{-1}k_i(y)$ are locally constant on $\mathcal{Z}'$, there exists an open subset $U \subset \mathcal E$ containing $x$ such that these are constant on $U \cap \mathcal{Z}'$. Let \[C= e \max(1,\max \{v_p(\phi_1(x)) + e^{-1}k_1(x) + \dots + v_p(\phi_i(x)) + e^{-1}k_i(x)| {i = 2,\dots,n}\})\]
  and $\mathcal{Z}'_C$ the set of points $z\in\mathcal{Z}'$ such that $\vabs{k_{i,\tau}-k_{i+1,\tau}}>C$ for all $1\leq i\leq n-1$ and $\tau \in\Phi$. We claim that any $z \in \mathcal Z'_C \cap U$, $z$ has a non-critical refinement at $v$ (in fact even the refinement given by $\mathcal E$ at $z$ is non-critical)\footnote{Those $z$ are \emph{numerically non-critical} as in \cite[Remark 2.4.6, (ii)]{BC2}}. Let $z \in \mathcal Z_C' \cap U$ and let $F_\bullet \subset D(\rho_{z,v})$ be the refinement defined by $z$. We note $F_i$ its $i$-dimensional part. The (linearized) Frobenius eigenvalues on $F_i$ are $\phi_1(z),\dots,\phi_i(z)$. Assume that the refinement $F_\bullet$ is critial. This means that the there exists some $1\leq i\leq n-1$ and some $\sigma : F_v \hookrightarrow \overline{\QQ_p}$ and such that the gaps of the Hodge-Tate filtration on $\overline{\QQ_p}\otimes_{F_{v,0},\sigma}D_{\cris}(\rho_{z,v})$ are \emph{not} $-k_{1,\sigma}(z),\dots,-k_{i,\sigma}(z)$. Choose $i$ minimal for this property. This implies that the sum of these gaps is greater or equal than $-(k_{1,\sigma}(z) + \dots+ k_{i-1,\sigma}(z) + k_{i+1,\sigma}(z))$. Moreover if $\tau\neq\sigma$, the sum of the gaps of the Hodge-Tate filtration on $\overline{\QQ_p}\otimes_{F_{v,0},\sigma}D_{\cris}(\rho_{z,v})$ is greater or equal than $-(k_{1,\tau}(z) + \dots+ k_{i-1,\tau}(z) + k_{i,\tau}(z))$. As $F_i$ is $\varphi$-stable, the weak admissibility of $D_{\cris}(\rho_{z,v})$ implies that
\[ v_p(\phi_1(z))+\dots + v_p(\phi_i(z)) \geq -\frac{1}{e}\left(\sum_{\tau \neq \sigma}\left(\sum_{j \leq i}k_{j,\tau}(z)\right)  +\sum_{j=1}^{i-1}k_{j,\sigma}(z) +k_{\sigma,i+1}(z)\right),\]
thus
\[ \sum_{j=1}^i \left(v_p(\phi_j(z)) + \frac{1}{e}k_j(z)\right) \geq \frac{1}{e}(k_{i,\sigma}(z)-k_{i+1,\sigma}(z)) > C,\]
contradicting $z \in \mathcal Z_C'$. Therefore $F_\bullet$ is non-critical.
\end{proof}

Let $\mathcal{X}^{\modag}$ be the set of almost generic points in
$\mathcal{F}(\overline{\rho})$.

\begin{theor}\label{theo:agpointsdense}
  The set of points $x$ in $\overline{\mathcal{F}(\overline{\rho})}$
  which are in the image of the set $\mathcal{Z}'$ such that
  $\rho_{x,v}$ is crystalline $\varphi$-generic, HT-regular, admits an non-critical refinement, and such
  that $\rho_x(G_{F_v})$ contains an open subgroup of $\SL(V_x)$ is a
  Zariski dense accumulation subset. Moreover the set
  $\mathcal{X}^{\modag}$ is a Zariski dense accumulation subset in
  $\overline{\mathcal{F}(\overline\rho)}$.
\end{theor}

\begin{proof}
  For any $v | p$, let $\Lambda_1,\dots,\Lambda_r$ be nonzero $\QQ$-linear forms on
  $\QQ^n$ as in Proposition \ref{prop:crit-image-ouverte}. Let
  $\sigma$ be some fixed embedding of $F_v$ into $K$. The set of
  classical points $x\in\mathcal{E}(\overline{\rho})$ which are
  crystalline and such that the $\sigma$-Hodge-Tate weights of the
  representation $\rho_x$ are not zeros of all the $\Lambda_i$, and such that the weights are sufficiently regular in the sense of of \ref{prop:noncritref}, form a
  Zariski dense accumulation subset in $\mathcal{E}(\overline{\rho})$
  (this is a direct consequence of the open image of Proposition
  \ref{prop:geoHecke} ; see \cite[Proposition 2.2.6]{Taibi}  by using $f_i = \delta_{i,v}(\pi_v)$, whose valuation is $v_p(\phi_i) + \frac{1}{e}k_i$). 
  
  As a
  consequence $\overline{\mathcal{F}(\overline{\rho})}$ is the
  Zariski-closure of the images of these points in
  $\mathcal{X}_{\overline{\rho}}^{\pol}$. By Proposition
  \ref{prop:density-absirred} and Proposition \ref{prop:noncritref}, the subspace of
  $\mathcal{X}_{\overline{\rho}}^{\pol}$ where $\rho_x|_{G_{K_n}}$
  is absolutely irreducible is Zariski-open and Zariski-dense in
  $\overline{\mathcal{F}(\overline{\rho})}$. We conclude from
  Proposition \ref{prop:crit-image-ouverte} that the set of classical
  points $x$ such that $\rho_x$ has an open image and posses a non-critical refinement is Zariski-dense and 
  an accumulation subset in $\overline{\mathcal{F}(\overline{\rho})}$.

  It is enough to prove that classical points such that
  $\rho_x(G_{F_v})$ contains an open subgroup of $\SL(V_x)$ have
  enormous image. At such a point $x$, the Zariski closure of
  $\rho_x(G_E)$ contains $\SL(V_x)$. As $E(\zeta_{p^\infty})/E$ is
  abelian, the derived subgroup $\rho_x(G_E)$ is included in
  $\rho_x(G_{E(\zeta_{p^\infty})})$. By \cite[I \textsection 2.1
  (e)]{Bor91}, the Zariski closure of
  $\rho_x(G_{E(\zeta_{p^\infty})})$ contain the derived subgroup of
  the Zariski closure of $\rho_x(G_E)$ and then contains
  $\SL(V_x)$. It follows from \cite{NT} Lemma 2.33, that
  $\rho_x(G_{E(\zeta_{p^\infty})})$ is enormous.

  For any $v'\mid p$, the subset $U_{v'}$ of
  $x\in\overline{\mathcal{F}(\rhobar)}$ such that
  $\Tr\rho_x|_{G_{F_{v'}}}$ is absolutely irreducible is Zariski open
  and dense. We easily deduce the second part of the statement.
\end{proof}

\section{Global and local settings}
\label{sect:localglobal}

We come back to our global situation as in section \ref{sect:aggen}, i.e.~$k$ is a finite field of characteristic $p$ and $\rhobar : G_E\rightarrow\GL_n(k)$ is a conveniently modular (polarized by $\chi=1$) semisimple continuous representation.

Let $x \in \mathcal{X}^{\modag}$. Then $x$ correspond to a cuspidal automorphic representation $\pi$ of $G$. Let $\Pi$ be the isobaric automorphic representation of $\GL_n(\mathbb{A}_E)$ associated to $\pi$ by Theorem \ref{thrA1}.

\begin{prop}\label{prop:generic}
The representation $\Pi$ is cuspidal and thus generic.
\end{prop}

\begin{proof}
We have
\[ \Pi = \Pi_1 \boxplus \Pi_2 \boxplus \dots \boxplus \Pi_k\]
and a character $\chi_{\pi}$ where $\Pi_i$ is a regular algebraic cuspidal automorphic representation of $\GL_{n_i}$ such that $\Pi_i\otimes\chi_\pi^{-1}$ is self-dual. Thus as $\rho_x =  \rho_\pi$ is up-to-twist by a character given by $\rho_\Pi = \rho_{\Pi_1}\oplus\dots\oplus \rho_{\Pi_k}$ but as $x$ is in $X^{\modag}$, $\rho_\pi$ is irreducible, thus $k=1$ and $\Pi$ is cuspidal. In particular it is generic by Piatieski-Shapiro, Shalika (\cite{lecturesLfun} Theorem 8.5).
\end{proof}

\begin{cor}
\label{corgen}
Let $(\rho_x,V_x)$ be the representation corresponding to a point $x\in \mathcal{X}^{mod, ag}$. For all $v \in S$, $v \nmid p$, we have $H^0(G_v,\ad(V_x)^*(1)) = \set{0}$, in particular $H^1(G_v,\ad(V_x)) = H^1_f(G_v,\ad(V_x))$ (see \cite[Notation 2.1]{BellBK}).
\end{cor}

\begin{proof}
Let $\pi$ be the automorphic representation associated to $x$. Since $v$ is split in $E$, the representation $(\rho_x|_{G_v},V_x)$ is a twist of the image of $\pi_v$ by the local Langlands correspondence (see \cite{Caraini_loc_glob_not_p}). By Proposition \ref{prop:generic}, the representation $\pi_v$ is generic thus $H^0(G_v,\ad(V_x)^*(1)) = Hom_{G_v}(V_x,V_x(1)) = \set{0}$, thus $H^1(G_v,\ad(V_x))/H^1_f(G_v,\ad(V_x))$ vanishes too (for example \cite{BellBK} Proposition 2.3 (i)).
\end{proof}

\begin{theor}[Newton-Thorne]
\label{theorNT}
Let $x \in \mathcal{X}^{\modag}$ and let $r_x : G_{F,S}\rightarrow\mathcal{G}_n(\overline{\Qp})$ be the associated representation. Then $H^1_f(G_{F,S},\ad(r_x)) = \set{0}$.
\end{theor}

\begin{proof}
This is consequence of \cite[Thm.~A]{NT}. Namely as $x\in \mathcal{X}^{\modag}$, the representation $r_x$ is associated to an automorphic representation $\pi$ whose base change to $\GL_n(\mathbb{A}_E)$ is cuspidal algebraic and regular by Proposition \ref{prop:generic}.
\end{proof}

Let $x\in \mathcal{X}^{mod, ag}$ and let $(\rho_x,V_x)$ be the associated representation of $G_{E,S}$ over a finite extension of $k(x)$. Our goal is to prove Theorem \ref{thr:globaltgt}, which is invariant by scalar extension, thus we freely extend the base field of $\mathcal{X}_{\overline{\rho}}^{\pol}$ so that $\rho_x$ is defined over $k(x)$. Let $\mathcal{F}$ be a refinement of $(\rho_x,V_x)$, that is, a family $(\mathcal{F}_v)_{v\in S_p}$ where $\mathcal{F}_v$ is a refinement of the crystalline representation $(\rho_x|_{G_v},V_x)$. Let $x_{\mathcal{F}}\in\mathcal{E}$ be the classical dominant point corresponding to $\rho_x$ and the refinement $\mathcal{F}$. In what follows, if $X$ is a rigid space and $x\in X$, we set $\widehat{X}_x\coloneqq\Spec(\widehat{\mathcal{O}_{X,x}})$. The projection map $\mathcal{E}\rightarrow\mathcal{X}_{\rhobar}$ induces a morphism $\widehat{\mathcal{E}}_{x_{\mathcal{F}}}\rightarrow\widehat{\left(\mathcal{X}_{\overline{\rho}}^{\pol}\right)}_x \simeq \mathfrak{X}_{\rho_x}^{\pol}$. For each $v\in S_p$, let $\rho_{x,v}=\rho_x|_{G_v}$, which is irreducible, and consider the composite map
\[ \widehat{\mathcal{E}_{x_{\mathcal{F}}}}\longrightarrow\widehat{\left(\mathcal{X}_{\overline{\rho}}^{\pol}\right)}_x \simeq \mathfrak{X}_{\rho_x}^{\pol} \longrightarrow\mathfrak{X}_{\rho_{x,v}}.\]
Using notation introduced in section \ref{sect:localdefrings}, we define \textit{quasi-trianguline} deformation spaces $\mathfrak{X}_{\rho_v,\mathcal{F}}^{\qtri,w} := \mathfrak{X}_{D_{rig}(\rho_v),\mathcal{F}[1/t]}^w$ for $w\in\mathfrak{S}_n^{[F_v:\QQ_p]}$ . Denote also $\mathfrak{X}_{\rho_v}^{\cris} := \mathfrak{X}_{D_{rig}(\rho_v)}^{\cris}$ the crystalline deformation space (see \cite[page 24]{CheFern} or \cite[page 18]{HMS}).
\begin{lemma}
The map $\widehat{\mathcal{E}_{x_{\mathcal{F}}}}\rightarrow 
\mathfrak{X}_{\rho_{x,v}}$ factors through $\mathfrak{X}_{\rho_{x,v},\mathcal{F}_v}^{\qtri,w_0}$.
\end{lemma}

\begin{proof}
Let $\rhobar_v : G_{F_v}\rightarrow\GL_n(k)$ be the composite of $\rhobar$ with $G_{F_v}\hookrightarrow G_{E,S}$ and let $\mathfrak{X}_{\rhobar_v}^\Box$ be the framed deformation space of $\rhobar_v$. Let $X_{\tri}(\rhobar_v)\subset\mathfrak{X}_{\rhobar_v}^{\Box,\rig}\times\widehat{F_v^\times}^n$ be the trianguline variety \cite[Def.~2.4]{BHS1}. We choose $y\in\mathcal{X}_{\rhobar_v}^{\Box}\coloneqq\mathfrak{X}_{\rhobar_v}^{\Box,\rig}$ be a point such that $\rho_y$ is conjugated to $\rho_{x,v}$ and let $y_{\mathcal{F}_v}$ be the dominant point of $X_{\tri}(\overline{\rho}_v)$ corresponding to $y$ and to the refinement $\mathcal{F}_v$. The projection map $X_{\tri}(\rhobar_v)\rightarrow\mathfrak{X}_{\rhobar_v}^{\Box,\rig}$ induces a map $\widehat{X_{\tri}(\rhobar_v)}_{y_{\mathcal{F}_v}}\rightarrow\mathfrak{X}_{\rho_y}^\Box$ and, by \cite[Cor.~3.7.8]{BHS}, this morphism factors through $\mathfrak{X}_{\rho_y,\mathcal{F}_v}^{\Box,w_0}$. As $\mathfrak{X}_{\rho_y,\mathcal{F}_v}^{\Box,w_0}$ is the pullback of $\mathfrak{X}_{\rho_y,\mathcal{F}_v}^{w_0}$ by the formally smooth map $\mathfrak{X}_{\rho_y}^\Box\rightarrow\mathfrak{X}_{\rho_y}$, it is sufficient to prove that there exists, locally at $x$, a factorization 
\[ \begin{tikzcd} \mathcal{E} \ar[r,dotted] \ar[rd] & X_{\tri}(\rhobar_v) \ar[d] \ar[r] & \mathcal{X}_{\rhobar_v}^{\Box}\ar[ld] \\ & \mathcal{X}_{\overline{\rho}_{v}}\end{tikzcd}\]
sending $x_{\mathcal{F}}$ on $y_{\mathcal{F}_v}$, where $\mathcal{X}_{\overline{\rho}_{v}}$ is the rigid fiber of the pseudo-deformation space, as in Definition \ref{defin:psdef}.
As $\rho_{x_\mathcal F}$ is irreducible, it follows that there exists some affinoïd neighborhood $U$ of $x_{\mathcal{F}}$ in $\mathcal{E}$ and a continuous morphism $\rho_V : G_{E,S}\rightarrow \GL_n(\mathcal{O}(U))$ such that $\Tr(\rho_U)(z)=\Tr(\rho_z)$ for all $z\in U$. Indeed, by \cite{Chedet} Theorem 2.22 there is a representation $\rho_A : G_{E,S} \rightarrow \GL_n(\mathcal{O}_{\mathcal{E},x_\mathcal{F}})$ whose trace is $D_{| \mathcal{O}_{\mathcal{E},x_\mathcal{F}}}$. As $\mathcal{O}_{\mathcal{E},x_{\mathcal{F}}}$ is a direct limit over $U$, there exists such a $U$ (see \cite{BC2} Lemma 4.3.7 for a precise argument).
This gives us a map $U\rightarrow \mathcal{X}_{\rhobar_v}^{\Box}$ and even $U\rightarrow \mathcal{X}_{\rhobar_v}^{\Box} \times \widehat{F_v^\times}^n$. As the set $\mathcal{Z}'$ is Zariski-dense and accumulation in $\mathcal{E}$, we can choose $U$ so that $U\cap\mathcal{Z}'_{}$ is Zariski-dense in $U$. A point of $U\cap\mathcal{Z}'_{}$ is sent to a point of $X_{\tri}(\rhobar)$ by $U\rightarrow\mathcal{X}_{\rhobar_v}^\Box\times\widehat{F_v^\times}^n$ (by definition of $X_{\tri}(\overline{\rho})$, \cite{BHS} section 3.7) so that we obtain the desired section.
\end{proof}

For any representation $\rho$ of $G_E$, we use the notation $\rho_p \coloneqq (\rho_v)_{v \in\Phi}$ where $\Phi$ is in our fixed CM type at the beginning of section \ref{se:eigenvarieties}. Then we write $\mathfrak{X}^{}_{\rho_p} := \prod_{v | p} \mathfrak{X}^{}_{\rho_v}, \mathfrak{X}^{\qtri,w}_{\rho_p,\mathcal{F}} := \prod_{v | p} \mathfrak{X}^{\qtri,w}_{\rho_v,\mathcal{F}_v}$ and $\mathfrak{X}^{\cris}_{\rho_p} := \prod_{v | p} \mathfrak{X}^{\cris}_{\rho_v}$.

The following corollary is very similar to \cite{Bergdall} and \cite{BHS}.
\begin{cor}	
\label{cortantri}
Let $x\in \mathcal{X}^{mod, ag}$ and let $\mathcal{F}$ be a refinement such that $(D_{\rig}(\rho_{x,v}),\mathcal{F}_v)$ is associated to a product of distinct transpositions for all $v\in\Phi$ (see definition \ref{def:transpo_simples}). Then $x_{\mathcal{F}}$ is a smooth point of $\mathcal E$ and we have an isomorphism
\[ T_{x_\mathcal{F}} \mathcal{E} \overset{\sim}{\fleche} T \mathfrak{X}^{\qtri,w_0}_{\rho_{x,p},\mathcal{F}}/T \mathfrak{X}^{\cris}_{\rho_{x,p}}.\]
\end{cor}

\begin{proof}
Denote by $\mathfrak{X}^{\pol}_{\rho_x}$ the (equicharacteristic) $\chi$-polarised global deformation space of $\tr\rho_x$. It is the completion of $\mathcal{X}^{\pol}_{\overline{\rho}}$ at $\rho_x$ by \cite[section 4.1]{Chedet}. Denote by $\mathfrak{X}_{\rho_x,\mathcal{F}}^{\qtri}$ the fiber product $\mathfrak{X}_{\rho_x}^{\pol} \times_{\mathfrak{X}_{\rho_{x,p}}} \mathfrak{X}^{\qtri,w_0}_{\rho_{x,p},\mathcal{F}}$. We have a map
\[ \widehat{\mathcal{E}_{x_\mathcal{F}}} \fleche \mathfrak{X}_{\rho_x,\mathcal{F}}^{\qtri},\]
induced from the map $\widehat{\mathcal{E}_{x_\mathcal{F}}} \fleche \mathfrak{X}_{\rho_{x,v},\mathcal{F}_v}^{\qtri,w_0}$ and $\widehat{\mathcal{E}_{x,\mathcal{F}}} \fleche \mathfrak{X}_{\rho_x}^{\pol}$. But then the standard argument that 
\[ f : \mathcal{O}_{\mathfrak{X}_{\rho_x,\mathcal{F}}^{\qtri}} \fleche \widehat{\mathcal{O}_{\mathcal{E}_{x,\mathcal{F}}}},\]
is surjective comes from the fact that $\mathcal E_{x_{\mathcal F}}$ is topologically generated by $\mathcal O_{\mathfrak X_{\rho_x}^{\pol}}$ and $\mathcal O_{\mathcal T}$ by construction, but $\mathfrak X_{\rho_{x,p},\mathcal F}^{\qtri,w_0}$ lies over $\mathcal T$.
Thus we have a closed immersion
\begin{equation}
  \label{eq:closed_immersion}
  \widehat{\mathcal{E}_{x,\mathcal{F}}} \hookrightarrow \mathfrak{X}_{\rho_x,\mathcal{F}}^{\qtri}.
\end{equation}
But the genericity assumption (Corollary \ref{corgen}) implies that 
\[H^1_f(G_{F},ad(\rho_x)) = \ker\left(H^1(G_{F},ad(\rho_x)) \fleche \prod_{v \mid p} H^1(G_{F_v},ad(\rho_x))/H^1_f(G_{F_v},ad(\rho_x))\right),\] thus we have an
the exact sequence
\begin{equation}
  \label{eq:h1f}
  0 \fleche H^1_f(G_F,ad(\rho_x))\cap T{\mathfrak{X}_{\rho_x,\mathcal{F}}^{\qtri}} \fleche T{\mathfrak{X}_{\rho_x,\mathcal{F}}^{\qtri}} \fleche \bigoplus_{v| p } T{\mathfrak{X}_{\rho_{x,v},\mathcal{F}_v}^{\qtri,w_0}}/T{\mathfrak{X}^{\cris}_{\rho_{x,v}}}.
\end{equation}

Moreover we have the following inequalites
\[ \dim T_{(x,\mathcal{F})}\mathcal{E} \leq \dim T{\mathfrak{X}_{\rho_x,\mathcal F}^{\qtri}} \leq n[F:\QQ].\]
The first one is a consequence of \eqref{eq:closed_immersion}. The last one is a consequence of the fact that we can compute all the dimensions in the exact sequence \eqref{eq:h1f}. Namely the Theorem \ref{theorNT} of Newton-Thorne assures that
\[ H^1_f(G_F, ad(\rho_x)) = \{0\}.\] The dimension of $T{\mathfrak{X}_{\rho_{x,v},\mathcal{F}_v}^{\qtri,w_0}}$ comes from \cite{BHS} Corollary 3.7.8 and Remark 4.1.6 (i) as $(\rho_{x,v},\mathcal{F}_v)$ is associated to a product of distinct transpositions for all $v\in\Phi$, and the dimension of $T{\mathfrak{X}^{\cris}_{\rho_{x,v}}}$ comes from \cite{Kisindef} Theorem 3.3.8 (we warn the reader that in these references framed deformation rings are considered, which are formally smooth of relative dimension $n^2-1$ over our rings).

As $\mathcal{E}$ is equidimensional of dimension $n[F:\QQ]$, we have $\dim T_{x_\mathcal{F}}\mathcal{E} = n[F:\QQ]$ and thus $x_\mathcal{F}$ is a smooth point of $\mathcal{E}$ and we have
\[ \widehat{\mathcal{E}_{x_\mathcal{F}}} \overset{\sim}{\fleche}  \mathfrak{X}_{\rho_x,\mathcal{F}}^{\qtri},\]
and thus \[ T_{x_\mathcal{F}}\mathcal{E} \simeq T{\mathfrak{X}_{\rho_x,\mathcal{F}}^{\qtri}} \simeq \bigoplus_{v| p } T{\mathfrak{X}_{\rho_{x,v},\mathcal{F}_v}^{\qtri,w_0}}/T{\mathfrak{X}^{\cris}_{\rho_{x,v}}}.\qedhere\]
\end{proof}

\begin{theor}\label{thr:globaltgt}
For $x \in \mathcal{X}^{\modag}$, the image of the natural map
\[ \bigoplus_{\mathcal{F}} T_{x_\mathcal{F}}\mathcal{E} \fleche T_x\mathcal{X}_{\overline{\rho}}^{\pol},\]
has dimension at least $\frac{n(n+1)}{2}[F:\QQ]$, where $\mathcal{F}$ runs over the $n![F:\QQ]$ refinements of $x$.
\end{theor}

\begin{proof}
Let $\mathcal{F}^\nc=(\mathcal{F}_{v}^\nc)_{v\in\Phi}$ be a refinement of $\rho_x$ such that $\mathcal{F}_{v}^\nc$ is non critical for any $v\in\Phi$. It exists by definition of $\mathcal X^{mod,ag}$. For all $(c_v)\in\mathfrak{C}_n^\Phi$, let $\mathcal{F}=(c_v\cdot\mathcal{F}_{v}^\nc)$. By Proposition \ref{propraf} (\ref{propraf2}) the pair $(D_{\rig}(\rho_{x,v}),c_v\cdot\mathcal{F}_{v}^\nc))$ is associated to a product of simple transposition for all $v\in\Phi$. Thus Corollary \ref{cortantri} implies that
\[ T_{x_\mathcal{F}}\mathcal{E} \simeq 
T \mathfrak{X}^{\qtri,w_0}_{\rho_{x,p},\mathcal{F}}/T \mathfrak{X}^{\cris}_{\rho_{x,p}}
,\]
and moreover, for all $v | p$, the map
\[ \bigoplus_{c \in \mathfrak C_n} 
T \mathfrak{X}^{\qtri,w_0}_{\rho_{x,v},c_v \cdot \mathcal{F}_{v}^\nc}
\fleche T\mathfrak{X}_{\rho_v},\]
is surjective by Proposition \ref{propraf}. In particular, the map
\[ \bigoplus_{c = (c_v)_v \in \mathfrak C_n^\Phi} T_{(x,c\cdot\mathcal{F}^\nc)}\mathcal{E} \fleche T_x\mathcal{X}_{\overline{\rho}}^{\pol} = T\mathfrak{X}_{\rho_x}^{\pol} \fleche T \mathfrak{X}_{\rho_p}/T \mathfrak{X}_{\rho_p}^{\cris},\]
which can also be factored,
\[ \bigoplus_{c \in \mathfrak C_n^\Phi} T_{(x,c\cdot\mathcal{F}^\nc)}\mathcal{E} \fleche \bigoplus_{v | p} \bigoplus_{c_v \in \mathfrak C_n} T \mathfrak{X}^{\qtri,w_0}_{\rho_{x,v},c_v\cdot\mathcal{F}_{v}^\nc}/T \mathfrak{X}^{\cris}_{\rho_{x,v}} \overset{\sum_\mathcal{F}}{\fleche} T \mathfrak{X}_{\rho_p}/T \mathfrak{X}_{\rho_p}^{\cris},\]
is surjective by Corollary \ref{cortantri} and Proposition \ref{propraf}, and thus has rank at least $\frac{n(n+1)}{2}[F:\QQ]$, thus the same is true for the map,
\[\bigoplus_{\mathcal{F}} T_{(x,\mathcal{F})}\mathcal{E} \fleche  T_x\mathcal{X}_{\overline{\rho}}^{\pol}. \qedhere\]
\end{proof}

\begin{rema}
Note that we don't actually need all the refinements (for a fixed $v$), only the $1+\frac{n(n-1)}{2}$ refinements given by $c_v = c_{i,j} := (i,i-1,\dots,j) \in \mathfrak S_n$ with $i \geq j$ (starting from a non-critical one). But this is still more than just the $n$ well-positioned refinements for weakly generic points of Chenevier \cite{CheFern}, even for $n=3$.
\end{rema}

\begin{theor}
\label{thr:dimrigid}
Let $\overline{\mathcal{F}(\overline{\rho})} \subset \mathcal{X}^{\pol}_{\overline{\rho}}$ be the Zariski closure of the image of $\mathcal{E}(\overline{\rho})$. Then $\overline{\mathcal{F}(\overline{\rho})}$ is equidimensional of dimension $\frac{n(n+1)}{2}[F:\QQ]$, and is a union of irreducible components of $ \mathcal{X}^{\pol}_{\overline{\rho}}$.
\end{theor}

\begin{proof}
We have already proven that almost generic point are Zariski dense in $\overline{\mathcal{F}(\overline{\rho})}$ (see Theorem \ref{theo:agpointsdense}). We will prove, following \cite{All1}, that these points are 
smooth points of $\mathfrak{X}_{\overline{\rho}}$ whose local ring are $\frac{n(n+1)}{2}[F:\QQ]$-dimensional. Let $x$ be such a almost generic point in $\mathcal{F}(\overline\rho)$, thus $\rho_x$ is irreducible and we can consider the polarised deformation space $\mathfrak{X}_{\rho_x}^{\pol}$. Then by an argument of Kisin (see \cite{Kisinfflat} and \cite[Thm.~1.2.1]{All1}),
\[ \mathfrak{X}_{\rho_x}^{\pol} \simeq \widehat{(\mathcal{X}_{\overline\rho}^{\pol})_x}.\]
Thus we need to show that $\mathfrak{X}_{\rho_x}^{\pol}$ is (formally) smooth of dimension $\frac{n(n+1)}{2}[F:\QQ]$. But as $\rho_x$ is absolutely irreducible we can choose a lift $r_x$ to $\mathcal G_n$ and by Proposition \ref{prop:defrrho} reduce to $\mathfrak{X}_{r_x}$. Remark here that because of Proposition \ref{prop:generic} and Theorem \ref{thr:BCsign}, we can apply Proposition \ref{prop:dimdefr}. Calculations on the dimension of deformation ring made in Proposition \ref{prop:dimdefr} show that 
we are thus reduced to show that $h^2(G_{F,S},\ad(r_x)) = 0$, or what is equivalent $h^1(G_{F,S},\ad(r_x)) = \frac{n(n+1)}{2}[F:\QQ]$. But as $\rho_x$ is generic at $p$ by Proposition \ref{prop:generic}, by Remark 1.2.9 of \cite{All1}, we get $H^1_g(G_{F,S},\ad(\rho_x)) = H^1_f(G_{F,S},\ad(\rho_x))$ which vanishes by Newton-Thorne's Theorem \ref{theorNT}. Thus the following map is injective,
\[ H^1(G_{F,S},\ad(\rho_x)) \fleche \prod_{v | p} H^1(F_v, \ad(\rho_x))/H^1_g(\ad(\rho_x)).\]
But then we prove exactly as in \cite{All1}, Lemma 1.3.5, as our $x$ is HT-regular, that the space $\mathfrak{X}_{r_x}$ is formally smooth of dimension $\frac{n(n+1)}{2}[F:\QQ]$, thus by Proposition \ref{prop:dimdefr} $\mathfrak{X}_{\rho_x}^{\pol}$ is formally smooth of dimension $\frac{n(n+1)}{2}[F:\QQ]$, but as it contains the local ring of the closure of $\mathcal{F}(\overline{\rho})$ at $\rho_x$, which is of dimension $\geq \frac{n(n+1)}{2}[F:\QQ]$ by Theorem \ref{thr:globaltgt}, both local rings are equal (and $\overline{\mathcal{F}(\overline{\rho})}$ is smooth at these points).
\end{proof}

\begin{rema}
Recall that in the previous theorem $\mathcal{E}(\overline{\rho})$ and thus $\overline{\mathcal{F}(\overline{\rho})}$ depend on the choice of an auxiliary level $K^p$ outside $p$. We can ask how the closure of the infinite fern depends on $K^p$. If we let $K^p$ appear in the notations, we can at least have an optimal $K^p$.
\end{rema}

\begin{cor}
\label{cor:Koutsidep}
There exists a level $K^p$ outside $p$, such that for all level $K^{\prime,p}$ outside $p$, the Zariski closure of the infinite fern of tame level $K^p$, $\overline{\mathcal{F}_{K^p}(\overline{\rho})}$, contains the infinite fern of level $K^{\prime,p}$, $\overline{\mathcal{F}_{K^{\prime,p}}(\overline \rho)}$.
\end{cor}

\begin{proof}
As $\mathcal X_{\overline\rho}^{\pol}$ is the generic fiber of a noetherian excellent formal scheme, it has a finite number of connected component (See \cite[Theorem 2.3.1]{ConradIrred}) . Thus as the number of components in the closure of the infinite fern (by Theorem \ref{thr:dimrigid}) grows with $K^p$, it eventually stabilizes.
\end{proof}

\begin{cor}\label{cor:main}
  Let $\chi : G_F\rightarrow\mathcal{O}^\times$ be a continuous
  character satisfying Hypothesis \ref{hypchi}. Let
  $\rhobar : G_E\rightarrow\GL_n(k)$ be a semi-simple polarized-by-$\overline{\chi}$ 
  conveniently modular continuous
  representation. The Zariski closure of the set points of
  $\mathcal{X}_{\overline{\rho}}^\chi$ which are holomorphically
  modular and crystalline at $p$ is a union of irreducible components
  of dimension $\tfrac{n(n+1)}{2}[F:\QQ]$.
\end{cor}

\begin{proof}
  This is a direct consequence of Lemma \ref{lemma:chi=1}, Theorem
  \ref{thr:dimrigid} and Corollary \ref{cor:finite_slope}.
\end{proof}

We can now deduce the following corollary, which is due to Allen, \cite{All2}, for which we need to take care that automorphic points given by \cite{All2} main's theorem are indeed inside our infinite fern. So we assume the following hypothesis as in \cite{All2},
\begin{hypothese}\label{HypAllen+}
\begin{enumerate}
\item\label{hyp812.1} $p > 2$, is unramified in $E$ and every prime $v | p$ in $F$ splits in $E$. Moreover, $\zeta_p \notin E$.
\item\label{hyp812.2} $\overline\rho(G_{E(\zeta_p)})$ is adequate, $\overline\rho$ is polarized by $\chi$ i.e.~ $\overline\rho^\vee \simeq \overline\rho^c \otimes \chi \eps^{n-1}$.
\item\label{HypAllen3} There exists a $GL_n$-automorphic representation $\Pi_0$, which is regular algebraic $\chi$-polarized cuspidal, such that $\rho_{\Pi_0}$ lifts $\overline\rho$ and such that $\rho_{{\Pi_0},v}$ is potentially diagonalisable for all $v|p$, and even ordinary for all $v | p$ if $p | n$.
\item\label{hyp812.4} $\chi$ satisfies $\chi = \chi^c$ and satisfies a sign condition (see Hypothesis \ref{hypchi} and section \ref{sect:signs})
\item\label{hyp812.5} $H^0(G_v,\ad(\overline\rho)(1))=0$ for all $v | p$.
\end{enumerate}
\end{hypothese}

We still hope that hypothesis \ref{hyp812.1}. is technical and we hope to be able to remove it, as for Theorem \ref{thr:dimrigid}. It is unkown at the moment if all potentially crystalline representations are potentially diagonalisable, i.e.~ if we could relax hypothesis \ref{HypAllen3}. to a classical modularity (for $GL_n$, crystalline at $p$ say).
We hope that hypothesis \ref{hyp812.2}. is unnecessary, but at the moment the main result of \cite{All2} relies on it, and also on \ref{hyp812.5}. We imagine that these
could be removed using new results on local deformations rings (e.g. \cite{BIP}).

Recall that we had constructed the infinite fern inside $\mathcal X_{\overline\rho}^{\pol}$, when $\chi = 1$. Let $\overline\rho$ as above, and let $\psi_0: G_{E,S} \fleche (\mathcal O')^\times$ as in \cite[Lemma 4.1.5]{CHT} and section \ref{sect:defspaces}, so that $\psi_0\psi_0^c = \chi$. Up to base change by $\mathcal O'[1/p]$, we can and do assume everything is defined over $\mathcal O[1/p]$. As $\overline\rho$ is $GL_n$-automorphic for $\Pi_0$ by the above assumption, we have $\overline\rho_0 := \overline\rho\psi_0$ is also $GL_n$-automorphic (for $\Pi_0 \psi_0\circ \Art_E$). If we assume that $\overline\rho_0$ is conveniently modular for our similitude unitary group $G$ as in section \ref{se:eigenvarieties}, then by Lemma \ref{lemma:chi=1} we have the following diagram
\begin{center}
\begin{tikzpicture}[description/.style={fill=white,inner sep=2pt}] 
\matrix (m) [matrix of math nodes, row sep=3em, column sep=2.5em, text height=1.5ex, text depth=0.25ex] at (0,0)
{ 
\mathcal E(\overline\rho_0) &  &  \\
\mathcal X^{\pol}_{\overline\rho\psi_0} &   & \mathcal X_{\overline\rho}^{\chi-\pol}\\
};
\path[->,font=\scriptsize] 
(m-1-1) edge node[auto] {$pr_1$} (m-2-1)
(m-2-1) edge node[auto] {$\psi_0^{-1}$} (m-2-3)
(m-1-1) edge node[auto] {$f$} (m-2-3)
;
\end{tikzpicture}
\end{center}
where $f:=\psi_0^{-1}\circ pr_1$. We call infinite fern, denoted by $\mathcal F(\overline\rho)$ the image of $\mathcal E(\overline\rho)$ by the diagonal map (Note that it a priori depends on $\overline\rho_0$, thus $\psi_0$).
We have the following

\begin{cor}[Allen] Assume the hypothesis \ref{HypAllen+}. 
Then $\overline\rho_0 := \overline\rho\psi_0$ is conveniently modular, the generic fiber of the global deformation ring $\mathcal X^{\chi-\pol}_{\overline\rho}$ is equidimensional of dimension $[F:\QQ]\frac{n(n+1)}{2}$ and the infinite fern $\mathcal F(\overline\rho)$ is Zariski dense in $\mathcal X^{\chi-\pol}_{\overline\rho}$ thus in $\Spec(R_{\overline\rho}^{\chi-\pol})$. In particular automorphic points are dense in $\Spec(R_{\overline\rho}^{\chi-\pol})$.
\end{cor}

\begin{proof}
As the set of Hypothesis 8.12 contains strictly the hypothesis of Theorem \ref{thr:dimrigid}, we have that the Zariski closure of $\mathcal F(\overline\rho)$  (if non-empty!) is a union of connected components of $\mathcal X_{\overline\rho}^{\chi-\pol}$. Thus, it is enough to prove that each component of $\mathcal X_{\overline\rho}^{\chi-\pol}$ contains a points in the infinite fern, and by the reduction of Lemma \ref{lemma:chi=1} and considering $\Pi_0\psi_0$ where $\psi_0$ given by \cite[Lem.4.1.5]{CHT} , we can assume $\chi=1$. By \cite{All2}, Corollary 5.3.3, we have that $R_{\overline\rho}^{\pol}$ is $\mathcal O$-flat, reduced, and complete intersection of the expected dimension, but we still need to check that the automorphic point in all components can be chosen to be in the infinite fern (i.e.~ holomorphic at infinity automorphic representations for $GU$). 
Let $\mathcal C$ be an irreducible component of $\mathcal X_{\overline\rho}^{\pol}$, which is of the form $\mathcal C = C^{rig}$ for an irreducible component $C$ of $\Spec(R_{\overline\rho}^{\pol})$ (\cite[Lem.~1.2.3]{All2}).
By \cite[Thm.~5.3.1 \& 5.3.2]{All2}, there is a
$GL_n$-automorphic cuspidal point $\Pi$ in $\mathcal C$, which is
moreover unramified at places above $p$, very regular, self dual, and
such that $\Pi$ is a smooth point of $\mathfrak
X_{\overline\rho}^{\pol}$ (see \cite[Thm.~C]{All1}). Let $\Pi_U$ be the global A-packet associate to $\Pi$ (see
\cite[2.3]{Mok} for example). Let $v | \infty$ be an archimedean place of $F$ and let
$\psi_v$ be the local Arthur parameter at $v$ associated to $\psi$. There exist $a_i,
b_i\in\CC$ such that $a_i-b_i\in\ZZ$ for $1\leq i\leq n$ and
$\psi_v|_{\CC^\times}$ is of the form:
\[ z \longmapsto
  (z^{a_1}\overline{z}^{b_1},\dots,z^{a_n}\overline{z}^{b_n}). \] As
$\Pi$ is regular, algebraic and cuspidal, Clozel's Purity Lemma
\cite[Lem.~4.9]{Clozel} implies that there exists $w\in\RR$ such that
$a_i+b_i=w$ for all $1\leq i\leq n$. As we have
$j \in W_\RR \backslash \CC^\times$ such that
$jzj^{-1} = \overline z$, we can check that $\psi_v|_{\CC^\times}$ is
conjugate to
$z\mapsto(z^{-b_1}\overline{z}^{-a_1},\dots,z^{-b_n}\overline{z}^{-a_n})$. As
the weight is regular this implies that there exists
$\sigma \in \mathfrak S_n$ an involution such that
$a_{\sigma(i)} = -b_i$ so that $w=0$ and $\psi_v$ is
tempered. Moreover the regularity of the weight implies also that
$\psi_v$ is discrete (see
\cite[Rem.~A.11.8]{BC2}). 
As $\psi_v$ is tempered, discrete, regular algebraic its associated
local A-packet is equal to the L-packet of discrete series
representations constructed in \cite{LanglandsR}, and in particular
contains the holomorphic discrete series. In particular, by
\cite[Thm.~2.4.2]{Mok}, there exists $\pi_0$ a regular holomorphic
algebraic, unramified above $p$ representation of the quasi-split
unitary group $U$ in $\Pi_U$. As $\Pi$ is cuspidal, the A-packet
$\Pi_U$ is stable and thus $\mathcal S_\psi = 1$ (see
\cite[2.2]{Rogmodp}), so that $\pi_0$ is also discrete and
automorphic. As $\pi_0$ is tempered at infinity, it follows from the
main result of \cite{Wallach} that $\pi_0$ is cuspidal.

Choose an algebraic extension of its central character which is unramified at $p$, then by \cite{Labesse-Schwermer} there exists an extension of $\pi_0$ to a cuspidal, regular algebraic representation $\pi$ of $GU$. Moreover $\pi_\infty$ is also the holomorphic discrete series thus contribute to coherent cohomology in degree 0 and thus gives a point in the Eigenvariety $\mathcal E$ (for $GU$), whose Galois representation (given in Corollary \ref{cor:A5}) is $\rho$. In particular, $\overline\rho$ is conveniently modular, $\mathcal F(\overline\rho)$ intersects $\mathcal C$, and the corollary is proved.
\end{proof}
\appendix

\section{Similitude Unitary groups, Tori, Base Change and Galois representations}
\label{AppC}

Fix an isomorphism $\iota : \CC \simeq \overline{\QQ_p}$. Let $E$ be (complex) CM number field and $F$ its maximal totally real subfield.

\begin{defin}
\label{SatakeGLn}
Let $n\geq1$ be an integer. Let $\Pi$ be an automorphic representation of $\GL_n(\mathbb{A}_E)$. Let $\rho : G_E \fleche \GL_n(\overline{\QQ_p})$ be a continuous semsimple representation. We say that $\rho$ is strongly (resp. weakly) associated to $\Pi$ if for 
 all finite place $v$ of $E$ (resp. for almost all $v$) not dividing $p$ and such that $\Pi_v$ is unramified, $\rho$ is unramified at $v$ and the semisimple class of 
$\rho(\Frob_v)$ coincides with $\iota Sat_v(\Pi_v)$ where $Sat_v(\Pi_v)$ is the Satake parameter of $\Pi_v$.
\end{defin}

By \cite{HLTT} and \cite{SchTor}, for any cuspidal regular algebraic automorphic representation $\Pi$, there is a unique $\rho_{\Pi,\iota}$ which is strongly associated to $\Pi$ and for any $\rho$ there is at most one $\Pi$ such that $\rho$ is weakly associated to $\Pi$.

Denote by $G = GU(V)$ a similitude unitary group over $\QQ$ (with similitude factor in $\QQ^\times$) associated to the CM extension $E/F$, and by $Z \simeq GU(E)$ its center.

Let $\ell$ be a rational prime, unramified in $E$, which is also unramified for $GU(V)$ (i.e.~ $GU(V)_{\QQ_p}$ is quasi-split, and split over an unramified extension). Let $\pi$ be a cuspidal automorphic representation of $G$, assume $\pi$ is unramified at $\ell$, and choose a maximal compact $K$ at $\ell$ for which $\pi$ is unramified. Then $\pi_\ell^K$ is a 1-dimensional representation of $H_\CC(G(\QQ_\ell),K)$, the Hecke algebra of bi-$K$-invariant $\CC$-valued functions on $G(\QQ_\ell)$ with compact support. The Satake isomorphism and the unramified local Langlands correspondance (\cite{BorelCorvallis}) associate to it an unramified representation with values in the $L$-group of $G$ (actually in the $L$-group of the maximal torus of $G$). Denote $T_U, T$ the maximal torus of $U=U(V)$ and $G$ respectively. The natural inclusion $T_U \subset T$, which is compatible with the Galois action and central, gives a map 
\begin{equation} \label{eqTT_U} ^{L}T \fleche ^{L}T_U.\end{equation}

\begin{prop}
There is a natural map of $L$-groups
\[ {^L}T_U \fleche {^L}T_{GL_{n,E}}  := \mathbb G_m^{n,\Hom(E,\CC)} \rtimes W_{\QQ_\ell}.\]
Denote by
\[ r : W_{\QQ_\ell} \fleche {^L}T_U,\]
the unramified Langlands parameter associated to $\pi_\ell$ as above. 
For all $\lambda | \ell$ in $F$ and $\lambda' | \lambda$ in $E$, the restriction of $r$ to $W_{E_\lambda'}$ followed by the previous map
\[ r : W_{E_{\lambda'}} \fleche {^L}T_{GL_{n,E}} ,\]
induces a well-defined class (up to conjugacy),
\[ r_{\lambda'} : W_{E_\lambda'} \fleche \GL_n(\CC).\]
\end{prop}

\begin{proof}
As $\ell$ is unramified for $E$ and $G$ (thus $U$), actually $U_{\QQ_p}$ is isomorphic to $U(n)_{E/F,\QQ_p}$ for any choice of (unramified) unitary group of rank $n$, so choose the one with anti-diagonal matrix form. With this form, we check that actually the upper triangular Borel is indeed a Borel over $\QQ_p$, with maximal torus the diagonal one, given by
\[ T_U = \{ Diag(a_1,\dots,a_n) | a_i \in E^\times, c(a_i)a_{n+1-i} = 1\} \subset T_{GL_{n,E}},\]
with $c \in \Gal(E/F)$ the complex conjugacy and $T_{\GL_{n,E}}$ the diagonal torus of $\Res_{E/\QQ}\GL_n$. Denote by $\Sigma_E$ the complex embeddings of $E$. We then have that its characters  $X^*(T_U)$ are given by the quotient of $(\ZZ^n)^{\Sigma_E} (= X^*(T_{GL_{n,E}}))$ by the relation
$(\lambda_{i,\sigma})_{i,\sigma} = (\lambda_{n+1-i,\sigma c}^{-1})_{i,\sigma}$. Its cocharacters 
$X_*(T_U) \subset (\ZZ^n)^{\Sigma_E} = X_*(T_{GL_{n,E}})$ are given by the collections $(\mu_{i,\sigma})_{i,\sigma}$ satisfying $\mu_{i,\sigma} = \mu_{n+1-i,\sigma c}^{-1}$. The Galois action of $\sigma \in G_\QQ$ sends the character $\lambda_{\tau,i}$ to $\sigma \cdot \lambda_{\tau,i} := \lambda_{\tau,i} \circ \sigma^{-1} = \lambda_{\sigma\tau,i}$. It sends the cocharacter $\mu_{\tau,i}$ to $\sigma \circ \mu_{\tau,i} = \mu_{\sigma\tau,i}$. Thus the dual torus is given by the subtorus of $\prod_{\Sigma_E}\mathbb G_m^n$, given by
\[ \widehat{T_U} = \{ (t_1^\sigma,\dots,t_n^\sigma)_\sigma | t_\sigma^i t_{\sigma c}^{n+1-i} = 1\}.\]
The action of $W_{\QQ_\ell}$ on $\widehat{T_U}$ is given by $s \cdot (h_\sigma)_\sigma = (h_{s^{-1}\sigma})_\sigma$. A priori $E$ is not Galois over $\QQ$. 
An analogous computation for the maximal (diagonal) torus of $\Res_{E/\QQ} \GL_n$, $T_{GL_{n,E}}$, gives
\[ ^{L}T_{\GL_{n,E}} = \mathbb G_m^{n,\Sigma_E} \rtimes W_{\QQ_\ell}, \quad s\cdot (t_{i,\sigma}) = (t_{i,s^{-1}\sigma}).\]
and we thus have a natural map $^{L}T_U \mapsto {^L}T_{\GL_{n,E}}$. As $\pi_\ell$ is unramified, by \cite{BorelCorvallis} there is a parameter $r_G : W_{\QQ_\ell} \fleche {^L}T$, which we can compose to get
\[ r : W_{\QQ_\ell} \fleche {^L}T_U,\]
and by the previous map we get and unramified Langlands parameter $r_{\GL_{n,E}} : W_{\QQ_\ell} \fleche {^L}T_{\GL_{n,E}}$. Restricting this last parameter to 
$W_{E_{\lambda'}}$, where $W_{E_{\lambda'}} \fleche W_{\QQ_\ell}$ is induced by some $i : E_{\lambda'} \fleche \overline\QQ_\ell$, we get $W_{E_{\lambda'}} \fleche {^L}T_{\GL_{n,E}} \rtimes W_{E_{\lambda'}}$. Fix an isomorphism $\phi : \overline\QQ_\ell \fleche \CC$, so we can identify complex and $\ell$-adic embeddings of $E$. But the action of $W_{E_{\lambda'}}$ fixes the $\sigma \in \Sigma_E$ over $\lambda'$, and we can thus project to any such using $\pr_\sigma : {^L}T_{GL_{n,E}} \fleche \GL_n$, so choose $\sigma_{\lambda'} = i \circ (E \fleche E_{\lambda'})$, the one corresponding to our embedding $W_{E_{\lambda'}} \fleche W_{\QQ_\ell}$, and denote
\[ r_{\lambda'} : W_{E_{\lambda'}} \fleche {^L}T_{\GL_{n,E}} \fleche \mathbb G_m^n, w \mapsto r_{\GL_{n,E}}(w) = (h_\sigma)_\sigma \rtimes \pi \mapsto h_{\sigma_{\lambda'}}.\]
Let us show that this is well defined and independant of choices of $i$ and $\phi$. Let $i,j : E_{\lambda'} \fleche \overline{\QQ_\ell}$ two choices. There exists $s \in W_{\overline\QQ_\ell}$ such that $s \circ j = i$. These two maps induce two maps $W_{E_\lambda'} \overset{i_*,j_*}{\fleche} W_{\overline\QQ_\ell}$, such that $j_*(-) = s^{-1}  i_*(-) s$.

Moreover using the canonical map $E \fleche E_{\lambda'}$ this induces two embeddings $\sigma_{\lambda'}^i, \sigma_{\lambda'}^j :E \fleche \overline{\QQ_\ell}$ above ${\lambda'}$ such that $\sigma_{\lambda'}^i = s \circ \sigma_{\lambda'}^j$. So we compute,
\[ r_{\GL_{n,E}}(j_* w) = r_{\GL_{n,E}}(s^{-1}i_* w s) = x \rtimes s^{-1} (h_\sigma)_\sigma \rtimes w (x \rtimes s^{-1})^{-1} = (x (h_{s\sigma})_\sigma s^{-1}w sx^{-1}) \rtimes w,\]
which is mapped under projection to the embedding $\sigma_{\lambda'}^j$ to
\[ x_{\sigma^j_{\lambda'}} h_{s\sigma^j_{\lambda'}} x^{-1}_{s^{-1}w^{-1}s\sigma^j_{\lambda'}},\]
but this is commutative, and $w^{-1} s\sigma^j_{\lambda'} = s\sigma_{\lambda'}^j$ as $w \in W_{E_{\lambda'}}$ thus we get
\[ x_{\sigma^j_{\lambda'}} x^{-1}_{\sigma^j_{\lambda'}} h_{\sigma^i_{\lambda'}},\] i.e.~ $\pi_{\sigma^i_{\lambda'}} \circ r_{\GL_{n,E}} \circ i_* = \pi_{\sigma^j_{\lambda'}} \circ r_{\GL_{n,E}} \circ j_*$ is well defined and independant of the choice of $i$. Now assume that $\phi,\phi'$ are different isomorphisms $\overline\QQ_\ell \fleche \CC$. So for each $i : E_{\lambda'} \fleche \overline \QQ_\ell$ we get two embeddings of $E$, namely $\sigma_\phi^i$ and $\sigma_{\phi'}^i = s \circ \sigma_\phi^i$, with $s = \phi' \circ \phi \in G_{\overline\QQ_\ell}$. Thus we are reduced to the previous computation with two different embedding above $\lambda'$. Thus $r_{\lambda'} := \pi_{\sigma^i_{\lambda'}} \circ r_{\GL_{n,E}} \circ i_*$ depends only on the choice of $\lambda' | \ell$ in $E$. 
\end{proof}

Using the previous Proposition, to $\pi$ we can for all unramified $\ell$ associate to $\pi_\ell$ a semi-simple conjugacy class in ${^L}T_U$ and for all $\lambda' | \ell$ in $E$ a 
system of semi-simple conjugacy classes $C_{\lambda'} = r_{\lambda'}(\Frob_{\lambda'})$ in $\GL_n$. We denote 
$Sat(\pi_\ell) = (Sat_\lambda(\pi_\ell))_{\lambda'} =:  (C_{\lambda'}|det|^{\frac{1-n}{2}})_{\lambda'}$.

\begin{defin}
\label{Satake}
Fix an isomorphism $\iota : \CC \simeq \overline{\QQ_p}$. Let $\rho : G_E \fleche \GL_n(\overline{\QQ_p})$. We say that $\rho$ is strongly (resp. weakly) essentially associated to $\pi$ if for 
 all $\ell\neq p$ (resp. for almost all $\ell\neq p$), unramified in $E$ and for $\pi$, for all $\lambda' | \ell$, $\rho$ is unramified at $\lambda'$ and the semi simple class of 
$\rho(\Frob_{\lambda'})$ and $\iota Sat_{\lambda'}(\pi_\ell)$ coincides. We say that $\rho$ is \textit{modular} if there exists a cuspidal $\pi$ as before such that $\rho$ is strongly essentially associated to $\pi$.
\end{defin}

\begin{rema}
\begin{enumerate}
\item This is not the natural definition, it would be more adequate to say \emph{essentially modular}. The reason is that because we want to work at fixed polarisation character, we have ignored the part of the similitude character for $\pi$ when looking at $Sat(\pi_\ell)$. We could do an analogous definition keeping track of the similitude character, but it would be more complicated to describe it, in particular at non split primes when $E/\QQ$ is not Galois.
\item It is enough to check the compatibility with the Satake parameter at $\ell$ totally split in $E$, in which case the previous association is easier to describe. Indeed, by Chebotarev density theorem the totally split primes in $E$ have density 1, thus $\rho$ is completely determined by the conjugacy class of Frobenius at those primes. Moreover, every $\lambda = \lambda' \lambda^{'c} | \ell$ is split above $F$ (with $\lambda$ is a prime of $F$). Thus $GU_{\overline\QQ_\ell} \simeq (\prod_{\lambda | \ell \text{ in } F} \GL_n) \times \mathbb G_m$\footnote{We should choose a CM type to write this isomorphism properly.} and the Satake parameter (for $GU$) associated to $\pi_\ell$ has the form
\[ (\Diag(t_1^{\lambda},\dots,t_n^\lambda)_\lambda,x).\]
Then $Sat(\pi_\ell)$ is just the collection \[((|det|^{\frac{n-1}{2}}\Diag(t_1^\lambda,\dots,t_n^\lambda))_{\lambda'},(|det|^{\frac{1-n}{2}}\Diag(t_1^{\lambda,-1},\dots,t_n^{\lambda,-1}))_{\lambda^{'c}}).\]
\item A modular $\rho$ is automatically polarized-by-1 (i.e. $\rho^\vee \simeq \rho^c \otimes \eps^{n-1}$). Indeed, elements $t \in \widehat{T_U}$ satisfies $t^{-1} = w_0 \cdot t^c$, where $w_0$ is the longuest Weyl element of $\GL_n$, thus (because of the twist) $\iota Sat_{\lambda'}(\pi_\ell)^{-1} = \iota Sat_{c\lambda'}(\pi_\ell)p^{n-1}$. By Chebotarev, this proves the claim.
\end{enumerate}
\end{rema}

\begin{defin}
\label{def:C2}
We say that a cuspidal automorphic representation $\pi$ of $G$ or $U(V)$ is sufficiently regular if it is a discrete series at infinity and satisfy property $(\star)$ of \cite{Lab} Section 5.1. This is automatic if the parameter at infinity is regular enough (by \cite[Lemma 3.6.1]{HarCor} and Mirkovic's Theorem \cite[Theorem 3.5]{HarCor}).
\end{defin}

\begin{rema}
Because of \cite[Lemma 3.6.1]{HarCor}  and Mirkovic, there exists $C > 0$ such that points of $\mathcal Z_C$ (see Proposition \ref{prop:geoHecke}) are sufficiently regular in the previous sense.
\end{rema}

    \begin{theor}
\label{thrA1}
Let $\pi$ be a cuspidal automorphic representation of $G = GU(V)$ which is cohomological and sufficiently regular.
There exists $L$ a Levi subgroup of $\Res_{E/\QQ}{G_E}$, a cuspidal automorphic representation $\Pi_L$ of $L(\mathbb A)$ together with an automorphic character $\chi_L$ of $L(\mathbb A)$ such that $\Pi_L\otimes\chi_L^{-1}$ is $\theta_L$-stable and $\pi$ and $\Pi_L$ corresponds to each other at all unramified (for $\pi$ and $E$) finite places.
Moreover each factor of $\Pi_L = \Pi_1 \boxplus \Pi_2 \boxplus\dots\boxplus \Pi_r$ is regular algebraic.
\end{theor}

\begin{proof}
If $F = \QQ$ this is \cite{Morel} Corollary 8.5.3, except the last part. But by Shin's appendix \cite{Gold}\footnote{This more generally applies if $E$ contains an imaginary quadratic field} Theorem 1.1 (iii), $\Pi_1\boxplus\Pi_2\boxplus\dots\boxplus\Pi_r$ is moreover regular algebraic as $\pi$ is. Remark that in this case we don't need $\pi$ to be sufficiently regular, just cohomological. If $[F:\QQ] \geq 2$, then we will use \cite{Lab}, thus we need the following lemma. Let us introduce some notation.
Let $Z = \{ x \in E^\times | N_{E/F}(x) \in \QQ^\times\}$ and $Z^1 :=
\Ker (N_{E/F})_{Z} \subset Z$. Then $Z,Z^1$ are tori. Moreover we have a maps
\begin{equation}
\label{ex1} 0 \fleche Z^1 \fleche Z \times U \fleche G \fleche 0,\end{equation}
and the last map is surjective on geometric points. Note that if
$\ell$ is a prime of $\QQ$, splitting in $E$, then the sequence (\ref{ex1}) is exact on $\QQ_\ell$-points.

\begin{lemma}\label{lem:GUtoU}
Let $\pi$ be an irreducible discrete automorphic representation of $G$
such that $\pi_\infty$ is cohomological for $\xi$. Then there exists
an automorphic discrete representations
$\psi \otimes \pi_0$ of $Z(\mathbb A)\times U(\mathbb A)$ such that
\begin{enumerate}
\item\label{GUtoU1} the restriction of $\psi\otimes\pi_0$ to the image of $Z^1(\mathbb{A})$ is trivial;
\item\label{GUtoU2} {$\psi = \psi_{\pi}$}, the restriction of $\pi$ to $Z$,
\item\label{GUtoU3} For all place $\ell$ of $\QQ$, splitting in $E$, we have
  $(\pi_\ell)|_{Z(\QQ_\ell)\times U(\QQ_\ell)} \simeq \psi_\ell \otimes \pi_{0,\ell}$;
\item\label{GUtoU4} $\pi_0$ is cohomological for $\xi|_U$, thus regular;
\item\label{GUtoU5} $\psi_\infty = \xi_{E_\infty^\times}^{-1}$.
\item\label{GUtoU6} If $\ell$ is a prime which is unramified in $E$,
  then $\pi_0$ is unramified if $\pi$ is.
\end{enumerate}
\end{lemma}
\begin{proof}
  This is analogous to the proof of \cite{HT} Theorem VI.2.1.
Choosing $(g_i)$ in $G(\mathbb{A})$ such that the $\nu(g_i)$ are representatives of the set $\nu(G(\mathbb A))/(\nu(G(\QQ))N_{E/F}(Z(\mathbb A))$ we get as in \cite{HT},
 \[
\begin{array}{ccc}
\mathcal A(G(\QQ)\backslash G(\mathbb A))  & \fleche  & \bigoplus_i \mathcal A((Z\times U)(\QQ)\backslash (Z\times U)(\mathbb A))^{Z^1(\mathbb A)}  \\
 f &  \longmapsto  & ((g_i \cdot f)|_{(Z\times U)(\mathbb{A})})_i  
\end{array}
\]
where the $g_i\cdot f$ is denotes the right translate of $f$. As a consequence we have an isomorphism of $(Z\times U)(\mathbb{A})$-representations
 \[
   \mathcal A(G(\QQ)\backslash G(\mathbb A))|_{(Z\times
     U)(\mathbb{A})} \simeq\bigoplus_i (\mathcal A((Z\times
   U)(\QQ)\backslash (Z\times U)(\mathbb A))^{Z^1(\mathbb
     A)})^{g_i}\]
where the upper script $g_i$ denotes a conjugate action by $g_i$.
This shows that, if $\pi$ is an automorphic representation of $G(\mathbb A)$ and if $\pi'$ is an irreducible subquotient of $\pi|_{(Z\times U)(\mathbb{A})}$, a conjugate of $\pi'$ by one 
of the $g_i$ is automorphic and trivial on $Z^1(\mathbb{A})$. Let
$\psi\otimes\pi_0$ be an automorphic representation ot $(Z\times
U)(\mathbb{A})$ whose conjugate by one of the $g_i$ is isomorphic to a
subquotient of $\pi|_{(Z\times U)(\mathbb{A})}$. Moreover, since $\pi$
is cohomological for $\xi$, there exists an integer $i$ such that
$H^i((\Lie
G(\RR))\otimes_{\RR}\CC,U_\infty,\pi_\infty\otimes\xi')\neq0$ (for
$U_\infty\subset U(\RR)$ a maximal compact subgroup of $G(\RR)$). So
we can choose $\psi\otimes\pi_0$ such that $H^i((\Lie (Z\times
U)(\RR))\otimes_{\RR}\CC,U_\infty,\psi_\infty\otimes\pi_{0,\infty}\otimes\xi'|_{(Z\times
  U)(\RR)})\neq0$. This proves that $\pi_0$ satisfies property
\ref{GUtoU4} of the statement. The property \ref{GUtoU1} has already
been checked and \ref{GUtoU2} is clear since $Z(\mathbb{A})$ is in the
center of $G(\mathbb{A})$. Property \ref{GUtoU3} is a direct
consequence of the fact that if $\ell$ is a prime that splits in $E$,
the map $(Z\times U)(\QQ_\ell)\rightarrow G(\QQ_\ell)$ is surjective
of kernel $Z^1(\QQ_\ell)$. Now assume that $\ell$ is unramified in
$E$. If $\pi$ is unramified at $\ell$, then $\pi$ has non zero fixed
vector under an hyperspecial subgroup of $G(\QQ_\ell)$. As the image
of $Z(\QQ_\ell)\times U(\QQ_\ell)$ has a finite index in
$G(\QQ_\ell)$, the restriction of $\pi_\ell$ to $U(\QQ_\ell)$ is
isomorphic to a finite direct sum of irreducible representation of
$U(\QQ_\ell)$ which are conjugated in $G(\QQ_\ell)$. As the
intersection of an hyperspecial subgroup of $G(\QQ_\ell)$ with $U(\QQ_\ell)$ is an
hyperspecial subgroup of $U(\QQ_\ell)$, all irreducible subquotients
of $\pi_{\ell}|_{U(\QQ_\ell)}$ have nonzero fixed vectors under some
hyperspecial subgroup of $U(\QQ_\ell)$. This proves property \ref{GUtoU6}.
Property \ref{GUtoU5} is a direct consequence of the equality $\pi|_{E_\infty^\times}=\xi'_{E_\infty^\times}$ following from the fact that $\pi$ is cohomological for $\xi'$.
\end{proof}
Thus by \cite{Lab} Cor. 5.3 applied to $\pi_0$, which is sufficiently regular thus satisfies property $(\star)$, there is a weak base change i.e.~ $L$ a standard Levi of $\Res_{E/\QQ}\GL_{n,E}$ that is $\theta$-stable, and a $\theta_L$-stable discrete 
automorphic representation of $L$ $\Pi_L = \Pi'_1\otimes\dots\otimes\Pi'_s$ such that $ \Pi'_1 \boxplus \Pi'_2 \boxplus \dots \boxplus \Pi'_s$ is a weak base change for $\pi_0$. As each $\Pi'_i$ is discrete, 
then by the main theorem of \cite{MoegWald} we can write $\Pi'_i$ as an automorphic induction of $\tau_i \otimes Sp(\ell_i)$ for an integer $\ell_i$ and $\tau_i$ a cuspidal automorphic representation of 
$\GL_{n_i/\ell_i}(\mathbb A_E)$. But the proof of \cite{Morel} 8.5.6 shows that as each $\Pi_i'$ is 
$\theta_{n_i}$-stable, each $\tau_i$ is $\theta_{n_i/\ell_i}$-stable.
In particular, up to reduce $L$, choosing $(\Pi_j)_{j=1,\dots,r}$ to be the collection $(\tau_i |\det|^{(\ell_i-2k+1)/2})$ for $i = 1,\dots,s$ and $k = 1,\dots,\ell_i$, we get that $\Pi_j$ is 
cuspidal, and $\Pi_j$ is $\theta$-stable up to twist (by an
automorphic character of $L$). But by \cite[Cor.~5.3]{Lab} again we know that the infinitesimal characters of $\pi_0$ and 
$\Pi'_1\boxplus \dots \boxplus\Pi'_s = \Pi_1\boxplus \dots\boxplus \Pi_r$ coincides after base change, in particular the latter representation is regular algebraic. Moreover at unramified places this is compatible with the local base change.
\end{proof}

\begin{cor}
\label{cor:A5}
Let $\pi$ be a cuspidal automorphic representation of $G = GU(V)$ which is cohomological, sufficiently regular, and unramified outside $S$, which contains ramified places of $E$. 
Then to $\pi$ is (strongly) essentially associated a unique Galois representation,
\[ \rho^u : G_{E,S} \fleche \GL_n(\overline{\QQ_p}), \]
satisfying,
\[ (\rho^u)^\vee \simeq (\rho^u)^c \eps^{n-1}.\] 
In particular, for all prime $\lambda = v\overline v$ of $F$ split in $E$, not in $S$, we have that the semi-simple conjugacy class of $\rho^u(\Frob_v)$ is equal to the image of the Satake parameter of $\pi_{0,\lambda}|\det|^{\frac{1-n}{2}}$, seen as a representation of $U(F_\lambda) \overset{\iota_v}{\simeq} GL_{n,E_v}$.
\end{cor}

\begin{proof} The previous proof allow us to reduce to $\pi_0$ an automorphic representation of $U$ whose weak base change is $\Pi_1 \boxplus \dots \boxplus \Pi_r$; $\theta$-stable, each $\Pi_i$ being automorphic for $GL_{n_i}$, cuspidal, conjugate self dual up to twist by a character. Thus to $\pi_0$ we can associate by \cite{CH} again
\[ \rho^u = \rho_{\Pi_1} \oplus \dots \oplus\rho_{\Pi_r}.\]
As $\Pi_1 \boxplus \dots \boxplus \Pi_r$ is $\theta$-stable, $\rho^u$ satisfies $(\rho^u)^\vee \simeq (\rho^u)^c \otimes \eps^{n-1}$. 
On the other hand, we know the compatibility of the association of $\rho^u$ with local Langlands : at all ramified primes $\rho^u_v = LL(\pi_{0,v}|.|^{\frac{1-n}{2}})$, i.e.~ $\rho^u$ is strongly associated to $\pi$.
\end{proof}

\bibliographystyle{smfalpha}
\bibliography{biblio}
\end{document}